\documentclass[11pt]{article}

\usepackage{latexsym}
\usepackage{amssymb}
\usepackage{amsthm}
\usepackage{amscd}
\usepackage{amsmath}
\usepackage{mathrsfs}
\usepackage{graphicx}
\usepackage{hyperref}
\usepackage{shuffle}

\usepackage[all]{xy}
\input xy \xyoption{frame}
\xyoption{dvips}

\usepackage[colorinlistoftodos]{todonotes}
\usepackage{fullpage}
\usetikzlibrary{chains,scopes,decorations.markings}

\theoremstyle{definition}
\newtheorem* {theorem*}{Theorem}
\newtheorem* {conjecture*}{Conjecture}
\newtheorem{theorem}{Theorem}[section]
\newtheorem{thmdef}[theorem]{Theorem-Definition}
\newtheorem{propdef}[theorem]{Proposition-Definition}

\theoremstyle{definition}

\newtheorem* {example*}{Example}

\newtheorem{lemma}[theorem]{Lemma}
\theoremstyle{definition}
\newtheorem{definition}[theorem]{Definition}
\theoremstyle{definition}

\newtheorem{conjecture}[theorem]{Conjecture}
\newtheorem{proposition}[theorem]{Proposition}
\newtheorem{corollary}[theorem]{Corollary}

\newtheorem *{remark}{Remark}
\theoremstyle{definition}
\newtheorem {example}[theorem]{Example}
\theoremstyle{definition}

\theoremstyle{definition}

\theoremstyle{definition}

\xyoption{dvips}

\def\rW{]}
\def\lW{[}

\def\modu{\ (\mathrm{mod}\ }

\def\({\left(}
\def\){\right)}

\newcommand{\cP}{\mathcal{P}}

\newcommand{\cC}{\mathcal{C}}
\newcommand{\cD}{\mathcal{D}}

\def\NN{\mathbb{N}}

\def\ZZ{\mathbb{Z}}

\newcommand\wt{\operatorname{wt}}

\newcommand{\cM}{\mathcal{M}}

\def\barr{\begin{array}}
\def\earr{\end{array}}
\def\ba{\begin{aligned}}
\def\ea{\end{aligned}}
\def\be{\begin{equation}}
\def\ee{\end{equation}}

\def\qquand{\qquad\text{and}\qquad}
\def\quand{\quad\text{and}\quad}
\def\qquord{\qquad\text{or}\qquad}
\def\quord{\quad\text{or}\quad}

\def\inv{\mathrm{Inv}}

\def\cH{\mathcal H}
\def\cM{\mathcal M}
\def\DesR{\mathrm{Des}_R}
\def\DesL{\mathrm{Des}_L}

\def\hs{\hspace{0.5mm}}

\def\ds{\displaystyle}

\def\dash{\hs\text{---}\hs}

\def\tphi{\tilde \phi}
\def\tpsi{\tilde \psi}

\def\ben{\begin{enumerate}}
\def\een{\end{enumerate}}

\def\uL{ \underline{L}}
\def\uR{R}

\def\hs{\hspace{0.5mm}}

\def\ellhat{\hat\ell}

\def\x{\textbf{x}}

\newcommand{\xRightarrow}[2][]{\ext@arrow 0359\Rightarrowfill@{#1}{#2}}

\newcommand{\rank}{\operatorname{rank}}

\newcommand{\cA}{\mathcal{A}}

\newcommand{\arc}[2]{ \ar @/^#1pc/ @{-} [#2] }
\def\arcstop{\endxy\ }
\def\arcstart{\ \xy<0cm,-.06cm>\xymatrix@R=.1cm@C=.2cm }
\newcommand{\arcstartc}[1]{\ \xy<0cm,-.15cm>\xymatrix@R=.1cm@C=#1cm}

\def\ellhat{\hat\ell}

\def\Luc{\mathrm{Luc}}

\newcommand{\ct}{\operatorname{ct}}

\def\ttS{{\tt S}}

\def\tS{\tilde S}

\def\tI{\tilde I}
\def\cG{\mathcal{G}}

\def\cW{\tI^{\wt}}
\def\cWmin{\cM}

\def\diagramAB{
\arcstart
{
*{\circ}     & *{\circ}
}
\arcstop}

\def\diagramBA{
\arcstart
{
*{\circ}   \arc{.5}{r}  & *{\circ}
}

\arcstop
}
\def\diagrambAC{
\arcstart
{
*{\bullet}    \arc{.5}{r}  & *{\circ} & *{\circ}
}
\arcstop
}
\def\diagramBaC{
\arcstart
{
*{\circ}    \arc{.5}{r}  & *{\bullet} & *{\circ}
}
\arcstop
}
\def\diagramcBA{
\arcstart
{
*{\bullet}    \arc{.5}{rr}  & *{\circ} & *{\circ}
}
\arcstop
}

\def\diagramACb{
\arcstart
{
*{\circ}     & *{\circ}  \arc{.5}{r} & *{\bullet}
}
\arcstop
}
\def\diagramAcB{
\arcstart
{
*{\circ}     & *{\bullet}  \arc{.5}{r} & *{\circ}
}
\arcstop
}
\def\diagramCbA{
\arcstart
{
*{\circ}     \arc{.5}{rr}   & *{\bullet} & *{\circ}
}
\arcstop
}
\def\diagramCBa{
\arcstart
{
*{\circ}     \arc{.5}{rr}   & *{\circ} & *{\bullet}
}
\arcstop
}

\def\diagramBaDc{
\arcstart
{
*{\circ}  \arc{.5}{r}   & *{\bullet}    & *{\circ} \arc{.5}{r}& *{\bullet}
}
\arcstop
}
\def\diagramBadC{
\arcstart
{
*{\circ}  \arc{.5}{r}   & *{\bullet}    & *{\bullet} \arc{.5}{r}& *{\circ}
}
\arcstop
}
\def\diagrambAdC{
\arcstart
{
*{\bullet}  \arc{.5}{r}   & *{\circ}    & *{\bullet} \arc{.5}{r}& *{\circ}
}
\arcstop
}
\def\diagrambADc{
\arcstart
{
*{\bullet}  \arc{.5}{r}   & *{\circ}    & *{\circ} \arc{.5}{r}& *{\bullet}
}
\arcstop
}

\def\diagramCDab{
\arcstart
{
*{\circ}  \arc{.5}{rr}   & *{\circ} \arc{.5}{rr}   & *{\bullet} & *{\bullet}
}
\arcstop
}

\def\diagramCdaB{
\arcstart
{
*{\circ}  \arc{.5}{rr}   & *{\bullet} \arc{.5}{rr}   & *{\bullet} & *{\circ}
}
\arcstop
}
\def\diagramcDAb{
\arcstart
{
*{\bullet}  \arc{.5}{rr}   & *{\circ} \arc{.5}{rr}   & *{\circ} & *{\bullet}
}
\arcstop
}

\def\diagramcdAB{
\arcstart
{
*{\bullet}  \arc{.5}{rr}   & *{\bullet} \arc{.5}{rr}   & *{\circ} & *{\circ}
}
\arcstop
}

\def\diagramDbcA{
\arcstart
{
*{\circ}  \arc{.5}{rrr}   & *{\bullet}   & *{\bullet} & *{\circ}
}
\arcstop
}
\def\diagramDbCa{
\arcstart
{
*{\circ}  \arc{.5}{rrr}   & *{\bullet}   & *{\circ} & *{\bullet}
}
\arcstop
}
\def\diagramdBcA{
\arcstart
{
*{\bullet}  \arc{.5}{rrr}   & *{\circ}   & *{\bullet} & *{\circ}
}
\arcstop
}

\def\diagramDcBa{
\arcstart
{
*{\circ}  \arc{.5}{rrr}   & *{\bullet} \arc{.25}{r}   & *{\circ} & *{\bullet}
}
\arcstop
}

\def\diagramDCba{
\arcstart
{
*{\circ}  \arc{.5}{rrr}   & *{\circ} \arc{.25}{r}   & *{\bullet} & *{\bullet}
}
\arcstop
}
\def\diagramDcbA{
\arcstart
{
*{\circ}  \arc{.5}{rrr}   & *{\bullet} \arc{.25}{r}   & *{\bullet} & *{\circ}
}
\arcstop
}
\def\diagramdcBA{
\arcstart
{
*{\bullet}  \arc{.5}{rrr}   & *{\bullet} \arc{.25}{r}   & *{\circ} & *{\circ}
}
\arcstop
}

\numberwithin{equation}{section}
\allowdisplaybreaks[1]
\UseCrayolaColors

\makeatletter
\renewcommand{\@makefnmark}{\mbox{\textsuperscript{}}}
\makeatother

\begin{document}
\title{On some actions of the $0$-Hecke monoids of affine symmetric groups}
\author{
Eric Marberg \\
    Department of Mathematics \\
    Hong Kong University of Science and Technology \\
    {\tt eric.marberg@gmail.com}
}

\date{}

\maketitle

\begin{abstract}
There are left and right actions of the $0$-Hecke monoid of the affine symmetric group $\tilde S_n$
on involutions whose cycles are labeled periodically by nonnegative integers. 
Using these actions we construct two bijections, which are length-preserving in an appropriate sense,
 from the set of involutions in $\tilde S_n$ to the set of $\mathbb{N}$-weighted
matchings in the $n$-element cycle graph. As an application, we compute a formula for the bivariate generating 
function counting the involutions in $\tilde S_n$ by length and absolute length.
The $0$-Hecke monoid of $\tilde S_n$ also acts on involutions (without any cycle labelling) 
by Demazure conjugation.
The atoms of an involution $z \in \tilde S_n$ are the minimal length permutations $w$ 
which transform the identity to $z$ under this action. 
We prove that the set of atoms for an involution in $\tilde S_n$ is naturally a bounded, 
graded poset, and give a formula for the set's minimum and maximum elements.
Using these properties, we classify the covering relations in the Bruhat order  
restricted to involutions in $\tilde S_n$. 
\end{abstract}

\setcounter{tocdepth}{2}

\section{Introduction}\label{intro-sect}

For each integer $n\geq 1$, let  $\tS_n$ be the \emph{affine symmetric group} of rank $n$, consisting of the bijections $w : \ZZ \to \ZZ$
with
$w(i+n) = w(i) + n$ for all $i \in \ZZ$ and $w(1) + w(2) + \dots + w(n) = \binom{n+1}{2}$.
When $n=1$, these conditions imply that $\tS_1 = \{1\}$.
Assume $n \geq 2$, and
 define $s_i \in \tS_n$ for $i \in \ZZ$
as the permutation which exchanges $i +mn$ and $i+1 +mn$ for each $m \in \ZZ$, while fixing every integer
not congruent to $i$ or $i+1$ modulo $n$. 
The elements $s_1,s_2,\dots,s_n$ generate $\tS_n$, and with respect to these generators $\tS_n$ is the Coxeter group of type $\tilde A_{n-1}$ \cite[\S8.3]{BB}.

If $W$ is any Coxeter group with simple generating set $S$ and length function $\ell :W \to \NN$,
then there is a unique associative product $\circ : W \times  W \to W$ 
such that $w\circ s = w$ if $\ell(ws) < \ell(w)$ and $w\circ s = ws$ if $\ell(ws) > \ell(w)$
for
$w \in W$ and $s \in S$ \cite[Theorem 7.1]{Humphreys}.
The product $\circ$ is often called the \emph{Demazure product},
and the pair $(W,\circ)$ is usually referred to as the \emph{$0$-Hecke monoid} or \emph{Richardson-Springer monoid} of $(W,S)$.
We frame the
results of this paper around the discussion of three actions of the $0$-Hecke monoid of $\tS_n$.
Each action will be on objects related to the group's involutions, 
that is, the elements $z \in \tS_n$ with $z^2=1$.

Let $I_n $ be the set of involutions in the finite symmetric group $S_n$,
which we identify with the parabolic subgroup of $\tS_n$ generated by $s_1,s_2,\dots,s_{n-1}$.
A \emph{matching} in a graph is a subset of edges with no shared vertices;
with slight abuse of notation, a \emph{matching} on a set is a matching in the complete graph on that set.
Elements of $I_n$ are permutations whose cycles have length at most two, and so may be
viewed as matchings on $\{1,2,\dots,n\}$. For example, 
\be\label{match-eq}\barr{c} \\[-8pt]  (1,4)(2,7)(3,6) \in I_8\qquad\text{corresponds to}\qquad
\arcstart
{
*{\bullet} \arc{0.6}{rrr}     & *{\bullet}    \arc{1.5}{rrrrr} & *{\bullet}   \arc{0.6}{rrr}  & *{\bullet}     & *{\bullet}  & *{\bullet}    & *{\bullet}  & *{\bullet} \\
1 & 2 & 3 & 4 & 5 & 6 & 7 & 8
} 
\arcstop.
\earr
\ee
There are several ways to adapt this combinatorial model
 to the elements of   $\tI_n = \{ z\in \tS_n : z^2=1\}$.
The simplest method is to represent $z \in \tI_n$
 as the matching on $\ZZ$ in which $i$ and $j$ are connected by an edge whenever $z(i) = j \neq i = z(j)$.
 This gives a bijection between $\tI_n$ and matchings on $\ZZ$ which are ``$n$-periodic''
 in the sense of having $\{ i,j\}$ as an edge if and only if $\{i+n,j+n\}$ is also an edge.
 One can make this model more compact by converting $n$-periodic matchings on $\ZZ$
  to $\ZZ$-weighted matchings on $\{1,2,\dots,n\}$: to represent $z \in \tI_n$, include the edge $\{i,j\}$ labeled by $m \in \ZZ$
whenever $i<j$ and  $z(i) = j + mn$ and $z(j) = i - mn$. For example,
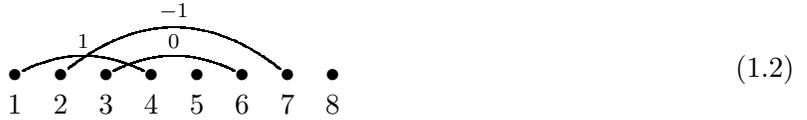
\be\label{z-picture}
\arcstart
{
*{\bullet} \arc{0.6}{rrr}^1     & *{\bullet}    \arc{1.5}{rrrrr}^{-1} & *{\bullet}   \arc{0.6}{rrr}^0  & *{\bullet}     & *{\bullet}  & *{\bullet}    & *{\bullet}  & *{\bullet} \\
1 & 2 & 3 & 4 & 5 & 6 & 7 & 8
} 
\arcstop
\ee
 would correspond to 
 $z= \prod_{m \in \ZZ} (1+mn,12+mn)(7+mn,10+mn)(3+mn,6+mn) \in \tI_8$.
Diagrams of this type are most useful when $\tS_n$ is viewed as a semidirect product $S_n \ltimes \ZZ^{n-1}$. 
When the structure of $\tS_n$ as a Coxeter group is significant,
a better approach is to view $n$-periodic matchings as \emph{winding diagrams}.
To construct the \emph{winding diagram} of $z \in \tI_n$, arrange $1,2,\dots,n$
clockwise on a circle, and
whenever $i < z(i) \equiv j \modu n)$, connect $i$ to $j$ by an arc winding $\frac{z(i)-i}{n}$ times in the clockwise direction around the circle's exterior.  
For the involution in \eqref{z-picture}, this produces the picture
  \be
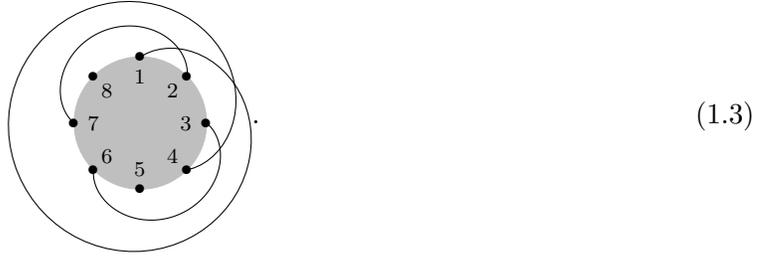
\label{winding-picture}
\begin{tikzpicture}[baseline=0,scale=0.22,label/.style={postaction={ decorate,transform shape,decoration={ markings, mark=at position .5 with \node #1;}}}]
{
\draw[fill,lightgray] (0,0) circle (4.0);
\node at (2.44929359829e-16, 4.0) {$_\bullet$};
\node at (1.71450551881e-16, 2.8) {$_1$};
\node at (2.82842712475, 2.82842712475) {$_\bullet$};
\node at (1.97989898732, 1.97989898732) {$_2$};
\node at (4.0, 0.0) {$_\bullet$};
\node at (2.8, 0.0) {$_3$};
\node at (2.82842712475, -2.82842712475) {$_\bullet$};
\node at (1.97989898732, -1.97989898732) {$_4$};
\node at (2.44929359829e-16, -4.0) {$_\bullet$};
\node at (1.71450551881e-16, -2.8) {$_5$};
\node at (-2.82842712475, -2.82842712475) {$_\bullet$};
\node at (-1.97989898732, -1.97989898732) {$_6$};
\node at (-4.0, -4.89858719659e-16) {$_\bullet$};
\node at (-2.8, -3.42901103761e-16) {$_7$};
\node at (-2.82842712475, 2.82842712475) {$_\bullet$};
\node at (-1.97989898732, 1.97989898732) {$_8$};
\draw [-,>=latex,domain=0:100,samples=100] plot ({(4.0 + 4.0 * sin(180 * (0.5 + asin(-0.9 + 1.8 * (\x / 100)) / asin(0.9) / 2))) * cos(90 - (0.0 + \x * 4.95))}, {(4.0 + 4.0 * sin(180 * (0.5 + asin(-0.9 + 1.8 * (\x / 100)) / asin(0.9) / 2))) * sin(90 - (0.0 + \x * 4.95))});
\draw [-,>=latex,domain=0:100,samples=100] plot ({(4.0 + 2.0 * sin(180 * (0.5 + asin(-0.9 + 1.8 * (\x / 100)) / asin(0.9) / 2))) * cos(90 - (90.0 + \x * 1.35))}, {(4.0 + 2.0 * sin(180 * (0.5 + asin(-0.9 + 1.8 * (\x / 100)) / asin(0.9) / 2))) * sin(90 - (90.0 + \x * 1.35))});
\draw [-,>=latex,domain=0:100,samples=100] plot ({(4.0 + 2.0 * sin(180 * (0.5 + asin(-0.9 + 1.8 * (\x / 100)) / asin(0.9) / 2))) * cos(90 - (270.0 + \x * 1.35))}, {(4.0 + 2.0 * sin(180 * (0.5 + asin(-0.9 + 1.8 * (\x / 100)) / asin(0.9) / 2))) * sin(90 - (270.0 + \x * 1.35))});
}
\end{tikzpicture}.
\ee
 Formally, a winding diagram is a collection of continuous paths between disjoint pairs of marked points on the boundary of the plane minus an open disc, up to homotopy.
Each winding diagram corresponds to a unique involution in some affine symmetric group.
For our purposes, this construction is the correct generalisation of \eqref{match-eq} to the affine case.

Write $\ell(w)$ for the usual Coxeter length of $w \in \tS_n$,
and define the \emph{absolute length} $\ell'(z)$ of $z \in \tI_n$ to be the number of arcs in its winding diagram.
 Our first main result, Theorem~\ref{main-thm},   identifies
two  bijections $\omega_R$  and $\omega_L$ from $\tI_n$ to the set $\cWmin_n$ of $\NN$-weighted matchings 
 in $\cC_n$, the cycle graph on $n$ vertices.
 These bijections 
 preserve length and absolute length, where
 the absolute length of an $\NN$-weighted matching
 is its number of edges and the length is its number of edges plus twice the sum of their weights.
The images of the element $z \in \tI_8$ in our running example \eqref{winding-picture} are
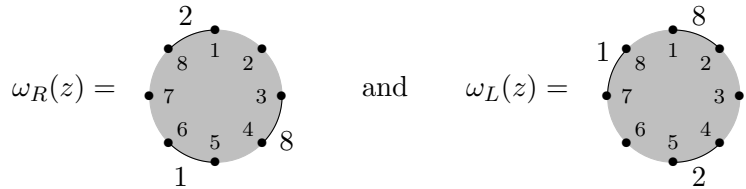
\[
\omega_R(z) =\
\begin{tikzpicture}[baseline=0,scale=0.22,label/.style={postaction={ decorate,decoration={ markings, mark=at position .5 with \node #1;}}}]
{
\draw[fill,lightgray] (0,0) circle (4.0);
\node at (2.44929359829e-16, 4.0) {$_\bullet$};
\node at (1.71450551881e-16, 2.8) {$_{1}$};
\node at (2.82842712475, 2.82842712475) {$_\bullet$};
\node at (1.97989898732, 1.97989898732) {$_{2}$};
\node at (4.0, 0.0) {$_\bullet$};
\node at (2.8, 0.0) {$_{3}$};
\node at (2.82842712475, -2.82842712475) {$_\bullet$};
\node at (1.97989898732, -1.97989898732) {$_{4}$};
\node at (2.44929359829e-16, -4.0) {$_\bullet$};
\node at (1.71450551881e-16, -2.8) {$_{5}$};
\node at (-2.82842712475, -2.82842712475) {$_\bullet$};
\node at (-1.97989898732, -1.97989898732) {$_{6}$};
\node at (-4.0, -4.89858719659e-16) {$_\bullet$};
\node at (-2.8, -3.42901103761e-16) {$_{7}$};
\node at (-2.82842712475, 2.82842712475) {$_\bullet$};
\node at (-1.97989898732, 1.97989898732) {$_{8}$};
\draw [-,>=latex,domain=0:100,samples=100,label={[below]{$\ \ 8$}}] plot ({(4.0 + 0.0 * sin(180 * (0.5 + asin(-0.9 + 1.8 * (\x / 100)) / asin(0.9) / 2))) * cos(90 - (90.0 + \x * 0.45))}, {(4.0 + 0.0 * sin(180 * (0.5 + asin(-0.9 + 1.8 * (\x / 100)) / asin(0.9) / 2))) * sin(90 - (90.0 + \x * 0.45))});
\draw [-,>=latex,domain=0:100,samples=100,label={[below]{$1\ \ $}}] plot ({(4.0 + 0.0 * sin(180 * (0.5 + asin(-0.9 + 1.8 * (\x / 100)) / asin(0.9) / 2))) * cos(90 - (180.0 + \x * 0.45))}, {(4.0 + 0.0 * sin(180 * (0.5 + asin(-0.9 + 1.8 * (\x / 100)) / asin(0.9) / 2))) * sin(90 - (180.0 + \x * 0.45))});
\draw [-,>=latex,domain=0:100,samples=100,label={[above]{$2\ $}}] plot ({(4.0 + 0.0 * sin(180 * (0.5 + asin(-0.9 + 1.8 * (\x / 100)) / asin(0.9) / 2))) * cos(90 - (315.0 + \x * 0.45))}, {(4.0 + 0.0 * sin(180 * (0.5 + asin(-0.9 + 1.8 * (\x / 100)) / asin(0.9) / 2))) * sin(90 - (315.0 + \x * 0.45))});
}
\end{tikzpicture} 
\qquand
\omega_L(z) =\ 
\begin{tikzpicture}[baseline=0,scale=0.22,label/.style={postaction={ decorate,decoration={ markings, mark=at position .5 with \node #1;}}}]
{
\draw[fill,lightgray] (0,0) circle (4.0);
\node at (2.44929359829e-16, 4.0) {$_\bullet$};
\node at (1.71450551881e-16, 2.8) {$_{1}$};
\node at (2.82842712475, 2.82842712475) {$_\bullet$};
\node at (1.97989898732, 1.97989898732) {$_{2}$};
\node at (4.0, 0.0) {$_\bullet$};
\node at (2.8, 0.0) {$_{3}$};
\node at (2.82842712475, -2.82842712475) {$_\bullet$};
\node at (1.97989898732, -1.97989898732) {$_{4}$};
\node at (2.44929359829e-16, -4.0) {$_\bullet$};
\node at (1.71450551881e-16, -2.8) {$_{5}$};
\node at (-2.82842712475, -2.82842712475) {$_\bullet$};
\node at (-1.97989898732, -1.97989898732) {$_{6}$};
\node at (-4.0, -4.89858719659e-16) {$_\bullet$};
\node at (-2.8, -3.42901103761e-16) {$_{7}$};
\node at (-2.82842712475, 2.82842712475) {$_\bullet$};
\node at (-1.97989898732, 1.97989898732) {$_{8}$};
\draw [-,>=latex,domain=0:100,samples=100,label={[above]{$8$}}] plot ({(4.0 + 0.0 * sin(180 * (0.5 + asin(-0.9 + 1.8 * (\x / 100)) / asin(0.9) / 2))) * cos(90 - (0.0 + \x * 0.45))}, {(4.0 + 0.0 * sin(180 * (0.5 + asin(-0.9 + 1.8 * (\x / 100)) / asin(0.9) / 2))) * sin(90 - (0.0 + \x * 0.45))});
\draw [-,>=latex,domain=0:100,samples=100,label={[below]{$2$}}] plot ({(4.0 + 0.0 * sin(180 * (0.5 + asin(-0.9 + 1.8 * (\x / 100)) / asin(0.9) / 2))) * cos(90 - (135.0 + \x * 0.45))}, {(4.0 + 0.0 * sin(180 * (0.5 + asin(-0.9 + 1.8 * (\x / 100)) / asin(0.9) / 2))) * sin(90 - (135.0 + \x * 0.45))});
\draw [-,>=latex,domain=0:100,samples=100,label={[above]{$1\ \ $}}] plot ({(4.0 + 0.0 * sin(180 * (0.5 + asin(-0.9 + 1.8 * (\x / 100)) / asin(0.9) / 2))) * cos(90 - (270.0 + \x * 0.45))}, {(4.0 + 0.0 * sin(180 * (0.5 + asin(-0.9 + 1.8 * (\x / 100)) / asin(0.9) / 2))) * sin(90 - (270.0 + \x * 0.45))});
}
\end{tikzpicture}
\]
and indeed it holds that $\ell'(z) = 3$ and $\ell(z ) = 25$.
The proof of Theorem~\ref{main-thm} relies on the construction
of a left and right action of the $0$-Hecke monoid of $\tS_n$ on the set of \emph{weighted involutions},
which may be defined informally as $n$-periodic, $\NN$-weighted matchings on $\ZZ$;
see Section~\ref{weighted-sect}.
Our results
 provide a fourth model for  $\tI_n$,   which makes it easy to count the elements 
 of $ \tI_n$ by length.
As an application, we show (see Corollary~\ref{A-cor}) that 
\be \sum_{z \in \tI_n} q^{\ell(z)} x^{\ell'(z)} =  \sum_{k=0}^{\lfloor n/2 \rfloor} \frac{n}{n-k} \binom{n-k}{k} \( \frac{qx}{1-q^2}\)^k.\ee
This is an analogue of a more complicated identity proved in \cite{MW}.

The $0$-Hecke monoid  of $\tS_n$ also acts directly on $\tI_n$ by \emph{Demazure conjugation}: the right action mapping
$(z,w) \mapsto w^{-1} \circ z \circ w$ for $z \in \tI_n$ and $w \in \tS_n$.
This monoid action is a degeneration of the Iwahori-Hecke algebra representation studied by Lusztig and Vogan in \cite{LV1,LV2}.
The orbit of the identity under Demazure conjugation is all of $\tI_n$, and we define $\cA(z)$ for $z \in \tI_n$ as the set of  elements $w \in \tS_n$ of minimal length such that $z = w^{-1} \circ w$. Following \cite{HMP1,HMP2}, we call these permutations the \emph{atoms} of $z$. There are a few reasons why these elements merit further study, beyond their interesting combinatorial properties. The sets $\cA(z)$ may be defined for involutions in any Coxeter group
and, in the case of finite Weyl groups,  are closely related to the sets
$W(Y)$ which Brion  \cite{Brion98} attaches to $B$-orbit closures $Y$ in a spherical homogeneous space $G/H$ (where $G$ is a connected complex reductive group, $B$ a Borel subgroup, and $H$ a spherical subgroup).
Results of Hultman \cite{H1,H2}, extending work of Richardson and Springer \cite{RichSpring,RichSpring2}, show the atoms to be intimately connected to the Bruhat order of a Coxeter group restricted to its involutions. 
Finally, the atoms of involutions in finite symmetric groups play a central role in recent work
of Can, Joyce, Wyser, and Yong on the geometry of the orbits of the orthogonal group on the type $A$ flag variety; see \cite{CJ,CJW,Wyser,WY}.

Our object in the second half of this paper is to generalise a number of results about the atoms of involutions
in finite symmetric groups to the affine case.
In Section~\ref{dem-sect}, extending results in \cite{HMP2,HuZhang1}, we show that there is a natural partial order which makes $\cA(z)$ for $z \in \tI_n$
into a bounded, graded poset. We conjecture that this poset is a lattice.
Generalising results of Can, Joyce, and Wyser \cite{CJ,CJW}, we show in Section~\ref{cycle-sect} that 
there is an explicit set of inequalities governing the ``one-line'' representation of a permutation in $\tS_n$
which determines whether it belongs to $\cA(z)$. 
This result translates to a ``local'' criterion for membership in $\cA(z)$ involving a notion of \emph{(affine) standardisation}; see Corollary~\ref{std-cor}.
In Section~\ref{ct-sect}, extending \cite{HMP3,Incitti1}, we describe all covering relations in the Bruhat order of $\tS_n$ restricted to $\tI_n$.
Using this information, we prove that involutions in $\tS_n$ have what we call the \emph{Bruhat covering property}:

\begin{theorem}[Bruhat covering property] \label{bruh-cov-thm}
If $y \in \tI_n$ and $t \in \tS_n$ is a reflection, then there exists at most
one $z \in \tI_n$ such that $\left\{wt : w \in \cA(y)\text{ and }\ell(wt)=\ell(w)+1\right\} \cap \cA(z) \neq\varnothing$.
\end{theorem}

The analogue of this result for involutions in $S_n$ was shown in \cite{HMP3}, 
and served as a key lemma  in proofs of ``transition formulas'' for certain \emph{involution Schubert polynomials}. We conjecture that the same property holds for arbitrary Coxeter systems, in the following sense.

Let $(W,S)$ be a  Coxeter system with length function $\ell : W \to \NN$
and Demazure product $\circ : W\times W \to W$.
Suppose $w\mapsto w^*$ is an automorphism of $W$ with $S^* = S$.
The corresponding set of \emph{twisted involutions} is  $I_* = \{w \in W : w^{-1}=  w^*\}$.
For $y \in I_*$ let $\cA_*(y)$ be the set of elements of minimal length with $(w^*)^{-1} \circ w = y$,
and write $T = \{ wsw^{-1} : w \in W,\ s \in S\}$.

\begin{conjecture}
If $y \in I_*$ is a twisted involution in an arbitrary Coxeter group and $t \in T$,
then there exists at most one $z \in I_*$ such that $
\left\{wt : w \in \cA_*(y) \text{ and } \ell(wt) = \ell(w)+1\right\} \cap \cA_*(z) \neq \varnothing$.
\end{conjecture}

This statement is analogous to \cite[Lemma 21]{LamShim}, but seems harder to prove.
A useful consequence of the new methods in this paper is that we are able to replace 
certain computer dependent proofs for type $A_n$ in \cite{HMP3} by
simpler and more general arguments for type $\tilde A_n$.
Concerning future work, we anticipate that our results will be useful in developing 
a theory of affine involution Stanley symmetric functions, simultaneously generalising \cite{Lam,LamShim} and \cite{HMP1,HMP3,HMP4,HMP5}.

\subsection*{Acknowledgements}

This article grew out of helpful conversations with Brendan Pawlowski and Graham White.

\section*{Notation}

A comprehensive index of symbols is provided in Section~\ref{not-sect}.
Throughout,  $\ZZ$ denotes the set of integers, $\NN$ the set of nonnegative integers,
and $[n]$ the set $\{1,2,\dots,n\}$.
We write $\tS_n$ for the rank $n$ affine symmetric group
and $\tI_n$ for its subset of involutions. Having fixed $n\geq 2$,
we define $s_i \in \tI_n$ for $i \in \ZZ$
as the permutation interchanging $i+mn$ and $i+1+mn$ for   $m \in \ZZ$ while fixing all numbers outside of $\{i,i+1\}+n\ZZ$.
Let $\ell $ denote the length function of $\tS_n$ relative to the generating set $\{ s_i : i \in [n]\}$,
and write $<$ for the (strong) Bruhat order on $\tS_n$.

\section{Preliminaries}\label{prelim-sect}

This section recalls some basic facts about affine symmetric groups.
We omit 
 most proofs, since the properties we mention are either well-known (see, e.g., \cite[\S8.3]{BB} or \cite[\S4]{Humphreys}) or follow as simple exercises.
For any map $w: \ZZ\to\ZZ$, write $\inv(w)$ for the set of pairs $(i,j) \in \ZZ\times \ZZ$ with $i<j$ and $w(i) > w(j)$.
If $w \in \tS_n$ and $(i,j) \in \inv(w)$ then $(i+mn,j+mn) \in \inv (w)$ for all $m \in \ZZ$.



%

\begin{proposition}\label{inv-prop}
Let $w \in \tS_n$. Then $\ell(w)$ is the number of equivalence classes in $\inv(w)$ 
under the relation on $\ZZ\times \ZZ$ generated by $(i,j) \sim (i+n,j+n)$.
\end{proposition}


Let $\DesR(w)$ and $\DesL(w)$
denote the right and left descents sets of a permutation $w \in \tS_n$, consisting of the elements $s \in \{s_1,s_2,\dots,s_n\}$
with $\ell(ws)<\ell(w)$ and $\ell(sw)<\ell(w)$, respectively.

\begin{corollary}\label{des-ell-cor}
Let $w \in \tilde S_n$. 
Then $\DesR(w) =  \{ s_i : i \in \ZZ\text{ such that }w(i) > w(i+1)\}$. 
\end{corollary}

For $i < j \not\equiv i \modu n)$, let $t_{ij}= t_{ji} \in \tS_n$ be the permutation
which interchanges $i + mn$ and $j + mn$ for each $m \in \ZZ$ and which fixes all
integers not in $\{i,j\}+n\ZZ$.
Note that $t_{i,i+1} = s_i$.
The elements $t_{ij}$ are precisely the
 \emph{reflections} in $\tS_n$. 
The following is \cite[Proposition 8.3.6]{BB}.



\begin{lemma}\label{bruhat0-lem}
Let $w \in \tS_n$ and $i,j \in \ZZ$ with
 $i <j \not \equiv i \modu n)$.
Suppose $w(i) < w(j)$. Then $\ell(wt_{ij}) \geq  \ell(w)+1$, with equality if and only 
 if no $e \in \ZZ$ satisfies $i<e<j$ and $w(i) < w(e) < w(j)$.
 \end{lemma}

Note for  $z \in \tI_n = \{ w \in \tS_n : w^2=1\}$ and $ i \in \ZZ$ that $z(i) \equiv i \modu n)$ if and only if $z(i) =i$.
 

\begin{lemma}\label{+2lem}
If $i \in \ZZ$ and $z \in \tI_n$ and $i\neq z(i)<z(i+1) \neq i+1$ then $\ell(s_i z s_i) = \ell(z) + 2$.
\end{lemma}
%

For $z \in \tI_n$, let $\cC(z) = \{  (i,j) \in \ZZ\times \ZZ : i < j = z(i)\}$, so that $\cC(z)$ is the set of ordered 2-cycles of $z$.
If $(i,j) \in \cC(z)$ then $(i+mn, j+mn) \in \cC(z)$ for all $m \in \ZZ$.

\begin{definition}\label{abs-len-def}
Define $\ell'(z)$ for $z \in \tI_n$ as the number of equivalence classes in $\cC(z)$ under the relation on $\ZZ\times \ZZ$
generated by $(i,j) \sim (i+n,j+n)$.
\end{definition}

Hultman shows in \cite{H1} that  $\ell'$ is the \emph{absolute length function} on $\tS_n$ restricted to $\tI_n$:
 the function which returns the minimum number of reflections whose product is a given permutation.

\begin{lemma}\label{ell'lem}
If  $i$ and $i+1$ are fixed points of $z \in \tI_n$ then $zs_i  \in \tI_n$ and $\ell'(zs_i) = \ell'(z)+1$.
\end{lemma}


\begin{lemma}\label{commuting-product-lem}
Let $z \in \tI_n$. Then $\ell(z) = \ell'(z)$ if and only if $z$ is a product of commuting simple generators,
i.e., $z=s_{i_1}s_{i_2}\cdots s_{i_l}$ for distinct indices satisfying $i_j \not \equiv i_k \pm 1 \modu n)$ for all $j,k \in [l]$.
\end{lemma}


Say that $j \in \ZZ$ is a \emph{left endpoint} of $z \in \tI_n$ if $j < z(j)$,
a \emph{right endpoint} if $z(j) < j$, and a \emph{fixed point} if $z(j) = j$.
Let $w \in \tS_n$ act on $\ZZ\times \ZZ$ by $w(a,b) = (w(a),w(b))$.

\begin{lemma}\label{triv-lem}
Let $z \in \tI_n$. If
$(i,i+1) \notin \cC(z)$
then $\cC(s_i zs_i) =s_i \cC(z)$, and otherwise $\cC(z) = \cC(s_i zs_i)$.
\end{lemma}

Thus $\ell'(wzw^{-1}) = \ell'(z)$ for all $w \in \tS_n$ and $z \in \tI_n$.
Our last two lemmas are more technical:

\begin{lemma}\label{leftright-bruhat-lem}
Let $z \in \tI_n$ and $i,j \in \ZZ$ with $i < j \not\equiv i \modu n)$.
Assume $j \neq z(i) < z(j) \neq i$.
 If  $i$ is a right endpoint of $z$ or  if $j$ is left endpoint 
 then $\ell(t_{ij}zt_{ij}) > \ell(t_{ij}z) = \ell(zt_{ij})> \ell(z)$.
 \end{lemma}

\begin{proof}
Note that $\ell(t_{ij}z) = \ell(zt_{ij}) > \ell(z)$.
Assume $z(i)<i$.
Since $z(i) < i$, we cannot have $z(i) \equiv i \modu n)$.
If $z(i) \equiv j \modu n)$ then $z(j) \equiv i \modu n)$
so $t_{ij}z(i) < z(i) < z(j) < t_{ij}z(j)$.
If $z(i) \not \equiv j \modu n)$ then $z(j) \not \equiv i \modu n)$,
so either $z(j) = j$ and $t_{ij}z(i) = z(i) < i = t_{ij}z(j)$,
 or $z(j) \not \equiv j \modu n)$
 and $t_{ij}z(i) = z(i) < z(j) = t_{ij}z(j)$.
We deduce that  $\ell(t_{ij} zt_{ij}) > \ell(t_{ij}z) > \ell(z)$.
When $j<z(j)$, the same conclusion follows by a similar argument.
\end{proof}

\begin{lemma}\label{conj-cover-lem}
Let $z \in \tI_n$ and $i,j \in \ZZ$ with $i < j \not\equiv i \modu n)$.
Suppose that either $i,i+1,\dots,j-1$ are 
all right endpoints of $z$ or $i+1,\dots,j-1,j$ are all left endpoints of $z$.
Then:
 \ben
\item[(a)] $\ell(t_{ij}z t_{ij}) > \ell(t_{ij}z)> \ell(z)$ if  $z(i) < z(j)$ 
and
 $\ell(t_{ij}z t_{ij}) < \ell(t_{ij}z) < \ell(z)$ if $z(i) > z(j)$.

\item[(b)] 
$\ell(t_{ij}zt_{ij}) = \ell(z)+2$
 if and only if $\ell(zt_{ij}) = \ell(z) + 1$.
 \een
\end{lemma}

\begin{proof}
The first half of part (a) follows from the previous lemma.
The second half follows from the first
(with $z$ replaced by $t_{ij}zt_{ij}$)
on noting that $i,i+1,\dots,j-1,j$ must be all  
left or all right endpoints of $z$ if $z(i)>z(j)$.
Part (a) implies that if $\ell(zt_{ij}) > \ell(z)+1$ 
then   $\ell(t_{ij}zt_{ij}) > \ell(z)+2$. 
Assume $\ell(zt_{ij}) = \ell(z)+1$ so that $z(i) < z(j)$.
By symmetry, it suffices to consider the case when
$i,i+1,\dots,j-1$ are all right endpoints of $z$.
Fix $e \in \ZZ$ with $i<e<j$, so that either $z(e) < z(i)$ or $z(j) < z(e)$.
If $z(j) < z(e)$,
then since $i$ and $e$ are right endpoints,  
$j$  must also be a right endpoint of $z$,
so
neither   $z(e)$ nor $z(j)$
belongs to $\{i,j\}+n\ZZ$ and $t_{ij}z(j) = z(j) < z(e) = t_{ij}z(e)$.
Suppose instead that $z(e) < z(i)$. We claim that $t_{ij}z(e) < t_{ij}z(i)$.
We cannot have $z(e) \equiv i \modu n)$ since $i$ and $e$ are right endpoints,
so if $t_{ij}z(i) < t_{ij}z(e)$
then we must have
$i - mn < z(e) < j-mn = z(i)$ for some $m >0$.
But if this occurred then  $f = z(e) + mn$ would have $i<f<j$ and $f<z(f)=e+mn$,
contradicting our assumption that $i,i+1,\dots,j-1$ are right endpoints.
This proves our claim, and we conclude that
no $i<e<j$
has $t_{ij}z(i) < t_{ij}z(e) < t_{ij}z(j)$, so
$\ell(t_{ij}z t_{ij}) = \ell(t_{ij}z)+1 = \ell(z)+2$.
\end{proof}

\section{Weighted involutions}\label{weighted-sect}

The goal of the next three sections is to construct a ``length-preserving'' bijection between 
$\tI_n$ and the set of $\NN$-weighted matchings of the cycle graph on $n$ vertices.
Our description of this correspondence will rely on an action of the $0$-Hecke monoid of $\tS_n$ 
on pairs of the following type:

\begin{definition}\label{weighted-def}
A \emph{weighted involution} in $\tS_n$ is a pair $(w, \phi)$
where $w \in \tI_n$ and $\phi$ is a map $\cC(w) \to \NN$ with $\phi(i,j) = \phi(i+n,j+n)$ for all $(i,j) \in \cC(w)$.
We refer to $\phi$ as the \emph{weight map} of $(w,\phi)$.
Define 
 the \emph{weight} of $(w,\phi)$ as the number $\wt(w,\phi) = \sum_\gamma \phi(\gamma)$ where the sum is over a set of cycles $\gamma$ representing the distinct equivalence classes in $\cC(w)$ under the relation  $(i,j) \sim (i+n,j+n)$.
\end{definition}

Let $\cW_n$ be the set of all weighted involutions in $\tS_n$.

\begin{example}\label{W5-ex}
We can represent a weighted involution $(w,\phi) \in \cW_n$ graphically by  drawing the winding diagram 
of $w$ with its arcs labeled by the values  of $\phi$. For example, if $\theta_1, \theta_2, \theta_3 \in \cW_5$
are 
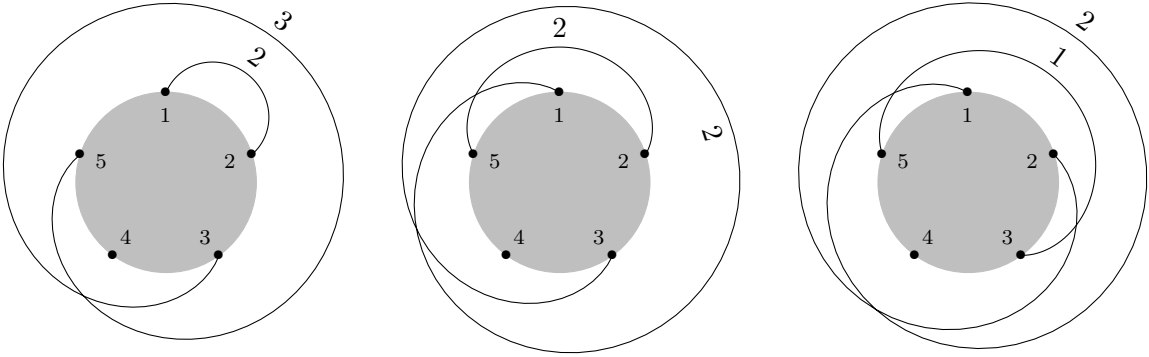
\[
\begin{tikzpicture}[baseline=0,scale=0.3,label/.style={postaction={ decorate,transform shape,decoration={ markings, mark=at position .5 with \node #1;}}}]
{
\draw[fill,lightgray] (0,0) circle (4.0);
\node at (2.44929359829e-16, 4.0) {$_\bullet$};
\node at (1.83697019872e-16, 3.0) {$_1$};
\node at (3.80422606518, 1.2360679775) {$_\bullet$};
\node at (2.85316954889, 0.927050983125) {$_2$};
\node at (2.35114100917, -3.2360679775) {$_\bullet$};
\node at (1.76335575688, -2.42705098312) {$_3$};
\node at (-2.35114100917, -3.2360679775) {$_\bullet$};
\node at (-1.76335575688, -2.42705098312) {$_4$};
\node at (-3.80422606518, 1.2360679775) {$_\bullet$};
\node at (-2.85316954889, 0.927050983125) {$_5$};
\draw [-,>=latex,domain=0:100,samples=100,label={[above]{2}}] plot ({(4.0 + 2.0 * sin(180 * (0.5 + asin(-0.9 + 1.8 * (\x / 100)) / asin(0.9) / 2))) * cos(90 - (0.0 + \x * 0.72))}, {(4.0 + 2.0 * sin(180 * (0.5 + asin(-0.9 + 1.8 * (\x / 100)) / asin(0.9) / 2))) * sin(90 - (0.0 + \x * 0.72))});
\draw [-,>=latex,domain=0:100,samples=100,label={[above]{3}}] plot ({(4.0 + 4.0 * sin(180 * (0.5 + asin(-0.9 + 1.8 * (\x / 100)) / asin(0.9) / 2))) * cos(90 - (144.0 + \x * 5.04))}, {(4.0 + 4.0 * sin(180 * (0.5 + asin(-0.9 + 1.8 * (\x / 100)) / asin(0.9) / 2))) * sin(90 - (144.0 + \x * 5.04))});
}
\end{tikzpicture}
\qquad
\begin{tikzpicture}[baseline=0,scale=0.3,label/.style={postaction={ decorate,transform shape,decoration={ markings, mark=at position .5 with \node #1;}}}]
{
\draw[fill,lightgray] (0,0) circle (4.0);
\node at (2.44929359829e-16, 4.0) {$_\bullet$};
\node at (1.83697019872e-16, 3.0) {$_1$};
\node at (3.80422606518, 1.2360679775) {$_\bullet$};
\node at (2.85316954889, 0.927050983125) {$_2$};
\node at (2.35114100917, -3.2360679775) {$_\bullet$};
\node at (1.76335575688, -2.42705098312) {$_3$};
\node at (-2.35114100917, -3.2360679775) {$_\bullet$};
\node at (-1.76335575688, -2.42705098312) {$_4$};
\node at (-3.80422606518, 1.2360679775) {$_\bullet$};
\node at (-2.85316954889, 0.927050983125) {$_5$};
\draw [-,>=latex,domain=0:100,samples=100,label={[below]{2}}] plot ({(4.0 + 4.0 * sin(180 * (0.5 + asin(-0.9 + 1.8 * (\x / 100)) / asin(0.9) / 2))) * cos(90 - (144.0 + \x * 5.76))}, {(4.0 + 4.0 * sin(180 * (0.5 + asin(-0.9 + 1.8 * (\x / 100)) / asin(0.9) / 2))) * sin(90 - (144.0 + \x * 5.76))});
\draw [-,>=latex,domain=0:100,samples=100,label={[above]{2}}] plot ({(4.0 + 2.0 * sin(180 * (0.5 + asin(-0.9 + 1.8 * (\x / 100)) / asin(0.9) / 2))) * cos(90 - (288.0 + \x * 1.44))}, {(4.0 + 2.0 * sin(180 * (0.5 + asin(-0.9 + 1.8 * (\x / 100)) / asin(0.9) / 2))) * sin(90 - (288.0 + \x * 1.44))});
}
\end{tikzpicture}
\qquad
\begin{tikzpicture}[baseline=0,scale=0.3,label/.style={postaction={ decorate,transform shape,decoration={ markings, mark=at position .5 with \node #1;}}}]
{
\draw[fill,lightgray] (0,0) circle (4.0);
\node at (2.44929359829e-16, 4.0) {$_\bullet$};
\node at (1.83697019872e-16, 3.0) {$_1$};
\node at (3.80422606518, 1.2360679775) {$_\bullet$};
\node at (2.85316954889, 0.927050983125) {$_2$};
\node at (2.35114100917, -3.2360679775) {$_\bullet$};
\node at (1.76335575688, -2.42705098312) {$_3$};
\node at (-2.35114100917, -3.2360679775) {$_\bullet$};
\node at (-1.76335575688, -2.42705098312) {$_4$};
\node at (-3.80422606518, 1.2360679775) {$_\bullet$};
\node at (-2.85316954889, 0.927050983125) {$_5$};
\draw [-,>=latex,domain=0:100,samples=100,label={[above]{2}}] plot ({(4.0 + 4.0 * sin(180 * (0.5 + asin(-0.9 + 1.8 * (\x / 100)) / asin(0.9) / 2))) * cos(90 - (72.0 + \x * 6.48))}, {(4.0 + 4.0 * sin(180 * (0.5 + asin(-0.9 + 1.8 * (\x / 100)) / asin(0.9) / 2))) * sin(90 - (72.0 + \x * 6.48))});
\draw [-,>=latex,domain=0:100,samples=100,label={[above]{1}}] plot ({(4.0 + 2.0 * sin(180 * (0.5 + asin(-0.9 + 1.8 * (\x / 100)) / asin(0.9) / 2))) * cos(90 - (288.0 + \x * 2.16))}, {(4.0 + 2.0 * sin(180 * (0.5 + asin(-0.9 + 1.8 * (\x / 100)) / asin(0.9) / 2))) * sin(90 - (288.0 + \x * 2.16))});
}
\end{tikzpicture}
\]
and we write $\theta_i = (w_i, \phi_i)$, then  $w_1 = t_{1,2} t_{3,10}$, $w_2 = t_{0,2}t_{3,11} $, and
$w_3 = t_{0,3}t_{2,11}$, while $\phi_1(1,2) =2$ and $\phi_1(3,10)=3$,
$\phi_2(3,11)=\phi_2(5,7)=2$, and $\phi_3(2,11)= 2$ and $\phi_3(5,8)=1$.
\end{example}

We identify $\tI_n$ with the subset of weighted involutions of the form 
 $(w,0) \in \cW_n$ with $0$ denoting the unique weight map $\cC(w) \to \{0\}$.
We extend $\ell : \tS_n \to \NN$ and $\ell' : \tI_n \to \NN$ to $\cW_n$ by setting 
\[ \ell(\theta) = \ell(w) + 2\wt(\theta) \qquand \ell'(\theta) = \ell'(w)\qquad\text{for }\theta = (w,\phi) \in \cW_n.\]
Given $(w, \phi) \in \cW_n$, define the \emph{right form} of $\phi$ to be the map $\phi_R : \ZZ \to \NN$ with
$\phi_R(i) =\phi(w(i),i)$ if $w(i) <i$ and with $\phi_R(i)=0$ otherwise.
Likewise, define the \emph{left form} of $\phi$ to be the map
$\phi_L : \ZZ \to \NN$ with $\phi_L(i) = \phi(i, w(i))$ if $i<w(i)$ and with $\phi_L(i) = 0$ otherwise.
Clearly $\phi_L$ and $\phi_R$ each determine $\phi$, given $w$.
We now define operators $\pi_1,\pi_2,\dots, \pi_n$ which act on $\cW_{n}$ on the right and left.
\begin{definition}\label{pi-def}
Let $\theta =(w, \phi) \in \cW_{n}$ and $i \in \ZZ$. 
\ben
\item[(a)] If $\phi_R(i) >\phi_R(i+1)$ then let $\theta\pi_i = (s_iws_i, \psi) \in \cW_n$ where
$\psi $ is the unique weight map with 
\[\psi_R (j) =\begin{cases} \phi_R(i)-1&\text{if }j\equiv i+1\modu n) \\ \phi_R(i+1) &\text{if }j \equiv i \modu n) \\ \phi_R(j) &\text{otherwise}
\end{cases}\qquad\text{for $j \in \ZZ$.}\]
If $\phi_R(i) \leq\phi_R(i+1)$ then  let $\theta \pi_i= \theta$. 

\item[(b)] If  $\phi_L(i+1)>\phi_L(i)$ then  let $\pi_i\theta = (s_iws_i, \chi) \in \cW_n$ where
$\chi $ is the unique weight map with 
\[\chi_L(j) = \begin{cases} \phi_L(i+1)-1&\text{if }j\equiv i \modu n) \\ \phi_L(i) &\text{if }j\equiv i+1 \modu n) \\ \phi_L(j)&\text{otherwise}
\end{cases}
\qquad\text{for $j \in \ZZ$.}\]
If $\phi_L(i+1) \leq \phi_L(i)$ then let $\pi_i\theta = \theta$. 
  \een
\end{definition}

\begin{example}
Define $\theta_1,\theta_2,\theta_3 \in \cW_5$ as in Example~\ref{W5-ex}.
Then $\theta_1 \pi_5 = \theta_2 \pi_1 = \theta_2$ and $\theta_2\pi_2 = \theta_3$.
Form $\theta_2' \in \cW_5$ from $\theta_2$ by replacing the label of the short arc in the picture in Example~\ref{W5-ex} by 1 and the label of the long arc by 3.
Then $\pi_5 \theta_1 =\theta_2'$ and $\pi_2 \theta_2' = \theta_3$.
\end{example}

It may hold that $(\pi_i \theta) \pi_j \neq \pi_i (\theta\pi_j)$;
for example,
if $\theta =(w,\phi)$ where $w = s_1 \in \tS_2$ and $\phi(1,2) = 1$,
then $\pi_0\theta  = (\pi_0\theta)\pi_2 \neq \pi_0 (\theta \pi_2) = \theta \pi_2$.
 Note that $\pi_i = \pi_{i+n}$, as a right and left operator.
 
\begin{proposition}
The left (respectively, right) operators $\pi_i$ satisfy 
(a) $\pi_i^2  = \pi_i$, (b) $\pi_i\pi_j  = \pi_j\pi_i $ if $i \not\equiv j\pm 1 \modu n)$, and 
(c) $\pi_i\pi_{i+1}\pi_i  = \pi_{i+1} \pi_i\pi_{i+1}$
for all $i,j \in \ZZ$.
\end{proposition}

\begin{proof}
It is clear that $\pi_i^2 = \pi_i$ and $\pi_i \pi_j = \pi_j \pi_i$ if $i \not\equiv j\pm 1 \modu n)$.
Fix $\theta = (w,\phi) \in \cW_n$.
Define $a = \phi_R(i)$, $b = \phi_R(i+1)$, and $c =  \phi_R(i+2)$.
We check part (c) carefully as follows:
\begin{itemize}
\item If $a>b>c$ then $\theta \pi_i \pi_{i+1} \pi_i  = \theta \pi_{i+1}\pi_i \pi_{i+1}$ since $s_is_{i+1}s_i = s_{i+1}s_is_{i+1}$.
\item If $a>c \geq b$ then $\theta \pi_i \pi_{i+1} \pi_i  = \theta \pi_{i+1}\pi_i \pi_{i+1} $
is $ \theta \pi_i \pi_{i+1} $ if $a>c+1$ and 
$\theta \pi_i $ if $a=c+1$.
\item If $b\geq a > c$ then $\theta \pi_i \pi_{i+1} \pi_i = \theta \pi_{i+1}\pi_i \pi_{i+1} = \theta \pi_{i+1} \pi_{i}$.
\item If $b> c \geq a$ then $\theta \pi_i \pi_{i+1} \pi_i  = \theta \pi_{i+1}\pi_i \pi_{i+1} = \theta \pi_{i+1} $.
\item If $c\geq a >b$ then $\theta \pi_i \pi_{i+1} \pi_i  = \theta \pi_{i+1}\pi_i \pi_{i+1} = \theta \pi_{i}$.
\item If $c\geq b \geq a$ then $\theta \pi_i \pi_{i+1} \pi_i  = \theta \pi_{i+1}\pi_i \pi_{i+1} = \theta $.
\end{itemize}
One of these cases must occur, so we conclude that $ \pi_i \pi_{i+1} \pi_i  =  \pi_{i+1}\pi_i \pi_{i+1}$
as a right operator. The argument that $ \pi_i \pi_{i+1} \pi_i  =  \pi_{i+1}\pi_i \pi_{i+1}$
as a left operator is symmetric.
\end{proof}

By Matsumoto's theorem, it follows that 
for each $g \in \tS_n$, 
we may define a right (respectively, left) operator $\pi_g$ on $\cW_n$ by 
setting $\pi_g = \pi_{i_1} \pi_{i_2} \cdots \pi_{i_k}$
where $g= s_{i_1} s_{i_2}\cdots s_{i_k}$ is any reduced expression.
Recall the definition of the Demazure product $\circ : \tS_n \times \tS_n \to \tS_n$
from the introduction.

\begin{corollary}
The map $g \mapsto \pi_g$ defines a right (also left) monoid action of $(\tS_n, \circ)$ on $\cW_{n}$.
\end{corollary}

Define 
 $ \tau : \ZZ \to \ZZ$
by $\tau(i) = n+1-i$ and let $w^* = \tau w \tau$ for $w \in \tS_n$.
Then $ w\mapsto w^*$ is an automorphism of $\tS_n$ with 
$s_i^* = s_{n-i}$ for $i \in [n]$, so $s_n^* = s_0= s_n$.
If $w \in \tI_n$ then \[\cC(w^*) = \{ (\tau(j),\tau(i)) : (i,j) \in \cC(w)\}.\]
For  $\phi : \cC(w) \to \NN$ let $\phi^*$ be the map $\cC(w^*) \to \NN$
given by $(\tau(j),\tau(i))\mapsto \phi(i,j)$.
Extend $*$ to $\cW_n$ by setting $\theta^* = (w^*, \phi^*)$ for 
$\theta=(w,\phi) \in \cW_n$. Clearly $(\theta^*)^* = \theta$.
The following is  easy to check:

\begin{lemma}\label{*lem}
Let $ i \in \ZZ$ and $\theta \in \cW_n$. Then
$\wt(\theta) = \wt(\theta^*)$
and $(\pi_i\theta)^* = \theta^* \pi_{n-i}$.  
\end{lemma}

By definition  $\pi_i \theta$ (also $\theta\pi_i$) is either   $\theta$ or has weight one less than $\theta$.
For $\theta \in \cW_n$, let $\DesL(\theta)$
and $\DesR(\theta)$ be the sets of generators $s_i$ for $i\in[n]$ such that
$\pi_i \theta \neq \theta$
and $\theta \pi_i \neq \theta$, respectively.
The map which fixes $s_n$ and maps $s_i \mapsto s_{n-i}$  is a bijection 
$\DesL(\theta) \leftrightarrow \DesR(\theta^*)$.
This means that we only need to prove the right-handed version of the following lemma:

\begin{lemma}\label{weight-lem}
Suppose $\theta \in \cW_n$ and $\wt(\theta) > 0$. Then $\DesL(\theta)$ and $\DesR(\theta)$ are both nonempty.
\end{lemma}

\begin{proof}
Write $\theta = (w,\phi)$.
We have $\DesR(\theta)=\varnothing$ only if no right endpoints $i$ of $w$ have $\phi(w(i), i)>0$ or
if whenever $i$ is  right endpoint   it holds that $i+1$ is also a right endpoint and
$\phi(w(i), i) \leq \phi(w(i+1), i+1)$.
The second case is impossible, and if 
 $\wt(\theta)>0$ then the first case is excluded. 
\end{proof}

Let $\leq$ denote the Bruhat order on $\tS_n$.

\begin{thmdef}\label{omega-thm}
Let $\theta =(w,\phi) \in \cW_n$. The following statements then hold:
\ben
\item[(a)] There are unique elements $g,h \in \tS_n$ with $\ell(g) = \ell(h)=\wt(\theta)$
and $\wt(\pi_g \theta) = \wt(\theta\pi_h)=0$.

\item[(b)] For the elements in part (a), we have
$\pi_g \theta = gwg^{-1}$ and $\theta \pi_h = h^{-1}wh$.

\item[(c)]
Any other elements $g',h' \in \tS_n$ with $\wt(\pi_{g'} \theta) = \wt(\theta\pi_{h'}) = 0$
satisfy $g \leq g'$ and $h\leq h'$.
\een
Define $g_L(\theta)=g^{-1}$ and $g_R(\theta)=h$,
and set
 $\omega_L(\theta) =\pi_g\theta = gwg^{-1}$ and $\omega_R(\theta)=\theta\pi_h =h^{-1}wh$.
 
\end{thmdef}

Note that $g_L(\theta),g_R(\theta) \in \tS_n$ while 
$\omega_L(\theta),\omega_R(\theta) \in \tI_n \subset \cW_n$.

\begin{proof}
Induction and 
Lemma~\ref{weight-lem} imply that $\wt(\theta \pi_{h'}) = 0$ 
for some $h' \in \tS_n$ with $\wt(\theta) \leq \ell(h')$.
Let $h' = s_{i_1}s_{i_2}\cdots s_{i_k}$ be a reduced expression.
Let $h_0 = 1$, for $j \in [k]$ define $h_j$ to be either $h_{j-1}  s_{i_j}$
if $s_{i_j} \in \DesR(\theta \pi_{h_{j-1}})$ or $h_{j-1}$ otherwise, and set $h = h_k$.
By construction $h \leq h'$, $\theta \pi_h = \theta \pi_{h'}$, $\ell(h) = \wt(\theta)$,
and $\theta \pi_h = h^{-1} wh$. 
It remains to show that $h \in \tS_n$ is the unique element 
satisfying both $\wt(\theta\pi_h)=0$ and $\ell(h) = \wt(\theta)$. 
This obviously holds if $\wt(\theta)=0$. 
Assume $g,h \in \tS_n$ are such that 
$\wt(\theta\pi_g) = \wt(\theta\pi_h) = 0$ and $\ell(g)=\ell(h) = \wt(\theta)>0$.
If $s_i$ is a left descent of both $g$ and $h$, then $s_i\in \DesR(\theta)$ 
and it follows by induction that $s_i g = s_ih$,
so $g=h$. Assume $\DesL(g)$ and $\DesL(h)$ are disjoint, and 
choose $s_i \in \DesL(g)$ and $s_j \in \DesL(h)$. 
Both $s_i$ and $s_j$ must belong to $\DesR(\theta)$. 
It is not hard to check that 
(i) if $i \not \equiv j \pm 1 \modu n)$ then 
$\theta \pi_i \pi_j= \theta\pi_j \pi_i$ has weight $\wt(\theta)-2$, while 
(ii) if $i \equiv j \pm 1 \modu n)$ then 
$\theta \pi_i \pi_j \pi_i=\theta \pi_j \pi_i \pi_j$ has weight $\wt(\theta)-3$.

Assume case (i) occurs. By induction, we may assume  
that unique elements $g',g'' \in \tS_n$ exist with $\wt(\theta' g') = \wt (\theta'' g'')=0$,
 $\ell(g') = \wt(\theta')=\wt(\theta)-1$, and 
 $\ell(g'') = \wt(\theta'')=\wt(\theta)-2$ for $\theta' = \theta\pi_i$ and $\theta'' = \theta\pi_i \pi_j = \theta \pi_j \pi_i$. 
 Uniqueness implies that $s_i g = g'$ and $ g' = s_jg''$, so $g = s_is_j g'' = s_js_i g''$ 
 where $\ell(g) =\ell(g'') +2$. But this means that $s_j \in \DesL(g)$, 
 contradicting our assumption otherwise.
One reaches a similar contradiction in case (ii). 
This proves the right-handed version of theorem.
The left-handed version follows by symmetric arguments.
\end{proof}

\begin{example}
If $\theta_1,\theta_2,\theta_3 \in \cW_5$ are as in Example~\ref{W5-ex}, then we have 
\[\omega_R(\theta_1) = \omega_R(\theta_2) = \omega_R(\theta_3) = t_{1,13} t_{5,9} =\
\begin{tikzpicture}[baseline=0,scale=0.22,label/.style={postaction={ decorate,transform shape,decoration={ markings, mark=at position .5 with \node #1;}}}]
{
\draw[fill,lightgray] (0,0) circle (4.0);
\node at (2.44929359829e-16, 4.0) {$_\bullet$};
\node at (1.71450551881e-16, 2.8) {$_{1}$};
\node at (3.80422606518, 1.2360679775) {$_\bullet$};
\node at (2.66295824563, 0.86524758425) {$_{2}$};
\node at (2.35114100917, -3.2360679775) {$_\bullet$};
\node at (1.64579870642, -2.26524758425) {$_{3}$};
\node at (-2.35114100917, -3.2360679775) {$_\bullet$};
\node at (-1.64579870642, -2.26524758425) {$_{4}$};
\node at (-3.80422606518, 1.2360679775) {$_\bullet$};
\node at (-2.66295824563, 0.86524758425) {$_{5}$};
\draw [blue,-,>=latex,domain=0:100,samples=100] plot ({(4.0 + 4.0 * sin(180 * (0.5 + asin(-0.9 + 1.8 * (\x / 100)) / asin(0.9) / 2))) * cos(90 - (0.0 + \x * 8.64))}, {(4.0 + 4.0 * sin(180 * (0.5 + asin(-0.9 + 1.8 * (\x / 100)) / asin(0.9) / 2))) * sin(90 - (0.0 + \x * 8.64))});
\draw [red,-,>=latex,domain=0:100,samples=100] plot ({(4.0 + 2.0 * sin(180 * (0.5 + asin(-0.9 + 1.8 * (\x / 100)) / asin(0.9) / 2))) * cos(90 - (288.0 + \x * 2.88))}, {(4.0 + 2.0 * sin(180 * (0.5 + asin(-0.9 + 1.8 * (\x / 100)) / asin(0.9) / 2))) * sin(90 - (288.0 + \x * 2.88))});
}
\end{tikzpicture}
.\]
The arcs in the winding diagram are coloured red and blue to make them easier to distinguish.
\end{example}

The following is clear by induction from Lemma~\ref{*lem}:

\begin{lemma}\label{*omega-lem}
If $\theta \in \cW_n$ then $g_L(\theta^*) = g_R(\theta)^*$ and $\omega_L(\theta^*) = \omega_R(\theta)^*$.
\end{lemma}

\section{Admissibility}\label{admissibility-sect}

We define $\ell(\theta) = \ell(w) + 2\wt(\theta)$ for $\theta = (w,\phi) \in \cW_n$.
The following terminology identifies a class of weighted involutions whose lengths are unaffected by  the operators
$\pi_i$.

\begin{definition}\label{adm-def}
Relative to a given weighted involution $\theta = (w,\phi) \in \cW_n$, we say that:
\ben
\item[(a)] 
A sequence  $(x,y),(a_0,b_0),(a_1,b_1),\dots,(a_k,b_k) \in \cC(w)$
with
$x<a_k<b_k < \dots < b_1 < b_0 < y$ 
is \emph{right-inadmissible} if 
$\phi(a_i,b_i) + b_i + i \geq \phi(x,y) +y$ for all $i$.

\item[(b)] \
A sequence $(x,y),(a_0,b_0),(a_1,b_1),\dots,(a_k,b_k) \in \cC(w)$
with 
$x<a_0 < a_1 < \dots < a_k <b_k < y$ 
is \emph{left-inadmissible} if  
$\phi(a_i,b_i) - a_i+i  \geq \phi(x,y) -x$ for all $i$.

\een
A weighted involution $\theta \in \cW_n$ is \emph{right-admissible} (respectively, \emph{left-admissible})
if it has no right-inadmissible (respectively, left-inadmissible) sequences of cycles.
\end{definition}

If $w$ is \emph{non-nesting} in the sense of having  no cycles $(x,y),(a,b) \in \cC(w)$
with $x<a<b<y$ then $\theta=(w,\phi) \in \cW_n$ is 
 both right- and left-admissible. The following is also easy to see:
 
 \begin{lemma}\label{*adm-lem}
If $\theta$ is right-admissible (respectively, left-admissible) 
then $\theta^*$ is left-admissible (respectively, right-admissible).
\end{lemma}

Say that $\theta' \in \cW_n$ is a \emph{descendant} of $\theta  \in \cW_n$ 
if $\theta'=\theta\pi_{g}$ for some $g \in \tS_n$.
The idea behind right-admissibility is to give a condition ensuring that
no descendant $\theta' = (w',\phi')$ of $\theta$ allows
$w'$ to have
nesting cycles  $(i,k+1),(j,k) \in \cC(w')$
with $i<j < k$ and  $\phi'(j,k) > \phi'(i,k+1)$, since such a weighted involution would have
 $\ell(\theta'\pi_k) < \ell(\theta')$.
Our formulation of left-admissibility is motivated by symmetric considerations.
This is enough to make the operators $\pi_i$ length-preserving:

\begin{theorem}\label{adm-thm}
Suppose $\theta \in \cW_n$ and $i \in \ZZ$.
\ben
 \item[(a)]  If $\theta$ is right-admissible 
then $\theta \pi_i$ is  right-admissible and $\ell(\theta\pi_i) = \ell(\theta)$.   

\item[(b)] If $\theta$ is left-admissible 
then $\pi_i \theta$ is  left-admissible and $\ell(\pi_i\theta) = \ell(\theta)$.
  
\een                                                                             
\end{theorem}
 
\begin{proof}
Write $\theta = (w,\phi) \in \cW_n$.
We only prove part (a), since 
part (b) is equivalent by Lemmas~\ref{*lem} and \ref{*adm-lem}.
Assume $\theta$ is right-admissible and $s_i \in \DesR(\theta)$.
Then $w(i) <i$ and $\phi_R(i) >  \phi_R(i+1)$.
It follows $w(i) < w(i+1)$ since otherwise we would have $x < a_0<b_0<y$
for the cycles $(x,y) = (w(i+1),i+1)$ and $(a_0,b_0) = (w(i),i)$,
and it would hold that $\phi_R(b_0)+b_0= \phi_R(i) + i  \geq \phi_R(i+1) +i+1  = \phi_R(y)+y$.
We conclude by Lemma~\ref{+2lem} that $\ell(s_iws_i) = \ell(w)+2$, so $\ell(\theta\pi_i) = \ell(\theta)$.
It remains to check that $\theta\pi_i$ is right-admissible.

 Suppose $(x,y),(a_0,b_0),(a_1,b_1),\dots,(a_k,b_k)  \in \cC(w)$
 is a right-inadmissible sequence for $\theta$.
 If $b_k < y-n$
 and $j \in \{0,1,\dots,k\}$ is the smallest index with $b_j < y-n$,
 then $j \leq n$ and 
the sequence  $(x-n,y-n),(a_j,b_j),(a_{j+1},b_{j+1}),\dots,(a_k,b_k) \in \cC(w)$
is also right-inadmissible.

Now suppose $s_i \in \DesR(\theta)$ but $\theta\pi_i$ is not right-admissible.
Write $\theta'=(w',\phi') = \theta\pi_i$ so that $w'=s_iws_i$ 
and recall by Lemma~\ref{triv-lem} that $\cC(w') = s_i\cC(w)$.
 The number $i$ must be a right endpoint of $w$ so $i+1$ must be a right endpoint of $w'$.
Suppose $(x',y'),(a'_0,b'_0),(a'_1,b'_1),\dots,(a'_k,b'_k)  \in \cC(w')$
 is a right-inadmissible sequence for $\theta'$.
 By the observations in the previous paragraph, 
 we may assume that $y'-n \leq b'_k < \dots < b'_1 < b'_0 < y'$.
 The first inequality cannot be strict since $w'(y') = x' < a'_k = w'(b'_k)$, so $y'-n < b'_k$.
 Let $(x,y) = (s_i(x'), s_i(y'))$ and $(a_j,b_j)  = (s_i(a_j'), s_i(b_j')) \in \cC(w)$ for each $j$.
  If at most one number among $b'_k < \dots <b'_1 <b'_0 < y'$ is 
  congruent to $i$ or $i+1$ modulo $n$ and $y' \not\equiv i\modu n)$,
 then the obvious sequence of cycles 
 $(x,y),(a_0,b_0),(a_1,b_1),\dots,(a_k,b_k)$ is inadmissible for $\theta$.
 If $y' \equiv i \modu n)$ and $b'_k \equiv i+1 \modu n)$ then $ (x-n,y-n),(a_k,b_k)$ 
 is an inadmissible sequence for $\theta$.
 If $y' \equiv i \modu n)$ and $b'_k \not\equiv i+1\modu n)$  then
 \[
 (x,y),(w(y'),y'),(a_0,b_0),(a_1,b_1),\dots,(a_k,b_k) 
 \]
 is  inadmissible  for $\theta$.
If $b'_j \equiv i \modu n)$ and $b'_{j-1} \equiv i+1 \modu n)$ for $j\in [k-1]$, then
the sequence
\[
 (x,y), (a_0,b_0),(a_1,b_1),\dots, (a_j,b_j), (a_{j-1},b_{j-1}),\dots, (a_k,b_k) 
 \]
 is inadmissible for $\theta$. Similarly, if $k>0$ and 
 $b'_k \equiv i \modu n)$ and $b'_{k-1} \equiv i+1\modu n)$ 
 then
\[
 (x,y), (a_0,b_0),(a_1,b_1),\dots,(a_{k-2},b_{k-2}), (a_k,b_k) 
 \] 
 is inadmissible for $\theta$. Finally, it cannot happen that 
 $b'_0 \equiv i \modu n)$ and $y' \equiv i+1\modu n)$,
 since then  the condition   $\phi'_R(b'_0) + b'_0 \geq \phi'_R(y') + y'$
 would imply that
$ \phi'_R(i) > \phi'_R(i+1) $ and $\phi_R(i) \leq \phi_R(i+1)$, contradicting the fact that $s_i \in \DesR(\theta)$.
In this way, we deduce the contrapositive of part (a), 
i.e., that if $\theta\pi_i$ is not right-admissible then $\theta$ is also not right-admissible.
\end{proof}

It is obvious from Theorem-Definition~\ref{omega-thm} that 
$\ell'(\omega_R(\theta)) = \ell'(\omega_L(\theta)) = \ell'(\theta)$ for 
all $\theta \in \cW_n$.
The following statement is immediate from the previous theorem by induction.

\begin{corollary}\label{limit-cor}
Let $\theta \in \cW_{n}$. 
If $\theta$ is right-admissible then   $\ell(\omega_R(\theta)) = \ell(\theta)$.
If $\theta$ is left-admissible then $\ell(\omega_L(\theta)) = \ell(\theta)$.
\end{corollary}


One natural set of ``extremal'' elements on $\cW_n$ is given by the subset $\tI_n$. 
Another is this:

\begin{definition}\label{cWmin-def}
Define $\cWmin_{n}$ as the set of weighted involutions $\theta=(w,\phi) \in \cW_n$
with $\ell'(w) = \ell(w)$, i.e., such that $w$ is a product of commuting simple reflections.
\end{definition}

Every element of $\cWmin_{n}$ is both left- and right-admissible.
The elements of $\cWmin_{n}$ 
are in bijection with 
$\NN$-weighted matchings in $\cC_n$, the cycle graph on $n$ vertices,
 which explains our notation.

\begin{example}
The set $\cWmin_4$ consists of the weighted involutions of the form
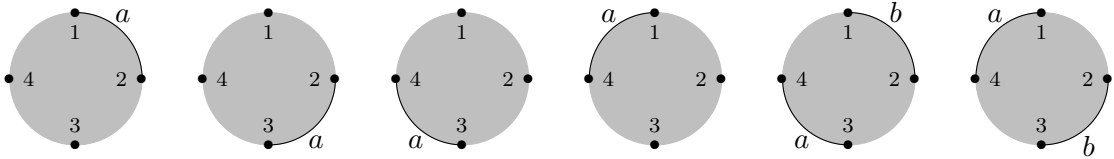
\[
\begin{tikzpicture}[baseline=0,scale=0.22,label/.style={postaction={ decorate,decoration={ markings, mark=at position .5 with \node #1;}}}]
{
\draw[fill,lightgray] (0,0) circle (4.0);
\node at (2.44929359829e-16, 4.0) {$_\bullet$};
\node at (1.71450551881e-16, 2.8) {$_1$};
\node at (4.0, 0.0) {$_\bullet$};
\node at (2.8, 0.0) {$_2$};
\node at (2.44929359829e-16, -4.0) {$_\bullet$};
\node at (1.71450551881e-16, -2.8) {$_3$};
\node at (-4.0, -4.89858719659e-16) {$_\bullet$};
\node at (-2.8, -3.42901103761e-16) {$_4$};
\draw [-,>=latex,domain=0:100,samples=100,label={[above]{$a$}}] plot ({(4.0 + 0.0 * sin(180 * (0.5 + asin(-0.9 + 1.8 * (\x / 100)) / asin(0.9) / 2))) * cos(90 - (0.0 + \x * 0.9))}, {(4.0 + 0.0 * sin(180 * (0.5 + asin(-0.9 + 1.8 * (\x / 100)) / asin(0.9) / 2))) * sin(90 - (0.0 + \x * 0.9))});
}
\end{tikzpicture}
\quad
\begin{tikzpicture}[baseline=0,scale=0.22,label/.style={postaction={ decorate,decoration={ markings, mark=at position .5 with \node #1;}}}]
{
\draw[fill,lightgray] (0,0) circle (4.0);
\node at (2.44929359829e-16, 4.0) {$_\bullet$};
\node at (1.71450551881e-16, 2.8) {$_1$};
\node at (4.0, 0.0) {$_\bullet$};
\node at (2.8, 0.0) {$_2$};
\node at (2.44929359829e-16, -4.0) {$_\bullet$};
\node at (1.71450551881e-16, -2.8) {$_3$};
\node at (-4.0, -4.89858719659e-16) {$_\bullet$};
\node at (-2.8, -3.42901103761e-16) {$_4$};
\draw [-,>=latex,domain=0:100,samples=100,label={[below]{$a$}}] plot ({(4.0 + 0.0 * sin(180 * (0.5 + asin(-0.9 + 1.8 * (\x / 100)) / asin(0.9) / 2))) * cos(90 - (90.0 + \x * 0.9))}, {(4.0 + 0.0 * sin(180 * (0.5 + asin(-0.9 + 1.8 * (\x / 100)) / asin(0.9) / 2))) * sin(90 - (90.0 + \x * 0.9))});
}
\end{tikzpicture}
\quad
\begin{tikzpicture}[baseline=0,scale=0.22,label/.style={postaction={ decorate,decoration={ markings, mark=at position .5 with \node #1;}}}]
{
\draw[fill,lightgray] (0,0) circle (4.0);
\node at (2.44929359829e-16, 4.0) {$_\bullet$};
\node at (1.71450551881e-16, 2.8) {$_1$};
\node at (4.0, 0.0) {$_\bullet$};
\node at (2.8, 0.0) {$_2$};
\node at (2.44929359829e-16, -4.0) {$_\bullet$};
\node at (1.71450551881e-16, -2.8) {$_3$};
\node at (-4.0, -4.89858719659e-16) {$_\bullet$};
\node at (-2.8, -3.42901103761e-16) {$_4$};
\draw [-,>=latex,domain=0:100,samples=100,label={[below]{$a$}}] plot ({(4.0 + 0.0 * sin(180 * (0.5 + asin(-0.9 + 1.8 * (\x / 100)) / asin(0.9) / 2))) * cos(90 - (180.0 + \x * 0.9))}, {(4.0 + 0.0 * sin(180 * (0.5 + asin(-0.9 + 1.8 * (\x / 100)) / asin(0.9) / 2))) * sin(90 - (180.0 + \x * 0.9))});
}
\end{tikzpicture}
\quad
\begin{tikzpicture}[baseline=0,scale=0.22,label/.style={postaction={ decorate,decoration={ markings, mark=at position .5 with \node #1;}}}]
{
\draw[fill,lightgray] (0,0) circle (4.0);
\node at (2.44929359829e-16, 4.0) {$_\bullet$};
\node at (1.71450551881e-16, 2.8) {$_1$};
\node at (4.0, 0.0) {$_\bullet$};
\node at (2.8, 0.0) {$_2$};
\node at (2.44929359829e-16, -4.0) {$_\bullet$};
\node at (1.71450551881e-16, -2.8) {$_3$};
\node at (-4.0, -4.89858719659e-16) {$_\bullet$};
\node at (-2.8, -3.42901103761e-16) {$_4$};
\draw [-,>=latex,domain=0:100,samples=100,label={[above]{$a$}}] plot ({(4.0 + 0.0 * sin(180 * (0.5 + asin(-0.9 + 1.8 * (\x / 100)) / asin(0.9) / 2))) * cos(90 - (270.0 + \x * 0.9))}, {(4.0 + 0.0 * sin(180 * (0.5 + asin(-0.9 + 1.8 * (\x / 100)) / asin(0.9) / 2))) * sin(90 - (270.0 + \x * 0.9))});
}
\end{tikzpicture}
\quad
\begin{tikzpicture}[baseline=0,scale=0.22,label/.style={postaction={ decorate,decoration={ markings, mark=at position .5 with \node #1;}}}]
{
\draw[fill,lightgray] (0,0) circle (4.0);
\node at (2.44929359829e-16, 4.0) {$_\bullet$};
\node at (1.71450551881e-16, 2.8) {$_1$};
\node at (4.0, 0.0) {$_\bullet$};
\node at (2.8, 0.0) {$_2$};
\node at (2.44929359829e-16, -4.0) {$_\bullet$};
\node at (1.71450551881e-16, -2.8) {$_3$};
\node at (-4.0, -4.89858719659e-16) {$_\bullet$};
\node at (-2.8, -3.42901103761e-16) {$_4$};
\draw [-,>=latex,domain=0:100,samples=100,label={[above]{$b$}}] plot ({(4.0 + 0.0 * sin(180 * (0.5 + asin(-0.9 + 1.8 * (\x / 100)) / asin(0.9) / 2))) * cos(90 - (0.0 + \x * 0.9))}, {(4.0 + 0.0 * sin(180 * (0.5 + asin(-0.9 + 1.8 * (\x / 100)) / asin(0.9) / 2))) * sin(90 - (0.0 + \x * 0.9))});
\draw [-,>=latex,domain=0:100,samples=100,label={[below]{$a$}}] plot ({(4.0 + 0.0 * sin(180 * (0.5 + asin(-0.9 + 1.8 * (\x / 100)) / asin(0.9) / 2))) * cos(90 - (180.0 + \x * 0.9))}, {(4.0 + 0.0 * sin(180 * (0.5 + asin(-0.9 + 1.8 * (\x / 100)) / asin(0.9) / 2))) * sin(90 - (180.0 + \x * 0.9))});
}
\end{tikzpicture}
\quad
\begin{tikzpicture}[baseline=0,scale=0.22,label/.style={postaction={ decorate,decoration={ markings, mark=at position .5 with \node #1;}}}]
{
\draw[fill,lightgray] (0,0) circle (4.0);
\node at (2.44929359829e-16, 4.0) {$_\bullet$};
\node at (1.71450551881e-16, 2.8) {$_1$};
\node at (4.0, 0.0) {$_\bullet$};
\node at (2.8, 0.0) {$_2$};
\node at (2.44929359829e-16, -4.0) {$_\bullet$};
\node at (1.71450551881e-16, -2.8) {$_3$};
\node at (-4.0, -4.89858719659e-16) {$_\bullet$};
\node at (-2.8, -3.42901103761e-16) {$_4$};
\draw [-,>=latex,domain=0:100,samples=100,label={[below]{$b$}}] plot ({(4.0 + 0.0 * sin(180 * (0.5 + asin(-0.9 + 1.8 * (\x / 100)) / asin(0.9) / 2))) * cos(90 - (90.0 + \x * 0.9))}, {(4.0 + 0.0 * sin(180 * (0.5 + asin(-0.9 + 1.8 * (\x / 100)) / asin(0.9) / 2))) * sin(90 - (90.0 + \x * 0.9))});
\draw [-,>=latex,domain=0:100,samples=100,label={[above]{$a$}}] plot ({(4.0 + 0.0 * sin(180 * (0.5 + asin(-0.9 + 1.8 * (\x / 100)) / asin(0.9) / 2))) * cos(90 - (270.0 + \x * 0.9))}, {(4.0 + 0.0 * sin(180 * (0.5 + asin(-0.9 + 1.8 * (\x / 100)) / asin(0.9) / 2))) * sin(90 - (270.0 + \x * 0.9))});
}
\end{tikzpicture}
\]
where $a,b \in \NN$ are arbitrary natural numbers.
\end{example}

\begin{proposition}\label{match-prop}
There are $\frac{n}{n-k}\binom{n-k}{k}$ distinct $k$-element matchings in $\cC_n$.
\end{proposition}

\begin{proof}
This is well-known (see \cite[A034807]{OEIS}) and also follows as an instructive exercise.
\end{proof}

In the next section we show that the length-preserving 
 maps
$
\omega_R : \cWmin_{n} \to \tI_{n}
$
and 
$ \omega_L : \cWmin_{n} \to \tI_{n}
$
are bijections.
Here we construct the maps which will turn out to be their inverses.

\begin{propdef}\label{lambda-def}
Fix $ \theta=(w,\phi) \in \cW_{n}$ and $m \in \ZZ$.
Let $k = \ell'(w)$ and suppose
\[a_1 < a_2 < \dots <a_k
\qquand
 a'_1<a'_2<\dots<a'_k
\] are the respective sequences of left and right endpoints of $w$ in $m+[n]$.
For $j \in [k]$, define $p_j$, $p_j'$, $q_j$, and $q_j'$ as 
the numbers of cycles $(x,y) \in \cC(w)$ respectively
satisfying
\[ x < a_j < y,\qquad x < a'_j < y,\qquad x<a_j <y < w(a_j), \qquand w(a'_j)<x<a'_j < y.\]
Let $i_j = a_j + p_j  $ and $i_j' = a'_j - p_j'-1$, and let 
$
u= s_{i_1} s_{i_2}\cdots s_{i_k}
$ 
and 
$
v= s_{i'_1} s_{i'_2}\cdots s_{i'_k}
$.
Finally:
\ben
\item[(a)]
Let
 $\lambda_R(\theta) = ( u,\psi) \in \cW_n$ where 
$\psi: \cC(u) \to \NN$ is the weight map with
\[\psi(i_j,i_j+1) = \phi(a_j, w(a_j)) + w(a_j) - a_j - q_j - 1\qquad\text{for $j \in [k]$}.\]

\item[(b)]
Let $\lambda_L(\theta) = ( v,\chi) \in \cW_n$ where $\chi : \cC(v) \to \NN$ is the weight map with
\[\chi(i'_j,i'_j+1) =\phi(w(a'_j), a'_j) + a'_j - w(a'_j)- q'_j - 1\qquad\text{for $j \in [k]$.}\]
\een
The weighted involutions $\lambda_R(\theta)$ and $\lambda_L(\theta)$ then 
both belong to $\cWmin_{n}$, and do not depend on $m$.
\end{propdef}

If $w \in \tI_n \subset \cW_n$ then $\lambda_R(w) = \lambda_R(w,0)$ and $\lambda_L(w) = \lambda_L(w,0)$.
We delay the proof of the proposition  to give an example and state a lemma.

\begin{example}
Suppose $n=4$, $k=2$, $m=0$, and 
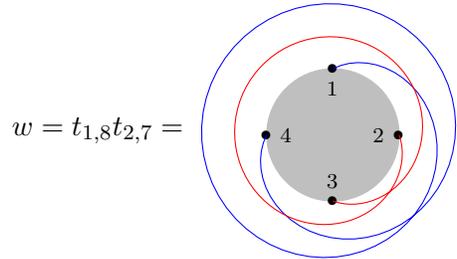
\[
w = t_{1,8}t_{2,7} = \
\begin{tikzpicture}[baseline=0,scale=0.22,label/.style={postaction={ decorate,transform shape,decoration={ markings, mark=at position .5 with \node #1;}}}]
{
\draw[fill,lightgray] (0,0) circle (4.0);
\node at (2.44929359829e-16, 4.0) {$_\bullet$};
\node at (1.71450551881e-16, 2.8) {$_1$};
\node at (4.0, 0.0) {$_\bullet$};
\node at (2.8, 0.0) {$_2$};
\node at (2.44929359829e-16, -4.0) {$_\bullet$};
\node at (1.71450551881e-16, -2.8) {$_3$};
\node at (-4.0, -4.89858719659e-16) {$_\bullet$};
\node at (-2.8, -3.42901103761e-16) {$_4$};
\draw [blue,-,>=latex,domain=0:100,samples=100] plot ({(4.0 + 4.0 * sin(180 * (0.5 + asin(-0.9 + 1.8 * (\x / 100)) / asin(0.9) / 2))) * cos(90 - (0.0 + \x * 6.3))}, {(4.0 + 4.0 * sin(180 * (0.5 + asin(-0.9 + 1.8 * (\x / 100)) / asin(0.9) / 2))) * sin(90 - (0.0 + \x * 6.3))});
\draw [red,-,>=latex,domain=0:100,samples=100] plot ({(4.0 + 2.0 * sin(180 * (0.5 + asin(-0.9 + 1.8 * (\x / 100)) / asin(0.9) / 2))) * cos(90 - (90.0 + \x * 4.5))}, {(4.0 + 2.0 * sin(180 * (0.5 + asin(-0.9 + 1.8 * (\x / 100)) / asin(0.9) / 2))) * sin(90 - (90.0 + \x * 4.5))});
}
\end{tikzpicture}
.\]
Then $(a_1, a_2) = (1,2)$ and $(a'_1,a'_2) = (3,4)$ and the 
numbers $p_j,p_j',q_j,q_j',i_j,i_j'$ are computed as follows.
 First, we have $(p_1,p_2) = (2,3)$ and $(q_1,q_2) = (2,2)$ 
so $(i_1,i_2) = (3,5)$ and $\lambda_R(w) = (s_1s_3,\psi)$ where 
$\psi(3,4) = 4 $
and
$\psi(1,2) = 2 .$
 We have $(p_1',p_2') = (3,2)$ and $(q_1',q_2') = (2,2)$ so $(i_1',i_2') = (-1, 1)$
and $\lambda_L(w) =(s_1s_3,\chi)$ where 
$\chi(3,4) = 2  
$
and
$
\chi(1,2) = 4 .
$
In terms of winding diagrams, 
\[
\lambda_R(w) =\ 
\begin{tikzpicture}[baseline=0,scale=0.22,label/.style={postaction={ decorate,decoration={ markings, mark=at position .5 with \node #1;}}}]
{
\draw[fill,lightgray] (0,0) circle (4.0);
\node at (2.44929359829e-16, 4.0) {$_\bullet$};
\node at (1.71450551881e-16, 2.8) {$_1$};
\node at (4.0, 0.0) {$_\bullet$};
\node at (2.8, 0.0) {$_2$};
\node at (2.44929359829e-16, -4.0) {$_\bullet$};
\node at (1.71450551881e-16, -2.8) {$_3$};
\node at (-4.0, -4.89858719659e-16) {$_\bullet$};
\node at (-2.8, -3.42901103761e-16) {$_4$};
\draw [-,>=latex,domain=0:100,samples=100,label={[above]{$2$}}] plot ({(4.0 + 0.0 * sin(180 * (0.5 + asin(-0.9 + 1.8 * (\x / 100)) / asin(0.9) / 2))) * cos(90 - (0.0 + \x * 0.9))}, {(4.0 + 0.0 * sin(180 * (0.5 + asin(-0.9 + 1.8 * (\x / 100)) / asin(0.9) / 2))) * sin(90 - (0.0 + \x * 0.9))});
\draw [-,>=latex,domain=0:100,samples=100,label={[below]{$4$}}] plot ({(4.0 + 0.0 * sin(180 * (0.5 + asin(-0.9 + 1.8 * (\x / 100)) / asin(0.9) / 2))) * cos(90 - (180.0 + \x * 0.9))}, {(4.0 + 0.0 * sin(180 * (0.5 + asin(-0.9 + 1.8 * (\x / 100)) / asin(0.9) / 2))) * sin(90 - (180.0 + \x * 0.9))});
}
\end{tikzpicture}
\qquand
\lambda_L(w) =\ 
\begin{tikzpicture}[baseline=0,scale=0.22,label/.style={postaction={ decorate,decoration={ markings, mark=at position .5 with \node #1;}}}]
{
\draw[fill,lightgray] (0,0) circle (4.0);
\node at (2.44929359829e-16, 4.0) {$_\bullet$};
\node at (1.71450551881e-16, 2.8) {$_1$};
\node at (4.0, 0.0) {$_\bullet$};
\node at (2.8, 0.0) {$_2$};
\node at (2.44929359829e-16, -4.0) {$_\bullet$};
\node at (1.71450551881e-16, -2.8) {$_3$};
\node at (-4.0, -4.89858719659e-16) {$_\bullet$};
\node at (-2.8, -3.42901103761e-16) {$_4$};
\draw [-,>=latex,domain=0:100,samples=100,label={[above]{$4$}}] plot ({(4.0 + 0.0 * sin(180 * (0.5 + asin(-0.9 + 1.8 * (\x / 100)) / asin(0.9) / 2))) * cos(90 - (0.0 + \x * 0.9))}, {(4.0 + 0.0 * sin(180 * (0.5 + asin(-0.9 + 1.8 * (\x / 100)) / asin(0.9) / 2))) * sin(90 - (0.0 + \x * 0.9))});
\draw [-,>=latex,domain=0:100,samples=100,label={[below]{$2$}}] plot ({(4.0 + 0.0 * sin(180 * (0.5 + asin(-0.9 + 1.8 * (\x / 100)) / asin(0.9) / 2))) * cos(90 - (180.0 + \x * 0.9))}, {(4.0 + 0.0 * sin(180 * (0.5 + asin(-0.9 + 1.8 * (\x / 100)) / asin(0.9) / 2))) * sin(90 - (180.0 + \x * 0.9))});
}
\end{tikzpicture}
.
\]
\end{example}

\begin{lemma}
\label{*lambda-lem}
If $\theta \in \cW_n$ then $\lambda_L(\theta^*) = \lambda_R(\theta)^*$.
\end{lemma}

\begin{proof}
This is evident from the symmetric definitions of $\lambda_R(\theta)$ and $\lambda_L(\theta)$.
\end{proof}

\begin{proof}[Proof of Proposition-Definition~\ref{lambda-def}]
The fact that $\lambda_R(\theta)$ and $\lambda_L(\theta)$ do not depend on $m$
holds since $w(i+n) = w(i) + n$ for all $i \in \ZZ$.
We must show that $\lambda_R(\theta)$ and $\lambda_L(\theta)$ actually belong to $\cWmin_{n}$.
It is enough to prove that $\lambda_R(\theta) \in \cWmin_{n}$.
For this, we need to check that $i_j + 1 < i_{j+1}$ for $j \in [k-1]$ and $i_k + 1 < i_1 + n$.
Let $a$ be any left endpoint of $w$ and let $\tilde a$ be 
the smallest left endpoint greater than $a$.
Define $p$ and $\tilde p$ as the respective numbers of 
cycles $(x,y) \in \cC(w)$ with $x<a<y$ and $x<\tilde a <y$.
It is enough to verify  that $a+p +1 < \tilde a + \tilde p$. 
This holds if and only if $a+r+1 < \tilde a+\tilde r$
where $r$ and $\tilde r$ are the respective numbers of cycles $(x,y) \in \cC(w)$
with $x<a<y<\tilde a$ and $a\leq x < \tilde a < y$. Clearly $a+r+1\leq \tilde a$,
with equality only if $a < \tilde a < w(a)$ in which case $\tilde r >0$.
\end{proof}


Given $\theta= (w,\phi) \in \cW_n$, recall from Section~\ref{weighted-sect}
the definition of  $\phi_R,\phi_L : \ZZ \to \NN$.

\begin{lemma}
Let $\theta=(w,\phi) \in \cW_n$. If $s_i \notin \DesR(w)$
then $\lambda_R( \theta\pi_i) = \lambda_R(\theta)$ and $\lambda_L(\pi_i\theta) = \lambda_R(\theta)$.
\end{lemma}

\begin{proof}
Choose $i \in \ZZ$ with $s_i \notin \DesR(w)$.
By Lemmas~\ref{*lem} and \ref{*lambda-lem} it suffices to prove 
that $\lambda_R(\theta\pi_i) = \lambda_R(\theta)$.
Assume $\theta\pi_i \neq \theta$.
Let $\tilde w=s_i ws_i$ and define $\tilde \phi : \cC(\tilde w) \to \NN$
 such that $\theta\pi_i = (\tilde w, \tilde \phi)$. 
Choose $m \in \ZZ$ such that $\{i,i+1\} \subset m + [n]$, and define
Define $a_j, b_j,p_j, q_j, i_j$, and $\psi$ relative to $w$ and $\phi$ 
as in Proposition-Definition~\ref{lambda-def}. Define 
$\tilde a_j, \tilde b_j, \tilde p_j, \tilde q_j, \tilde i_j,$ and $\tilde \psi$
analogously relative to 
 $\tilde w$ and $\tilde\phi$. We deduce that $\lambda_R(\theta\pi_i) = \lambda_R(\theta)$ 
 by comparing these quantities as follows:
 \begin{itemize}
\item
If $w(i)  < w(i+1) < i < i+1$ then $a_j = \tilde a_j$ and $p_j = \tilde p_j$ for all $j$,
and there are two indices $j$ for which $\phi_R(b_j)$, $b_j$, $q_j$ and 
$\tilde \phi_R(\tilde b_j)$, $\tilde b_j$, $\tilde q_j$ are different:
 for one index $\tilde \phi_R(\tilde b_j) = \phi_R(b_j)-1$ and $\tilde b_j= b_j+1$ and $\tilde q_j=q_j$,
and for the other $\tilde\phi_R(\tilde b_j) = \phi_R(b_j)$ and $\tilde b_j = b_j-1$ and $\tilde q_j=q_j-1$.

\item
If $w(i)  <  i < i+1= w(i+1)$ then  $a_j=\tilde a_j$ and $p_j=\tilde p_j$ for all $j$,
and the numbers $\phi_R(b_j)$, $b_j$, $q_j$ are the same as $\tilde\phi_R(\tilde b_j)$, $\tilde b_j$, $\tilde q_j$
except when $j$ is such that $b_j = i$, in which case
it holds that $\tilde \phi_R(\tilde b_j) = \phi_R(b_j) -1$ and $\tilde b_j=b_j+1$ and $\tilde q_j=q_j$.

\item
If $w(i)  <  i < i+1< w(i+1)$ then $a_j=\tilde a_j$ and $p_j=\tilde p_j$ except when $a_j = i+1$,
in which case $\tilde a_j = a_j-1$ and $\tilde p_j = p_j+1$.
Likewise,
the numbers $\phi_R(b_j)$, $b_j$, $q_j$ are the same as $\tilde\phi_R(\tilde b_j)$, $\tilde b_j$, $\tilde q_j$,
except when $b_j = i$, in which case
$\tilde \phi_R(\tilde b_j) = \phi_R(b_j)-1$ and $\tilde b_j= b_j+1$ and $\tilde q_j=q_j$.
\end{itemize}
One of these cases must occur as since it cannot hold that $w(i+1) < w(i) < i < i+1$.
We deduce  that $i_j = \tilde i_j$ for all $j$ and $\psi=\tilde \psi$,
 so $\lambda_R(\theta \pi_i) = \lambda_R(\theta)$.
\end{proof}

\begin{corollary}
If $\theta \in \cW_n$ is right-admissible (respectively, left-admissible) 
then $\lambda_R(\theta \pi_g) = \lambda_R(\theta)$  (respectively, $\lambda_L(\pi_g \theta) = \lambda_L(\theta)$)
 for all $g \in \tS_n$.
\end{corollary}

\begin{proof}
This follows from the previous lemma and 
 Theorem~\ref{adm-thm}.
\end{proof}

\begin{corollary}\label{leftinverse-cor}
If $\theta \in \cWmin_{n}$ then $\lambda_R \( \omega_R(\theta)\) = \lambda_R(\theta) = \theta$
and
$\lambda_L\( \omega_L(\theta)\) = \lambda_L(\theta) = \theta$.
\end{corollary}

\begin{proof}
Apply the previous corollary after checking that 
any $\theta \in \cWmin_{n}$
has  $\lambda_R(\theta)=\lambda_L(\theta)=\theta$.
\end{proof}

\section{Order isomorphisms}\label{order-sect}

Recall that $t_{ij}$ for $i<j \not \equiv i \modu n)$
denotes the unique element of $\tS_n$ swapping $i$ and $i+1$ and fixing all $j \notin \{i,i+1\} + n\ZZ$. 
Write $u \lessdot v$ if $v$ covers $u$ in the Bruhat order on $\tS_n$,
i.e., if $v = ut_{ij}$ for some $i,j$ and $\ell(v) =\ell(u)+1$.
Let $\prec$ be the partial order on $\cWmin_{n}$ with $(w,\phi) \preceq (w',\phi')$
if and only if $w=w'$ and $\phi(a,b) \leq \phi'(a,b)$ for all $(a,b) \in \cC(w)$.

\begin{lemma}\label{main-lem}
Let $\theta\in \cWmin_{n}$ and $z,z' \in \tI_n$ with $\omega_R(\theta) = z$.
The following are then equivalent:
\ben
\item[(a)] There exist $i,j \in \ZZ$ with $z(i) < i<j = \min\{e \in \ZZ : i<e\text{ and }z(i) < z(e)\}$ and
$z' = t_{ij} z t_{ij}$.

\item[(b)] There exists $\theta' \in \cWmin_n$ which covers $\theta$ in the order $\prec$
such that $z' = \omega_R(\theta')$.

\een
Moreover, if these conditions hold then $g_R(\theta) \lessdot g_R(\theta) t_{ij} = g_R(\theta')$.
\end{lemma}

This lemma has a left-handed version, 
given by applying the $*$-operator and
 Lemma~\ref{*omega-lem}.

\begin{proof}
Let $w \in \tI_n$ and $\phi : \cC(w) \to \NN$ be such that $\theta = (w,\phi)$.
The theorem will derive from simple arguments involving the following specialised notation.
Consider an integer sequence $I = (i_1,i_2,\dots,i_l) \in \ZZ^l$.
For each $0\leq t \leq l$, define
$w^{I,t} \in \tI_n$ and $\phi^{I,t} : \cC(w^{I,t}) \to \NN$ such that $\theta \pi_{i_1}\pi_{i_2}\cdots \pi_{i_t} = (w^{I,t}, \phi^{I,t})$,
so that $(w,\phi) = (w^{I,0},\phi^{I,0})$.
As usual, write $\phi^{I,t}_R$ for the right form of $\phi^{I,t}$ as given after Definition~\ref{weighted-def}.
Let  $\alpha^{I,0} = 1 \in \tS_n$ and for $i \in [l]$, define $\alpha^{I,t} \in \tS_n$ to be $\alpha^{I,t}   s_{i_t}$ 
if $w^{I,t} \neq w^{I,t-1}$, and otherwise set  $\alpha^{I,t} = \alpha^{I,t-1}$.
If $\phi^{I,t}=0$ for any $t$ then 
we have $z = w^{I,t}$ and $g_R(\theta) = \alpha^{I,t}$. 
As an abbreviation, write $w^I = w^{I,l}$ and $\phi^I  = \phi^{I,l}$ and $\alpha^I = \alpha^{I,l}$.
We refer to $\alpha^{I,t}(x)$ as the \emph{trajectory} of $x \in \ZZ$ at time $t$, relative to the index sequence $I$.
By construction,
a given integer $x$ is a left endpoint, right endpoint, 
or fixed point of $w$ if and only if its trajectory at any time $t$
is respectively a left endpoint, right endpoint, of fixed point of $w^{I,t}$.
Moreover, if $X$ is the set of numbers $i \in \ZZ$ with $\phi_R^{I,t}(x) = 0$ for $x = \alpha^{I,t}(i)$,
then the trajectories of any two $i,j \in X$ have the same relative order at all times greater than or equal to $t$.

Choose a right endpoint $a \in \ZZ$ of $w$. 
Corresponding to this choice, we 
define a particular integer sequence $I$.
This sequence will be the concatenation of three subsequences $J$, $K$,  and $L$, given as follows.
If $\phi_R(a) = 0$ then define $J$ to be the empty sequence.
Otherwise, let $i_0$ be an arbitrary integer and define $J=(i_1,i_2,\dots,i_l)$ as the sequence
characterised by the following properties:
\begin{itemize}
\item For each $t\geq 0$, $i_{t+1}$ is the smallest integer $x>i_t$ 
such that $\phi^{J,t}_R(x-1) = 0 < \phi^{J,t}_R(x)$.
\item The first time $t$ at which $\phi^{J,t}(x)=0$ for $x=\alpha^{J,t}(a)$ is $t=l$.
\end{itemize}
Let $k = \alpha^{J}(a)$.
Consider the sequence given by the infinite repetition of $k+1,k+2,\dots,k+n-2$.
Define $K$ as the shortest initial subsequence of this sequence
 such that $\phi^{JK}(k+1) = 0$, where $JK$ denotes the concatenation of $J$ and $K$.
We then have $k = \alpha^{JK}(a)$.
 Define $b \in \ZZ$ as the integer with $k+1 = \alpha^{JK}(b)$.
 If $w^{JK}(k+1) < w^{JK}(k)$, then we must also have 
 $w^{J,l-1}(\alpha^{J,l-1}(b)) < w^{J,l-1}(\alpha^{J,l-1}(a))$ and
 $\phi^{J,l-1}(\alpha^{J,l-1}(b))=0 < \phi^{J,l-1}(\alpha^{J,l-1}(a))=1$.
 But this is impossible, since the way we have constructed $J$ means that
 the weighted involution $(w^{J,t},\phi^{J,t})$ never has nesting cycles in 
 which the outer cycle has weight zero while the inner cycle 
 has a positive weight.
 We conclude, therefore, that $w^{JK}(k) < w^{JK}(k+1)$. 
 Finally, let $L$ be any sequence such that $\phi^{I} = 0$ and $z=\omega_R(\theta)= w^{I}$, where $I = JKL$.
Define $i = \alpha^I(a)$ and $j = \alpha^I(b)$; then  $z(i)<i<j$ and $z(i) < z(j)$.
 
 Since $\theta$ is right-admissible and since the right operators $\pi_1,\pi_2,\dots$ preserve right-admissibility,
each $t \in \{i+1,i+2,\dots,j-1\}$ must be a left endpoint of $z$ with $z(t)<z(i)$.
Thus $j = \min \{e \in \ZZ: i<j\text{ and }z(j)<z(i)\}$. Moreover, it is easy to see that any pair $(i,j) \in \ZZ$
with $j = \min \{e \in \ZZ: i<j\text{ and }z(j)<z(i)\}$ arises in  this way for some choice of $a$: 
given $i$, the desired value of $a$
is $w(z(i)+p)$ where $p$ is the number of pairs $(x,y) \in \cC(z)$ with $x<z(i)< y$.

Now consider the unique weighted involution $\theta' = (w,\phi') \in \cW_n$
which covers $\theta$ is the order $\prec$ and has $\phi'_R(a) = \phi_R(a)+1$.
Define $z' \in \tI_n$ such that $\omega_R(\theta')=z'$.
 If we define $K'$ by adding the index $k$ to the end of $K$ and set $I' = JK'L$,
 then it holds by construction that  $z=\omega_R(\theta') = w^{I'}$
and $g_L(\theta') = \alpha^{I'} = gsg^{-1}h$ for $g = \alpha^J$, $s=s_k$, and $h = \alpha^I=g_R(\theta)$.
Since $g^{-1}h(k) = i$ and $g^{-1}h(k+1) = j$,
it follows that $g_L(\theta') = h(g^{-1}h)^{-1} s (g^{-1}h) = g_R(\theta) t_{ij}$,
and therefore
$z' = t_{ij}zt_{ij}$. This suffices to show the equivalence of (a) and (b).
 For the last assertion, note 
 that it
 follows from Theorem-Definition~\ref{omega-thm} and Corollary~\ref{limit-cor}
 that $\ell(g_R(\theta')) = \ell(g_R(\theta))+1$,
 so 
 $g_R(\theta) \lessdot g_R(\theta) t_{ij}$.
\end{proof}

Define $\prec_R$ as the transitive closure of the relation
on $\tI_n$  
with $z \prec_R  t_{ij} z t_{ij}$
whenever $z(i) < i$ and $j = \min \{e \in \ZZ : i < e \text{ and }z(i) < z(e)\}$.
Define 
$\prec_L$ similarly as the transitive closure of the relation on $\tI_n$
with 
$z \prec_L  t_{ij} z t_{ij}$
whenever $j < z(j)$ and  $i = \max\{ e \in \ZZ : e <j \text{ and }z(e)<z(j)\}$.
A partially ordered set is \emph{graded} if 
 all maximal chains between two elements have the same length.
 A \emph{rank function} for a graded poset $\cP$ is a map $\cP \to \ZZ$,
the difference of whose values gives the common length of 
all maximal chains between two comparable elements.

\begin{proposition}\label{prec-prop}
The posets $(\tI_{n}, \prec_R)$ and $(\tI_{n}, \prec_L)$ are 
isomorphic via the map $z \mapsto z^*$.
Both are graded 
 subposets  of $(\tI_{n}, <)$
with
rank function $z \mapsto \frac{1}{2}\ell(z)$.
\end{proposition}

\begin{proof}
The first assertion is clear from the definitions.
For the second claim, note by Lemma~\ref{conj-cover-lem} that if $z$ covers $y$ in $\prec_R$ or $\prec_L$
then $y < z$ and $\ell(z) = \ell(y) + 2$.
\end{proof}

\begin{lemma}\label{min-lem}
An element $ z\in \tI_n$ is minimal in $ \prec_R$
(respectively, $\prec_L$) if and only if $\ell(z) = \ell'(z)$.
\end{lemma}

\begin{proof}
We may just consider $(\tI_n, \prec_R)$ since the other poset is isomorphic.
Assume $z \in \tI_n$ is minimal with respect to $\prec_R$. 
We cannot have $ z(i+1)<i+1$ and $z(i+1)<z(i)$ for any $i \in \ZZ$
since then $s_i zs_i \prec_R z$.
This implies that if $i+1$ is a right endpoint of $z$ 
then either $i$ is also a right endpoint with $z(i) <z(i+1)$,
 or $z(i) = i+1$. Applying this property in succession, 
 we deduce that in fact $z(i)=i+1$ whenever $i+1$ is a right endpoint of $z$,
so $\ell(z) = \ell'(z)$. Conversely, any involution $z \in \tI_n$ with $\ell(z) = \ell(z')$ is evidently minimal
with respect to $\prec_R$.
\end{proof}

Putting everything together gives the following.

\begin{theorem}\label{main-thm}
The maps 
$
 \omega_R : (\cWmin_{n}, {\prec}) \to (\tI_{n},{\prec_R})
 $
and
$ \omega_L : (\cWmin_{n}, {\prec}) \to (\tI_{n},{\prec_L})
$
are isomorphisms of partially ordered sets which preserve $\ell$ and $\ell'$,
 and have respective inverses $\lambda_R$ and $\lambda_L$.
 \end{theorem}
 
 \begin{proof}
By Corollaries~\ref{limit-cor} and \ref{leftinverse-cor}, 
the maps $\omega_R$ and $\omega_L$ 
preserve $\ell$ and $\ell'$,
are injective 
with left inverses $\lambda_R$ and $\lambda_L$,
and restrict to bijections $\{ \theta \in \cWmin_{n} : \wt(\theta)=0\}\to\{w \in \tI_n : \ell(w) = \ell'(w)\}$.
Given these facts, the maps' surjectivity follows
 from Lemmas~\ref{main-lem} and \ref{min-lem} by induction.
\end{proof}

Slightly abusing notation, we define $\tI_n(q,x)= \sum_{w \in \tI_n} q^{\ell(w)} x^{\ell'(w)}\in \NN[[q,x]]$.

 \begin{corollary}\label{A-cor} If $n\geq 1$ then
 $\ds\tI_n(q,x)= \sum_{k=0}^{\lfloor n/2 \rfloor} \frac{n}{n-k} \binom{n-k}{k} \( \frac{qx}{1-q^2}\)^k.$
\end{corollary}

\begin{proof}
Corollary~\ref{limit-cor} and Theorem~\ref{main-thm} imply
that $\tI_n(q,x)= \sum_{\theta \in \cWmin_n} q^{2\wt(\theta)} (qx)^{\ell'(\theta)}$. 
By Proposition~\ref{match-prop},
the coefficient of $x^k$ in
the latter power series is $\frac{n}{n-k}\binom{n-k}{k} q^k (1 + q^2 + q^4 + q^6+\dots)^k$. 
\end{proof}

\begin{corollary}\label{Ab-cor}
If $n\geq 3$ then
$\tI_n(q,x)= \tI_{n-1}(q,x) + \tfrac{qx}{1-q^2} \tI_{n-2}(q,x)$.
\end{corollary}

 The bivariate \emph{Lucas polynomials} $\Luc_n(q,x)$, introduced in \cite{BeHo}, are 
defined by the recurrence
$\Luc_n(x,s) = x \Luc_{n-1}(x,s) + s \Luc_{n-2}(x,s)$ with $\Luc_0(x,s) = 2$ and $\Luc_1(x,s) = x$.

\begin{corollary}\label{AA-cor}
If $n\geq 1$ then 
\[\tI_n(q,x) = \tfrac{1}{(1+q)^n} \Luc_n{\(1+q,\tfrac{1+q}{1-q}qx\)}
\qquand \Luc_n(x,s)=x^n\tI_n{\(x-1,\tfrac{2-x}{x(x-1)} s\)}.\]

\end{corollary}

\begin{proof}
Let $f_n = \tfrac{1}{(1+q)^n} \Luc_n{\(1+q,qx\frac{1+q}{1-q}\)}$.
Then $f_1 = 1 =\tI_1(q,x)$ and $f_2=1 + \frac{2qx}{1-q^2} = \tI_2(q,x)$ and 
if  $n\geq 3$ then
$f_n= f_{n-1} + \tfrac{qx}{1-q^2}f_{n-2}$.
The result therefore follows from Corollary~\ref{Ab-cor}.
\end{proof}

Define $\ellhat(w) = \frac{1}{2}(\ell(w) + \ell'(w))$ for $w \in \tI_n$. 
Corollary~\ref{A-cor} shows that  $\ellhat(w) \in \NN$.

\begin{corollary}\label{counts-cor}
 For each $n\geq 2$ and $m\geq 1$,
let $N_{n}(m)$ and $\hat N_{n}(m)$ be the number of involutions $w \in \tI_n$
with $\ell(w) = m$ and $\ellhat(w) = m$, respectively. Then 
\[
N_{n}(m) = 
\sum_{\substack{ 1 \leq j \leq \lfloor n/2 \rfloor \\ j \equiv m \modu 2)}} \frac{n}{n-j} \binom{n-j}{j} \binom{\frac{j+m}{2} -1 }{j-1}
\quand
\hat N_{n}(m) = \sum_{j=1}^{\lfloor n/2\rfloor} \frac{n}{n-j} \binom{n-j}{j} \binom{m-1}{j-1}
.\]
\end{corollary}

\begin{proof}
Extract the coefficients of $q^m$ in $\tI_n(q,1)$ and $q^{2m}$ in $\tI_n(q,q)=\sum_{w \in \tI_n} q^{2\ellhat(w)}$.
\end{proof}

\begin{remark}
The sequence $\{ \hat N_n(n) \}_{n=1,2,3,\dots} = (0, 2, 3, 10, 25, 71, 196, 554, 1569, \dots)$ coincides with \cite[A246437]{OEIS}, which gives the ``type $B$ analog for Motzkin sums,''
while $\{ \hat N_n(2n) \}_{n=1,2,3,\dots} = (0, 2, 3, 18, 50, 215, 735, 2898, \dots)$ gives \cite[A211867]{OEIS}, which counts certain Motzkin paths. 
\end{remark}

Define $\zeta_R,\zeta_L: \tI_n \to \tI_n$ as the maps
$
\zeta_R = {*} \circ \omega_L \circ \lambda_R = \omega_R \circ {*} \circ \lambda_R = \omega_R \circ \lambda_L \circ {*}
$
and
$
 \zeta_L = {*} \circ \omega_R \circ \lambda_L = \omega_L \circ {*} \circ \lambda_L = \omega_L \circ \lambda_R \circ {*}
$.
Corollary~\ref{limit-cor} and Theorem~\ref{main-thm} imply this property:

\begin{corollary}\label{zeta-cor}
The maps $\zeta_R$ and $\zeta_L$ are involutions of $(\tI_n,\prec_R)$ and $(\tI_n,\prec_L)$, respectively, which preserve length $\ell$ and absolute length $\ell'$,
and it holds that $\zeta_L \circ {*} ={*}\circ \zeta_R$.
\end{corollary}



Consider the following variation of the permutations $g_L(\theta)$ and $g_R(\theta)$ for $\theta \in \cW_n$.

\begin{definition}\label{alpha-def}
For each $z \in \tI_n$ let
 $\alpha_R(z) = g_R(\lambda_R(z)) \in \tS_n$ and  $\alpha_L(z) = g_L(\lambda_L(z)) \in \tS_n$.
 \end{definition}

As our final result in this section, we derive a more explicit formula for these elements.

\begin{propdef}\label{window-def}
If $a_1,a_2,\dots,a_n \in\ZZ$ represent the distinct congruence classes modulo $n$
then there is a unique  $m \in \ZZ$ and a unique 
$w \in \tS_n$ such that $w(m+i) = a_i$ for $i \in [n]$.
Moreover, it holds that $m = \frac{1}{n}\sum_{i=1}^n (a_i-i)$.
Define $\lW a_1,a_2,\dots,a_n\rW  = w\in \tS_n$.
\end{propdef}

\begin{proof}
If an affine permutation exists with the desired property,
it must be given by the map $w : \ZZ \to \ZZ$ 
with $w(m+i+jn) = a_i + jn$ for $i \in [n]$ and $j \in \ZZ$.
As this map is in $\tS_n$ if and only if 
$\sum_{i \in [n]} i  = \sum_{i \in [n]} w(i)  =  \sum_{i \in [n]} w(m+i) - nm=  \sum_{i \in [n]} a_i- nm $,
the result follows.
\end{proof}

If $a_1,a_2,\dots,a_N \in \ZZ$ represent all congruence classes modulo $n$,
and $i_1<i_2<\dots<i_n$ are the indices of the first representative of each class,
then we define $\lW a_1,a_2,\dots,a_N\rW = \lW a_{i_1},a_{i_2},\dots,a_{i_n}\rW  \in \tS_n$.
For example, if $n=3$ then $[1,0,1,3,8,4,2] = [ 1,0,8]$.

\begin{theorem}\label{alphaL-thm}
Let $z \in \tI_n$ and $m \in \ZZ$. Suppose 
$a_1 < a_2 <\dots <a_l
$
and
$ 
d_1 < d_2 < \dots < d_l
$ are the elements of $m+[n]$
with $a_i \leq z(a_i)  $ and $z(d_i) \leq d_i$. Define  $b_i = z(a_i)$ and $c_i = z(d_i)$.
Then 
\[\alpha_R(z) = \lW a_1,b_1,a_2,b_2,\dots,a_l,b_l\rW ^{-1}
\qquand
\alpha_L(z) = \lW c_1,d_1,c_2,d_2,\dots,c_l,d_l\rW ^{-1}.
\]
In particular, these formulas do not depend on the choice of $m$.
\end{theorem}

The number $l$ is necessarily  $n - \ell'(z)$.
Before the proof, we give an example.
\begin{example}
One has
$\alpha_R(1) = \alpha_L(1) = \lW 1,1,2,2,\dots,n,n\rW ^{-1} = 1$.
If $z =t_{1,8}t_{2,7} \in \tI_4$ then 
$\alpha_R(z) = \lW 1,8,2,7\rW ^{-1} =\lW 3,5,2,0\rW  $ and $ \alpha_L(z) = \lW -2,3,-3,4\rW ^{-1}=\lW 5,3,0,2\rW .$
\end{example}

\begin{proof}
By Lemma~\ref{*omega-lem}, we may 
just prove the assertions for $g_R(z)$.
Replacing $m$ by $m+1$ has no effect on the formula for $g_R(z)$ since 
$\lW a_2,b_2,\dots,a_l,b_l,a_1+n,b_1+n\rW  = \lW a_1,b_1,a_2,b_2,\dots,a_l,b_l\rW $,
and so we deduce that this formula is independent of the choice of $m$.
When $\ell(z) =\ell'(z)$,
we have $ \lW a_1,b_1,\dots,a_l,b_l\rW ^{-1} = \lW 1,2,\dots,n\rW ^{-1}=1 = g_R(z).$
For the general case, suppose $z' \in \tI_n$ covers $z$ in $\prec_R$, so that
$z' = t_{ij} z t_{ij} $ for some $i,j \in \ZZ$ with $z(i) < i$ and $j = \min\{ e \in \ZZ : i < e : z(i) < z(e)\}$.
Let $m = z(i)-1$, so that $a_1 = z(i)$ and $b_1 = i$,
and assume by induction that
$g_R(z) = \lW a_1,b_1,a_2,b_2,\dots,a_l,b_l\rW ^{-1}$. 
The numbers $i,i+1,\dots,j-1$ are necessarily all right endpoints of $z$,
and consequently $i<j<i+n$.
Define $\Delta = j-i \in [n-1]$.
The only way $\{i,z(i)\} + n\ZZ$ and $\{j,z(j)\} + n\ZZ$ can intersect is if $j<z(j) \equiv i \modu n)$.
We now compute the product $g_R(z) t_{ij}$.
There are four cases to consider:
\begin{itemize}
\item[(a)] Suppose $z(j) < j$.
Let $t \in \{2,3,\dots,l\}$ be the unique index with 
$a_t \equiv z(j) \modu n)$ and $b_t \equiv j \modu n)$.
Then $g_R(z) t_{ij} = \lW a_1,b_1+\Delta,a_2,b_2,\dots,a_{t-1},b_{t-1},a_t,b_t - \Delta,a_{t+1},b_{t+1},\dots\rW ^{-1}$.

\item[(b)] Suppose $z(j) = j$. 
Let $t \in \{2,3,\dots,l\}$ be the unique index with $a_t=b_t \equiv j \modu n)$.
Then $g_R(z) t_{ij} = \lW a_1,b_1+\Delta,a_2,b_2,\dots,a_{t-1},b_{t-1},a_t-\Delta,a_t-\Delta,a_{t+1},b_{t+1},\dots\rW ^{-1}$.

\item[(c)] Suppose $j < z(j) \not \equiv i \modu n)$. 
Let $t \in \{2,3,\dots,l\}$ be the unique index with $a_t \equiv j \modu n)$.
Then $g_R(z) t_{ij} = \lW a_1,b_1+\Delta,a_2,b_2,\dots,a_{t-1},b_{t-1},a_t-\Delta,b_t,a_{t+1},b_{t+1},\dots\rW ^{-1}$.

\item[(d)] Finally suppose $j < z(j)  \equiv i \modu n)$. Then $z(j) \equiv a_1 \modu n)$ and $j \equiv b_1 \modu n)$,
and
we have $g_R(z) t_{ij} = \lW a_1-\Delta,b_1+\Delta,a_2,b_2,\dots,a_l,b_l\rW ^{-1}$.

\end{itemize}
Let $a_1'<a_2'<\dots<a_l'$ be the numbers in $m+[n]$ with $a_i'\leq z'(a_i')$  and set $b_i' = z'(a_i')$.
In each case (a)-(d), our computations reduce to 
the identity $g_R(z) t_{ij} = \lW a_1',b_1',\dots,a_l',b_l'\rW ^{-1}$.
For example, in case (d) 
we have $a_i' = a_{i+1}$ and $b_i'=b_{i+1}$ for $i \in [l-1]$ 
while $a_l' = a_1 - \Delta  + n$ and $b_l' = a_1+\Delta+n$, so
$g_R(z) t_{ij} = \lW a_l'-n,b_l'-n,a_2',b_2',\dots,a'_{l-1},b'_{l-1}\rW ^{-1} = \lW a_1',b_1',\dots,a_l',b_l'\rW ^{-1}.$
Since $g_R(z') = g_R(z) t_{ij}$
by Lemma~\ref{main-lem},
the desired formula for $g_R(z)$ holds for all $z \in \tI_n$ by induction.
\end{proof}

\section{Demazure conjugation}\label{dem-sect}

Recall the definition of the {Demazure product} $\circ : \tS_n \times \tS_n \to \tS_n$ from the introduction.
The operation $(z,w) \mapsto w^{-1}\circ z \circ w$ for $z \in \tI_n$ and $w \in \tS_n$
defines another right action of the monoid $(\tS_n, \circ)$, now on  $\tI_n$.
We refer to this action as \emph{Demazure conjugation}.
If $z \in \tI_n$ and $i \in \ZZ$ then 
\be\label{dem-eq} s_i \circ z \circ s_i = \begin{cases} s_izs_i &\text{if }z(i)<z(i+1)\text{ and } zs_i \neq s_i z \\
zs_i &\text{if }z(i)<z(i+1) \text{ and }zs_i = s_i z \\
z &\text{if } z(i) > z(i+1)
\end{cases}
\ee
by \cite[Corollary 2.2]{HMP2}. If $z \in \tI_n$ and $z(i) < z(i+1)$, then $zs_i = s_i z$ if and only if $i$ and $i+1$ are both fixed points of $z$.
It follows by induction from \eqref{dem-eq} that the orbit of $1$ 
under Demazure conjugation is all of $\tI_n$, so
every $z \in \tI_n$ can be expressed as
$z = w^{-1}\circ w $ for some $w \in \tS_n$.

\begin{definition}\label{ca-def}
For $z \in \tI_n$ let
 $\cA(z)$ be the set of shortest elements $w \in \tS_n$ with $z=w^{-1}\circ w$.
 \end{definition}
 
 \begin{example}
 If $z = t_{0,5} = \lW -4,2,3,9\rW  \in \tI_4$ then $\cA(z) = \{ s_1s_2s_3s_4, s_2s_1s_3s_4, s_3s_2s_1s_4\}$.
 \end{example}

 The set $\cA(z)$ is nonempty for all $z \in \tI_n$, and we refer to its elements as the \emph{atoms} of $z$.
From Corollary~\ref{counts-cor}, we have  $\ellhat(z) = \frac{1}{2}(\ell(z)  + \ell'(z))$.
By \cite[Proposition 2.6]{MW} or \eqref{dem-eq}, we obtain the following:

 \begin{proposition}\label{ellhat-prop}
If $z \in \tI_n$ then $\ellhat(z)$ is the common value of $\ell(w)$ for $w \in \cA(z)$.
 \end{proposition}
 

Results in \cite{HMP2}, building on work of Can, Joyce, and Wyser \cite{CJ,CJW}, show that the set of atoms of an involution in any finite symmetric group naturally has the structure of a bounded, graded poset.
Here, we prove that this phenomenon generalises to involutions in the infinite group $\tS_n$.

Let $\lessdot_\cA$ be the relation on $\tS_n$ with 
 $u \lessdot_\cA v$ if and only if $u < s_{i+1}u = s_{i}v > v$ for some $i \in \ZZ$.
This relation is empty when $n \in \{1,2\}$.
Given $w \in \tS_n$, write $w[j:k]$ for the sequence of values $w(j)w(j+1)\cdots w(k)$. 
We note two additional characterisations of $\lessdot_\cA$:

\begin{lemma}\label{231-lem}
Let $u,v \in \tS_n$. Then $u\lessdot_\cA v$ if and only if for some $a,b,c,i \in \ZZ$ with $a<b<c$
we have  $u^{-1}[i:i+2]=cab$, $v^{-1}[i:i+2]= bca$,
and $u^{-1}(j) = v^{-1}(j)$ for all  $j  \notin \{i,i+1,i+2\} + n\ZZ$.
\end{lemma}

\begin{proof}
This follows as a simple exercise from Corollary~\ref{des-ell-cor}.
\end{proof}

\begin{lemma}\label{gp-lem}
Let $u,v \in \tS_n$. Then $u\lessdot_\cA v$ if and only if  
for some $i \in \ZZ$ and $w \in \tS_n$ 
 it holds that
$u=s_{i}s_{i+1}w$, $v= s_{i+1}s_i w$, and $\ell(w) +2= \ell(u)=\ell(v)$.
\end{lemma}

\begin{proof}
This is easy to derive from the preceding lemma given Lemma~\ref{bruhat0-lem}. 
\end{proof}

\begin{corollary}\label{equiv-cor}
Let $z \in \tI_n$ and $u,v \in \tS_n$ with $u \lessdot_\cA v$. Then $u \in \cA(z)$ if and only if $v \in \cA(z)$.
\end{corollary}

\begin{proof}
This follows from Lemma~\ref{gp-lem} since 
$\circ$ is associative and $s_{i+1}s_i \circ s_is_{i+1} = 
s_is_{i+1}\circ s_{i+1}s_i$.
\end{proof}

Recall  the maps $\lambda_R,\lambda_L : \tI_n \to \cW_n$ and
the elements $\alpha_R(z), \alpha_L(z)$
from Definitions~\ref{omega-thm}
and \ref{alpha-def}.

\begin{definition}\label{alpha-min-def}
Given $z \in \tI_n$, write $ (w_R,\phi_R) = \lambda_R(z) $
and $(w_L,\phi_L) = \lambda_L(z) $ and
define
\[
\alpha_{\min}(z) = w_R \cdot \alpha_R(z)= \alpha_R(z) \cdot z
\qquand
\alpha_{\max}(z) = w_L \cdot \alpha_L(z) = \alpha_L(z)\cdot z.
\]
\end{definition}

\begin{proposition}\label{min-prop}
For each $z \in \tI_n$ both $\alpha_{\min}(z)$ and $\alpha_{\max}(z)$ belong to $\cA(z)$.
\end{proposition}

\begin{proof}
Let $g = \alpha_R(z)$ and $w=w_R$.
By Theorem-Definition~\ref{omega-thm}(b) and Corollary~\ref{limit-cor}
we have $\ell(z) =\ell\(g^{-1}  wg\) = \ell(w) + 2\ell(g)$,
so $z = g^{-1} w g = g^{-1} \circ w \circ g$. Since $\ell(w) =\ell'(w)$,  Lemma~\ref{commuting-product-lem}
implies
that $w=w^{-1}\circ w$,  and therefore $z = g^{-1} \circ w^{-1} \circ w \circ g = (wg)^{-1} \circ wg$.
As $\ell(wg) = \ell(w) + \ell(g) = \ellhat(z)$,
we conclude that $wg \in \cA(z)$. The argument that $w_L \cdot \alpha_L(z) \in \cA(z)$ is similar.
\end{proof}

The atoms $\alpha_{\min}(z)$ and $\alpha_{\max}(z)$ have this convenient formula:

\begin{corollary}\label{min-cor}
Let $z \in \tI_n$, $m \in \ZZ$, and $l = n-\ell'(z)$. 
Define $a_i,b_i,c_i,d_i \in \ZZ$ for $i \in [l]$
as in Theorem~\ref{alphaL-thm}. Then
$ \alpha_{\min}(z) = \lW  b_1,a_1,b_2,a_2,\dots,b_l,a_l\rW ^{-1}
$
and
$
\alpha_{\max}(z) = \lW d_1,c_1,d_2,c_2,\dots,d_l,c_l\rW ^{-1}.$
\end{corollary}

\begin{proof}
 As $z\lW  a_1,b_1,\dots,a_l,b_l\rW  = \lW  b_1,a_1,\dots,b_l,a_l\rW $
and $z\lW  c_1,d_1,\dots,c_l,d_l\rW  = \lW  d_1,c_1,\dots,d_l,c_l\rW $
by construction,
the result follows from  Theorem~\ref{alphaL-thm}.
\end{proof}

\begin{example}\label{alpha-min-ex}
If $z = t_{1,8}t_{2,7} \in \tI_4$ then 
$\alpha_{\min}(z) = \lW 8,1,7,2\rW ^{-1} =\lW 4,6,1,-1\rW  $ and $ \alpha_{\max}(z) = \lW 3,-2,4,-3\rW ^{-1}=\lW 6,4,-1,1\rW .$
\end{example}

\begin{lemma}\label{321-lem}
Let $z \in \tI_n$ and $w \in \cA(z)$.
Then no $i \in \ZZ$  has 
$w^{-1}[i:i+2] = cba$ where $a<b<c$.
\end{lemma}

\begin{proof}
If $w^{-1}(i)>w^{-1}(i+1)>w^{-1}(i+2)$ for $i \in\ZZ$
then we could write $w = s_is_{i+1}s_iv$ for some $v \in \tS_n$ with $\ell(w) = \ell(v) + 3$. 
But then we would have $z = w^{-1} \circ w =  (s_is_{i+1}s_iv)^{-1} \circ (s_is_{i+1}s_i  v)
= (s_{i+1}s_iv)^{-1} \circ (s_{i+1} s_i v)$
despite $\ell(s_{i+1}s_iv) = \ell(v)+2 < \ell(w)$, contradicting the definition of $\cA(w)$.
\end{proof}

\begin{lemma}\label{extr-lem}
Let $z \in \tI_n$ and $v \in \cA(z)$.
\ben
\item[(a)]
If no $u \in \cA(z)$ exists with $u \lessdot_\cA v$ then $v = \alpha_{\min}(z)$.

\item[(b)]
If no $w \in \cA(z)$ exists with $v \lessdot_\cA w$ then $v = \alpha_{\max}(z)$.
\een
\end{lemma}

\begin{proof}
Assume no $u \in \cA(z)$ exists with $u\lessdot_\cA v$.
Lemmas~\ref{231-lem}  and \ref{321-lem} then imply that
if $v^{-1}(i+1) > v^{-1}(i+2)$ for some $i \in \ZZ$ then $v^{-1}(i)<v^{-1}(i+2)$.
Therefore $v = \lW b_1,a_1,b_2,a_2,\dots,b_l,a_l\rW ^{-1}$
for some numbers with $a_1<a_2<\dots<a_l$, $a_i\leq b_i$ for all $i$, and
 $b_i \neq b_j$ for all $i\neq j$.
Thus $v=\alpha_{\min}(y)$ 
for the involution $y \in \tI_n$ with $y(a_i) =b_i$ for $i \in [l]$ by Corollary~\ref{min-cor}.
Since $\alpha_{\min}(y) \in \cA(y)$ we must have $y=z$.
This proves part (a).
The proof of part (b) is similar, proceeding from the observation if no
$w \in \cA(z)$ exists with $v \lessdot_\cA w$
then $v^{-1}(i) > v^{-1}(i+1)$ implies $ v^{-1}(i) < v^{-1}(i+2)$.
\end{proof}

Fix $z \in \tI_n$ and $w \in \tS_n$. Define two sets
\[\uL_z = \{ a \in \ZZ : a \leq z(a)\}
\quand \uR_z = \{ b \in \ZZ : z(b) < b\}.\]
Let $\inv_\cA(w;z)$ be the set of pairs $(p,q) \in \ZZ\times \ZZ$
with $w(p)>w(q)$ and either
$p<q$ and $\{ p, q\}  \subset \uL_z$,
or
$p>q$ and  $\{ p, q\} \subset\uR_z$.
 The set $ \inv_\cA(w;z)$ is 
closed under the relation generated by $(i,j) \sim(i+n,j+n)$. 
Define $\rank_\cA(w;z)$ as the number of equivalence classes under this relation  (which a priori could be infinite) in
the set difference  $\inv_\cA(w;z) \setminus \inv_{\cA}\(\alpha_{\min}(z);z\)$.

\begin{remark}
All of our arguments would go through if we redefined $\uL_z$ to be $\{a \in \ZZ : a<z(a)\}$
and $\uR_z$ to be $\{ b \in \ZZ  : z(b) \leq b\}$. It makes no difference which set contains the fixed points of $z$.
\end{remark}

\begin{example}
Suppose $z = t_{1,8} t_{2,7} \in \tI_4$ so  that 
$\alpha_{\min}(z) = \lW 4,6,1,-1\rW $ and  $w=\lW 6,4,-1,1\rW  \in \cA(z)$
by Example~\ref{alpha-min-ex}. Then $\uL_z = \{1,2\} + 4\ZZ$ and $\uR_z = \{3,4\} + 4\ZZ$,
so we have
\[ \inv_\cA(\alpha_{\min}(z);z) = \left\{ (i+mn,j+ln) : i,j \in \{3,4\} \text{ and }l<m\right\}\]
and
$ \inv_\cA(w;z) = \{ (1+mn, 2+mn) : m \in \ZZ\} \sqcup \{ (4+mn, 3+mn) : m \in \ZZ\} \sqcup\inv_\cA(\alpha_{\min}(z);z) .$
Therefore $\rank_\cA(w;z) = 2$.
\end{example}

  Finally, let $<_\cA$ be the transitive closure of $\lessdot_\cA$.
The following generalises \cite[Theorem 6.10]{HMP2}.

\begin{theorem}\label{poset-thm}
Let $z \in \tI_n$. Restricted to the set $\cA(z)$,
the relation $<_\cA$ is a bounded, graded partial order
with rank function $w\mapsto \rank_\cA(w ;z)$.
Furthemore, it holds that  
\be\label{a-form}\cA(z) = \left\{ w \in \tS_n : \alpha_{\min}(z) \leq_\cA w\right\} = \left\{ w \in \tS_n : w\leq_\cA \alpha_{\max}(z) \right\}.
\ee
\end{theorem}

\begin{remark}
The situation described by this theorem has some formal similarities to Stembridge's results in \cite[\S4]{StemRed} about the top and bottom classes 
of a permutation in $S_{n+1}$.
\end{remark}

\begin{remark}
Let $\sim_\cA$ be the symmetric closure of $\leq_\cA$. 
The theorem implies that each set $\cA(z)$ for $z \in \tI_n$ is an equivalence class under 
$\sim_\cA$. This weaker statement is equivalent to the main results of \cite{HH,M} specialised to type $\tilde A$;
see \cite[Corollary 1.6]{M} or \cite[Theorem 1.2]{HH}.
\end{remark}

\begin{proof}
Fix $v \in \cA(z)$ and 
let $\rank(v) = \rank_\cA\(v;z\)$.
Since $\cA(z)$ is a  subset of the finite set of elements $w \in \tS_n$ with $\ell(w) = \ellhat(z)$, 
 Lemma~\ref{extr-lem} implies that $\alpha_{\min}(z) \leq_\cA v \leq_\cA \alpha_{\max}(z)$.
Equation \eqref{a-form} therefore follows from Corollary~\ref{equiv-cor}, though it remains to show that $<_\cA$ is a partial order.

To prove this, 
we argue (i) that $\rank(v)$ is finite, (ii) that $\inv_\cA(\alpha_{\min}(z))\subset \inv_\cA(v)$, 
and (iii) that $(a,c) \in \cC(z)$ whenever $c=u^{-1}(j)> u^{-1}(j+1)=a$ for some $j \in \ZZ$.
If $v = \alpha_{\min}(z)$ then these properties hold by Corollary~\ref{min-cor}.
Otherwise, there exists an atom $u \in \cA(z)$ with $u \lessdot_\cA v$ by 
Lemma~\ref{extr-lem}.
We may assume by induction that properties (i)-(iii) hold for $u$.
By Lemma~\ref{231-lem}, we know that
$u^{-1}[i:i+2] = cab$ and $v^{-1}[i:i+2] = bca$ where $a<b<c$ for some $i \in \ZZ$.
Property (iii) implies that $(a,c) \in \cC(z)$, so 
$\inv_\cA(v;z)$ is either 
\[\inv_\cA(u;z)  \sqcup \{ (a+mn, b+mn) : m \in \ZZ\}
\quord
\inv_\cA(u;z)  \sqcup \{ (c+mn, b+mn) : m \in \ZZ\}
\]
according to whether $b\leq z(b)$ or $z(b)<b$.
We conclude from (i) and (ii) that 
$\rank(v)=\rank(u)+1<\infty$ and $\inv_\cA(\alpha_{\min}(z))\subset \inv_\cA(v)$.
Finally suppose $v^{-1}(j) > v^{-1}(j+1)$ for some $j \in \ZZ$.
If $u^{-1}[j:j+1]=v^{-1}[j:j+1]$ then 
$(v^{-1}(j+1),v^{-1}(j)) \in \cC(z)$
 by
property (iii) for $u$.
Otherwise,  Lemma~\ref{231-lem} implies that 
for some integers  $a<b<c$ we have one of the following:
\begin{itemize}
\item   $u^{-1}[j+1:j+3]= cab$ and 
$v^{-1}[j+1:j+3] = bca$ and $u^{-1}(j)=v^{-1}(j)$.

\item
$u^{-1}[j-2:j] = cab$ and 
$v^{-1}[j-2:j] = bca$ and $u^{-1}(j+1)=v^{-1}(j+1)$.

\item $u^{-1}[j-1:j+1] = cab$ and $v^{-1}[j-1:j+1]= bca$.
\end{itemize}
The first two cases cannot occur since they contradict Lemma~\ref{321-lem},
while in the third case we have $(v^{-1}(j+1),v^{-1}(j)) = (a,c) = (u^{-1}(j),u^{-1}(j-1)) \in \cC(z)$
by induction.
We conclude that the desired properties hold for all $v \in \cA(z)$.

Since all atoms have properties (i)-(iii), we can repeat the argument in the preceding paragraph to deduce that 
 $\rank(v) = \rank(u)+1<\infty$
whenever $u,v \in \cA(z)$ have $u \lessdot_\cA v$.
This 
 shows precisely that $<_\cA$ restricted to $\cA(z)$ is a graded partial order
with rank function $w\mapsto \rank(w)$.
\end{proof}

\begin{corollary}\label{course-cor}
Let $z \in \tI_n$ and $w \in \cA(z)$. 
\ben
\item[(a)] If $b=w^{-1}(i) > w^{-1}(i+1)=a$ for some $i \in \ZZ$ then $(a,b) \in \cC(z)$.
\item[(b)] It holds that $\inv_\cA(\alpha_{\min}(z);z) \subset \inv_\cA (w;z) \subset \inv_\cA(\alpha_{\max}(z);z)$.
\item[(c)] If $(a,b) \in \cC(z)$ then $w(b) < w(a)$.
\een
\end{corollary}

\begin{proof}
Parts (a) and (b) were shown in the proof of Theorem~\ref{poset-thm}.
It follows from part (a) that if part (c) holds for $w$ then the same property holds for any $w ' \in \tS_n$
with $w \lessdot_\cA w'$. Since  (c)  is true for $w = \alpha_{\min}(z)$,
we conclude from Theorem~\ref{poset-thm} that this property holds for all atoms of $z$.
\end{proof}

\begin{figure}[h]
{\tiny\begin{center}
\begin{tikzpicture}
  \node (min) at (0,-10) {$\lW 4, 6, 8, 7, -1, -3\rW $};
     \node (i1) at  (-2,-9) {$\lW 4, 7, 8, 5, 0, -3\rW $};  \node (i2) at (2, -9) {$\lW 5, 6, 8, 7, -3, -2\rW $};
      \node (h1) at  (-4,-8) {$\lW 4, 8, 6, 5, 1, -3\rW $};  \node (h2) at (0, -8) {$\lW 5, 7, 8, 3, 0, -2\rW $}; \node (h3) at (4,- 8) {$\lW 6, 4, 8, 7, -3, -1\rW $};
  \node (g1) at  (-4,-7) {$\lW 5, 8, 6, 3, 1, -2\rW $};  \node (g2) at (0, -7) {$\lW 6, 7, 8, 3, -2, -1\rW $}; \node (g3) at (4, -7) {$\lW 7, 4, 8, 5, -3, 0\rW $};
     \node (f1) at  (-6,-6) {$\lW 5, 9, 6, 1, 2, -2\rW $};  \node (f2) at (-2,-6) {$\lW 6, 8, 4, 3, 1, -1\rW $}; \node (f3) at (2, -6) {$\lW 7, 5, 8, 3, -2, 0\rW $}; \node (f4) at (6, -6) {$\lW 8, 4, 6, 5, -3, 1\rW $};
   \node (e1) at  (-4,-5) {$\lW 6, 9, 4, 1, 2, -1\rW $};  \node (e2) at (0, -5) {$\lW 7, 8, 4, 3, -1, 0\rW $}; \node (e3) at (4, -5) {$\lW 8, 5, 6, 3, -2, 1\rW $};
   \node (d1) at  (-6,-4) {$\lW 6, 10, 2, 1, 3, -1\rW $};  \node (d2) at (-2,-4) {$\lW 7, 9, 4, -1, 2, 0\rW $}; \node (d3) at (2, -4) {$\lW 8, 6, 4, 3, -1, 1\rW $}; \node (d4) at (6, -4) {$\lW 9, 5, 6, 1, -2, 2\rW $};
    \node (c1) at  (-4,-3) {$\lW 7, 10, 2, -1, 3, 0\rW $};  \node (c2) at (0,- 3) {$\lW 8, 9, 4, -1, 0, 1\rW $}; \node (c3) at (4, -3) {$\lW 9, 6, 4, 1, -1, 2\rW $};
  \node (b1) at  (-4,-2) {$\lW 8, 10, 0, -1, 3, 1\rW $};  \node (b2) at (0, -2) {$\lW 9, 7, 4, -1, 0, 2\rW $}; \node (b3) at (4, -2) {$\lW 10, 6, 2, 1, -1, 3\rW $};
   \node (a1) at  (-2,-1) {$\lW 9, 10, 0, -1, 1, 2\rW $};  \node (a2) at (2, -1) {$\lW 10, 7, 2, -1, 0, 3\rW $};
  \node (max) at (0,0) {$\lW 10, 8, 0, -1, 1, 3\rW $};
    \draw  [->]
  (min) edge (i1)
  (min) edge (i2)
  (i1) edge (h1)
(i1) edge (h2)
(i2) edge (h3)
(h1) edge (g1)
(h2) edge (g1)
(h2) edge (g2)
(h3) edge (g3)
(g1) edge (f1)
(g1) edge (f2)
(g2) edge (f3)
(g3) edge (f3)
(g3) edge (f4)
(f1) edge (e1)
(f2) edge (e1)
(f2) edge (e2)
(f3) edge (e3)
(f4) edge (e3)
(e1) edge (d1)
(e1) edge (d2)
(e2) edge (d3)
(e3) edge (d3)
(e3) edge (d4)
(d1) edge (c1)
(d2) edge (c1)
(d2) edge (c2)
(d3) edge (c3)
(d4) edge (c3)
(c1) edge (b1)
(c2) edge (b2)
(c3) edge (b2)
(c3) edge (b3)
(b1) edge (a1)
(b2) edge (a2)
(b3) edge (a2)
  (a1) edge (max)
  (a2) edge (max)
;
\end{tikzpicture}\end{center}}
\caption{Hasse diagram of   $(\cA(z),<_\cA)$ for $z= t_{1,12}t_{2,11} t_{3,4} \in \tI_6$}
\label{poset-fig}
\end{figure}

Figure~\ref{poset-fig} shows an example of $(\cA(z),<_\cA)$. 
The lattice structure evident in this picture appears to be typical; we have used a computer to check the following conjecture
for $z \in \tI_n$ in the 333,307 cases when $0 < \ellhat(z)n \leq 100$. As a graded lattice, $(\cA(z),<_\cA)$ is
not necessarily Eulerian, distributive, or semi-modular, but
 its M\"obius function seems to always take values in $\{-1,0,1\}$.

\begin{conjecture}
The graded poset $(\cA(z), <_\cA)$ is a lattice for all $n$ and $z \in \tI_n$.
\end{conjecture}

A permutation $w \in \tS_n$ is \emph{321-avoiding} if no integers $a<b<c$
satisfy $w(a) > w(b) > w(c)$. 
An element $w \in \tS_n$ is \emph{fully commutative}
if we cannot write $w =u  s_i s_{i+1} s_i  v$
for any $u,v \in \tS_n$ and $i \in \ZZ$ with $\ell(w) = \ell(u) + \ell(v) + 3$.
The following extends \cite[Corollary 6.11]{HMP2} to affine type $A$.

\begin{corollary}
Let $z \in \tI_n$. The following are equivalent:
(a) $|\cA(z)|=1$,
(b) $\alpha_{\min}(z) =\alpha_{\max}(z)$,
(c) $\alpha_R(z) = \alpha_L(z)$,
(d) $z$ is 321-avoiding,
and
(e) $z$ is fully commutative.
\end{corollary}

\begin{remark}
Biagioli, Jouhet, and Nadeau \cite[Proposition 3.3]{BJN} have derived a length generating function for 
the involutions in $\tS_n$ with these equivalent properties.
\end{remark}

\begin{proof}
The equivalence of (a), (b), and (c) is clear  
from
Theorem~\ref{poset-thm}.
The equivalence of (d) and (e) is well-known; see the results of Green \cite[Theorem 2.7]{Green},
Lam \cite[Proposition 35]{Lam}, or Fan and Stembridge \cite{FS}.
If  $z$ is fully commutative then it follows from \cite[Proposition 7.12]{HMP2} that $|\cA(z)|=1$; this can
also be shown by a direct argument.
Conversely, if $z$ is not 321-avoiding, then there must exist $a,a' \in \ZZ$ with $a < a' \leq z(a') < z(a)$,
in which case the formulas in Corollary~\ref{min-cor} show that $\alpha_{\min}(z) \neq \alpha_{\max}(z)$.
Thus (e) $\Rightarrow$ (a) and (b) $\Rightarrow$ (d), which completes the proof.
\end{proof}

\section{Cycle removal process}\label{cycle-sect}

A \emph{doubly infinite sequence} $\ttS=(\dots a_{-2} a_{-1}a_0a_1a_2\dots)$
is an orbit of a map $f : \ZZ \to X$ under the action of $\ZZ$ by translation. 
If for some $m \in\ZZ$ we have $f(m+i) = a_i$ for all $i \in \ZZ$,
then we call 
 $\ttS$
the \emph{string representation} of $f$.
We sometimes use the term \emph{string} as a shorthand for doubly infinite sequence.
By Proposition-Definition~\ref{window-def}, no two elements of $\tS_n$ have the same string representation.

%

%
%
%

For the remainder of this section,
we fix the following notation.
Let $y \in \tI_n$ and $w \in \tS_n$ with $w^{-1} \in \cA(y)$,
and write $\ttS$ for the string representation of $w$.
We construct a sequence of strings $\ttS_0, \ttS_1,\dots,\ttS_p$
and pairs $(a_1,b_1),(a_2,b_2),\dots,(a_p,b_p) \in \ZZ\times \ZZ$ 
by the following algorithm. 

\begin{definition}[Cycle removal process]\label{cyc-def}
Start with $\ttS_0 = \ttS$.
Then, for each $i=0,1,2,\dots$, define $(a_{i+1},b_{i+1})$ to be any pair of 
integers with $a_{i+1}<b_{i+1}$ such that $b_{i+1}a_{i+1}$ appears as a consecutive subsequence
of $\ttS_i$. 
 If no such pair exists, so that $\ttS_i$ is an increasing sequence, then the process terminates with $p=i$.
Otherwise, we form $\ttS_{i+1}$ by removing all numbers congruent to $a_{i+1}$ or $b_{i+1}$
modulo $n$ from $\ttS_i$ and continue.
\end{definition}

We refer to this procedure as the \emph{cycle removal process} for $w$.
The process must terminate in a finite number of steps since each iteration removes two congruence classes of 
integers modulo $n$ from the current string.
By construction  the last string $\ttS_p$ is a strictly increasing sequence. 

\begin{example}
Suppose $y = t_{1,8}t_{2,7} \in \tI_4$ and $w = \lW 3,4,-3,-2\rW  \in  \cA(y)^{-1}$, so that
\ben
\item[] $\ttS =( \dots,-5,-4,-11,-10,-1,0,-7,-6,3, 4, -3, -2, 7, 8,1,2,11,12,5,6,\dots)$.
\een
For our first cycle, choose $(a_1,b_1) = (1,8)$. Then
\ben
\item[] $\ttS_1 =( \dots,-5,-10,-1,-6,3, -2, 7, 2,11,6,\dots)$.
\een
For our second cycle, choose $(a_2,b_2) = (2,7)$. Removing all numbers congruent to 2 or 7 modulo 4 from $\ttS_1$
leaves the empty string $\ttS_2 = \emptyset$, so our process terminates with $p=2$.
\end{example}

\begin{theorem}\label{cycle-process-thm}
No matter how the cycle removal process for $w$ is carried out, 
$ \ell'(y)=p$ and
 $ \cC(y) = \{ (a_i+mn,b_i+mn) : i \in [p],\ m \in \ZZ\}$,
 and $\ttS_p$ is the increasing sequence of fixed points of $y$.
\end{theorem}

\begin{proof}
Fix $j \in [p]$ and suppose $a,b \in \ZZ$ are such that $a-a_j = b-b_j \in n \ZZ$, so that $a<b$.
Then $ba$ must be a consecutive subsequence of $\ttS_{j-1}$,
so there are integers $c_1,c_2,\dots,c_m$ such that $bc_1c_2\cdots c_ma$
is a consecutive subsequence of $\ttS =\ttS_0$. The number of such integers must satisfy $m \leq n-2$,
since if $m=n-1$ then we would have $a=b+n \not < b$, while if $m \geq n$ then $c_n =b+n$ would appear between $b$ and $a$ in $\ttS_{j-1}$. It follows that $a,b,c_1,c_2,\dots,c_m$ belong to distinct congruences classes modulo $n$.

Since none of the numbers between $b$ and $a$ in $\ttS$  are present in $\ttS_{j-1}$,
there is a subset $I \subset [j-1]$
and integers $g_i,h_i \in \ZZ$ with $g_i-a_i = h_i-b_i \in n\ZZ$ for each $i \in I$
such that $\{c_1,c_2,\dots,c_m\} = \{ g_i : i \in I \} \cup \{ h_i : i \in I\}$. Note, as such, that $m=2|I|$.
Let $\sim$  be the equivalence relation on finite integer sequences generated by setting
 $\alpha \sim \beta$ whenever it is possible to obtain $\beta$ from $\alpha$
by replacing a consecutive subsequence of the form $e_2e_3e_1$ by $e_3e_1e_2$ where $e_1<e_2<e_3$.
We claim that $b c_1 c_2\cdots c_ma \sim h_{i_1}g_{i_1}h_{i_2}g_{i_2}\cdots h_{i_k}g_{i_k}ba$
where $i_1,i_2,\dots,i_k$ are the elements of $I$ in some order.
Applying this property inductively with $(a,b)$ replaced by the pairs $(g_i,h_i)$, we deduce that  at least 
$bc_1c_2\cdots c_ma \sim b h_{i_1}g_{i_1}h_{i_2}g_{i_2}\cdots h_{i_k}g_{i_k}a$.
Corollary~\ref{equiv-cor} and Lemma~\ref{321-lem} imply
that no sequence equivalent to $bc_1c_2\cdots c_ma$ under $\sim$
 contains a consecutive subsequence $e_3e_2e_1$ with $e_1<e_2<e_3$.
From this property, it is easy to see that 
$ b h_{i_1}g_{i_1}h_{i_2}g_{i_2}\cdots h_{i_k}g_{i_k}a$
is equivalent under $\sim$ to a sequence of the same form 
with  $g_{i_k} = \max \{ g_i : i \in I\}$, as well as to a (possibly different) sequence of the same form with $h_{i_1} = \min\{ h_i : i \in I\}$.
But then it follows that 
\be\label{g-eq}
g_i<a<b<h_i\quad\text{for all $i \in I$}
\ee
 so   
$b h_{i_1}g_{i_1} h_{i_2}g_{i_2}\cdots h_{i_k}g_{i_k}a \sim h_{i_1}g_{i_1} h_{i_2}g_{i_2}\cdots h_{i_k}g_{i_k}ba$ as desired.
We conclude by 
Corollary~\ref{equiv-cor} 
that there exists $v \in \cA(y)^{-1}$ whose string representation contains $ba$ as a consecutive subsequence,
so $(a,b) \in \cC(y)$  
by Corollary~\ref{course-cor}(a).
Hence  $(a_j,b_j) \in \cC(y)$ for each $j \in [p]$.
 Corollary~\ref{course-cor}(c) implies that each element of the increasing sequence  $\ttS_p$ is a fixed point of $y$,
so the theorem follows.
\end{proof}

Given $a,b,c,\dots$, we say that a string ``has the form $\dash a \dash b \dash c \dash \dots \dash$''
as a shorthand for  the statement that the string contains $(a,b,c,\dots)$ as a not necessarily consecutive subsequence.
Maintain the notation of Definition~\ref{cyc-def} and Theorem~\ref{cycle-process-thm} in the following corollaries.

\begin{corollary}\label{3412-cor}
If $\ttS$ has the form $\dash b'\dash b\dash a \dash a'\dash$
where $(a,b),(a',b') \in \cC(y)$, then $a<a'<b'<b$.
\end{corollary}

\begin{proof}
This is equivalent to  property \eqref{g-eq} shown in the proof of Theorem~\ref{cycle-process-thm}.
\end{proof}

\begin{corollary}\label{321cycle-cor}
None of $\ttS_0,\ttS_1,\dots,\ttS_p$ contains
a consecutive subsequence $cba$ where $a<b<c$.
\end{corollary}

\begin{proof}
If any string $\ttS_i$ had a consecutive subsequence of the form $cba$ where $a<b<c$,
then Theorem~\ref{cycle-process-thm} would imply that both $(a,b)$ and $(b,c)$ belong to $\cC(y)$, which is impossible.
\end{proof}

We can use the preceding results to give a more explicit set of conditions characterising the elements of $\cA(y)$.
Restricted to involutions in the finite symmetric group $S_n \subset \tS_n$,
the following theorem  is equivalent to a result of Can, Joyce, and Wyser \cite[Theorem 2.5]{CJW}.

\begin{theorem}\label{excluded-thm}
Let $y \in \tI_n$ and let $\ttS$ be the string representation of  an affine permutation $w \in \tS_n$.
Then $w^{-1} \in \cA(y)$ if and only if the following properties hold:
\ben
\item  If  $X,Y \in \ZZ$ are fixed points of $y$  with $X<Y$ then $\ttS$ has the form $\dash X \dash Y\dash$.

\item If $(a,b) \in \cC(y)$ then $\ttS$ has the form $\dash b \dash a \dash$.

\item If $X \in \ZZ$ is a fixed point of $y$ and $(a,b) \in \cC(y)$ then:
\ben
\item $\ttS$ does not have the form $\dash b \dash X\dash a \dash$.

\item If $X<a<b$ then $\ttS$ has the form $\dash X \dash b \dash a \dash$.

\item If $a<b<X$ then $\ttS$ has the form $ \dash b \dash a \dash X \dash$.
\een

\item If $(a,b),(a',b') \in \cC(y)$ and $a<a'<b'<b$ then
$\ttS$ has one of the forms
\[\dash b \dash a\dash b'\dash a' \dash
\qquord \dash b' \dash b\dash a\dash a'\dash \qquord \dash b'\dash a' \dash b \dash a\dash
\]
while if $a<a'$ and $b<b'$ then $\ttS$ has the form  $\dash b \dash a \dash b' \dash a' \dash$.

\een
\end{theorem}

\begin{proof}
First suppose $w^{-1} \in \cA(y)$.
Property 1 holds by Theorem~\ref{cycle-process-thm}, and property 2 follows by Corollary~\ref{course-cor}(c).
Property 3(a) holds since $ba$ must occur as a consecutive subsequence of some intermediate string $\ttS_i$
in the cycle removal process for $w$.
Following this observation, one can deduce
properties 3(b) and 3(c) from Corollary~\ref{course-cor}(b).
By similar reasoning,  if $(a,b),(a',b') \in \cC(y)$ then 
$\ttS$ cannot have the form $\dash b \dash b' \dash a \dash a'\dash$,
so
property 4 follows from  Corollary~\ref{course-cor}(b) and Corollary~\ref{3412-cor}.

Conversely, suppose $\ttS$ has the given properties.
Then $b$ appears to the left of $a$ in $\ttS$ whenever $(a,b) \in \cC(y)$,
and in this case no fixed points of $y$ appear between $b$ and $a$ in $\ttS$.
Moreover, if $(a,b),(a',b') \in \cC(y)$ and either $a$ or $b$
appears between $b'$ and $a'$ in $\ttS$, 
then $\ttS$ has the form $\dash b'\dash b\dash a \dash a' \dash$ and $a<a'<b'<b$.
Let $\sim_\cA$ be the symmetric closure of $\leq_\cA$,
from Theorem~\ref{poset-thm}.
By an inductive argument similar to the one 
in the proof of Theorem~\ref{cycle-process-thm},
we deduce from the preceding observations that 
$w^{-1} \sim_\cA (w')^{-1}$
for a permutation $w' \in \tS_n$
whose string representation $\ttS'$ also satisfies the given conditions
and  has that property that $ba$ appears as a consecutive subsequence for each $(a,b) \in \cC(y)$.
Property 1 implies that all fixed points of $y$ appear in order in $\ttS'$.
If $cab$ is a consecutive subsequence of $\ttS'$ where $(a,c) \in \cC(y)$ and $b=y(b) $,
then property 3 implies that $a<b$.
If $bca$ is a consecutive subsequence of $\ttS'$ where $(a,c) \in \cC(y)$ and $b=y(b) $,
then property 3 implies that either $b<a$ or $a<b<c$.
Finally, if $(a,b),(a',b') \in \cC(y)$ where $a<a'$, then $\ttS'$
may contain $bab'a'$ or $b'baa'$ or $b'a'ba$ as a consecutive subsequence,
and in the second two cases, property 4 implies that $a<a'<b'<b$.
Given these observations, it is straightforward to show that $(w')^{-1} \sim_\cA (w'')^{-1}$
for a permutation $w'' \in \tS_n$ whose string representation $\ttS''$
contains $ba$ as a consequence subsequence for each $(a,b) \in \cC(y)$,
and has the property that $a$ appears to the left of $a'$ 
whenever $a<a'$ and $a\leq y(a)$ and $a'\leq y(a')$.
By Corollary~\ref{min-cor}, the unique such permutation is $w'' = \alpha_{\min}(y)^{-1}$, so 
 by Theorem~\ref{poset-thm} we have $w^{-1} \in \cA(y)$.
\end{proof}

Fix a subset $E\subset [n]$ of size $m$. 
Let $\phi_E : [m] \to E$ and $\psi_E : E \to [m]$ be order-preserving bijections.
If $w \in S_n$ is a permutation in a finite symmetric group,
then its \emph{standardisation}  is the permutation $[w]_E = \psi_{w(E)} \circ w \circ \phi_E \in S_m$.
If $w^2=1$ and $w(E)=E$, then $([w]_E)^2=1$.

 The Demazure product $\circ$ on $\tS_n$ restricts to an associative product $S_n\times S_n \to S_n$
 and each involution $y \in I_n = \tI_n \cap S_n$ has $\cA(y) \subset S_n$.
Can, Joyce, and Wyser's description of $\cA(y)$ for $y \in I_n$ in \cite{CJW} 
implies that $w \in S_n$ belongs to $\cA(y)$ if and only if $[w]_E \in \cA([y]_E)$
for all subsets $E \subset [n]$ which are invariant under $y$ and contain at most two $y$-orbits; cf. \cite[Corollary 3.19]{HMP3}.
This ``local''  criterion for membership in $\cA(y)$ was an important tool in the proofs of the main results in \cite{HMP3}.

\begin{example}
The atoms of the reverse permutations in $S_2$, $S_3$, and $S_4$ are given by
$\cA(21) = \{ 21\}$,  $\cA(321) = \{ 312, 231 \}$, and $\cA(4321) = \{ 4213, 3412, 2431\}$.
If $y = n\cdots 321 \in I_n$ and $E\subset [n]$ is $y$-invariant with at most two orbits,
then $[y]_E$ is either $21$, $321$, or $4321$,
and one can check that requiring $[w]_E \in \cA([y]_E)$ 
imposes precisely the sort of conditions we saw in Theorem~\ref{excluded-thm}.
\end{example}

This result can be extended to the affine case, provided we 
give the right definition of the standardisation of an affine permutation.
Fix a subset $E \subset \ZZ$ with $|(E+n\ZZ) \cap [n]| = m$, 
and 
define $\tphi_{E,n}$ as the unique order-preserving
$\ZZ \to E +n\ZZ $ with $\tphi_{E,n}([m]) \subset[n]$.

\begin{lemma} \label{affine-std-lem}
If $ w\in \tS_n$ then there is a unique order-preserving bijection $\tpsi_{E,w} : w(E) + n\ZZ \to \ZZ$
such that $\tpsi_{E,w} \circ w \circ \tphi_{E,n} \in \tS_m$. 
If $w \in \tI_n$ and $w(E)=E$
then $\tphi_{E,n}$ and $\tpsi_{E,w}$ are inverses.
\end{lemma}

\begin{proof}
Let $w \in \tS_n$.
The first assertion follows on checking that the images of $1,2,\dots,m$ under $w\circ \tphi_{E,n}$
represent the distinct congruences classes in $w(E) +n\ZZ$ modulo $n$.  
For the second assertion,
one uses the fact that if $w$ is an involution, then for each $i \in [n]$ with $w(i)\neq i$, there is a unique $j \in [n]$
with $w(i) \equiv j \modu n)$ and $w(i)+w(j) = i+j$.
The details are left 
as an exercise.
%
%
\end{proof}

Given $w \in \tS_n$ and $E \subset \ZZ$ with  $|(E+n\ZZ) \cap [n]| = m$, define 
$[w]_{E,n} = \tpsi_{E,w}  \circ w \circ \tphi_{E,n} \in \tS_m.$
We refer to $[w]_{E,n}$ as the \emph{(affine) standardisation} of $w$.
One has $[w]_{E, n} = [w]_{E+mn,n} $ for all $m \in \ZZ$.
When $n$ is clear from context, we write $[w]_E$ instead of $[w]_{E,n}$. This is consistent with our earlier notation since
if $E \subset [n]$ and $w \in S_n\subset \tS_n$,
 then $\tphi_{E,n}|_{[m]} = \phi_E$ and $\tpsi_{E,w} |_E = \psi_{w(E)}$.

\begin{corollary} If $E \subset \ZZ$, 
 $y \in \tI_n$, and $y(E)=E$, then $[y]_E \in \tI_n$.
\end{corollary}

\begin{proof}
By Lemma~\ref{affine-std-lem}, both $ \tpsi_{E,y} \circ \tphi_{E,n} $ and $\tphi_{E,n} \circ \tpsi_{E,y}$
are identity maps so   $([y]_E)^2=1$.
\end{proof}

\begin{example}
Standardisation has a simple interpretation in terms of winding diagrams.
If  $E=y(E)\subset \ZZ$ then the winding diagram of $[y]_E$ is formed from that of $y \in \tI_n$
by erasing the vertices in $[n] \setminus (E+n\ZZ)$ along with their incident edges, and then replacing the numbers that remain by $1,2,\dots,m=|(E+n\ZZ) \cap [n]|$ in order.
For example,
 if $n=8$ and  $E = \{2,4,6,7,8\}$, then  
 
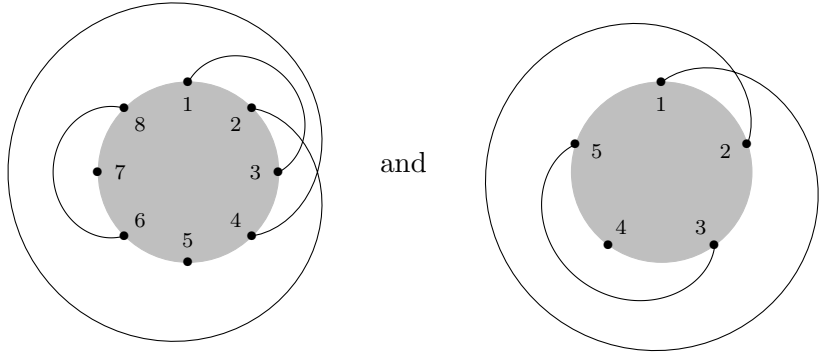
\[
\begin{tikzpicture}[baseline=0,scale=0.3,label/.style={postaction={ decorate,transform shape,decoration={ markings, mark=at position .5 with \node #1;}}}]
{
\draw[fill,lightgray] (0,0) circle (4.0);
\node at (2.44929359829e-16, 4.0) {$_\bullet$};
\node at (1.83697019872e-16, 3.0) {$_1$};
\node at (2.82842712475, 2.82842712475) {$_\bullet$};
\node at (2.12132034356, 2.12132034356) {$_2$};
\node at (4.0, 0.0) {$_\bullet$};
\node at (3.0, 0.0) {$_3$};
\node at (2.82842712475, -2.82842712475) {$_\bullet$};
\node at (2.12132034356, -2.12132034356) {$_4$};
\node at (2.44929359829e-16, -4.0) {$_\bullet$};
\node at (1.83697019872e-16, -3.0) {$_5$};
\node at (-2.82842712475, -2.82842712475) {$_\bullet$};
\node at (-2.12132034356, -2.12132034356) {$_6$};
\node at (-4.0, -4.89858719659e-16) {$_\bullet$};
\node at (-3.0, -3.67394039744e-16) {$_7$};
\node at (-2.82842712475, 2.82842712475) {$_\bullet$};
\node at (-2.12132034356, 2.12132034356) {$_8$};
\draw [-,>=latex,domain=0:100,samples=100] plot ({(4.0 + 2.0 * sin(180 * (0.5 + asin(-0.9 + 1.8 * (\x / 100)) / asin(0.9) / 2))) * cos(90 - (0.0 + \x * 0.9))}, {(4.0 + 2.0 * sin(180 * (0.5 + asin(-0.9 + 1.8 * (\x / 100)) / asin(0.9) / 2))) * sin(90 - (0.0 + \x * 0.9))});
\draw [-,>=latex,domain=0:100,samples=100] plot ({(4.0 + 4.0 * sin(180 * (0.5 + asin(-0.9 + 1.8 * (\x / 100)) / asin(0.9) / 2))) * cos(90 - (45.0 + \x * 4.5))}, {(4.0 + 4.0 * sin(180 * (0.5 + asin(-0.9 + 1.8 * (\x / 100)) / asin(0.9) / 2))) * sin(90 - (45.0 + \x * 4.5))});
\draw [-,>=latex,domain=0:100,samples=100] plot ({(4.0 + 2.0 * sin(180 * (0.5 + asin(-0.9 + 1.8 * (\x / 100)) / asin(0.9) / 2))) * cos(90 - (225.0 + \x * 0.9))}, {(4.0 + 2.0 * sin(180 * (0.5 + asin(-0.9 + 1.8 * (\x / 100)) / asin(0.9) / 2))) * sin(90 - (225.0 + \x * 0.9))});
}
\end{tikzpicture}
\qquand
\begin{tikzpicture}[baseline=0,scale=0.3,label/.style={postaction={ decorate,transform shape,decoration={ markings, mark=at position .5 with \node #1;}}}]
{
\draw[fill,lightgray] (0,0) circle (4.0);
\node at (2.44929359829e-16, 4.0) {$_\bullet$};
\node at (1.83697019872e-16, 3.0) {$_1$};
\node at (3.80422606518, 1.2360679775) {$_\bullet$};
\node at (2.85316954889, 0.927050983125) {$_2$};
\node at (2.35114100917, -3.2360679775) {$_\bullet$};
\node at (1.76335575688, -2.42705098312) {$_3$};
\node at (-2.35114100917, -3.2360679775) {$_\bullet$};
\node at (-1.76335575688, -2.42705098312) {$_4$};
\node at (-3.80422606518, 1.2360679775) {$_\bullet$};
\node at (-2.85316954889, 0.927050983125) {$_5$};
\draw [-,>=latex,domain=0:100,samples=100] plot ({(4.0 + 4.0 * sin(180 * (0.5 + asin(-0.9 + 1.8 * (\x / 100)) / asin(0.9) / 2))) * cos(90 - (0.0 + \x * 4.32))}, {(4.0 + 4.0 * sin(180 * (0.5 + asin(-0.9 + 1.8 * (\x / 100)) / asin(0.9) / 2))) * sin(90 - (0.0 + \x * 4.32))});
\draw [-,>=latex,domain=0:100,samples=100] plot ({(4.0 + 2.0 * sin(180 * (0.5 + asin(-0.9 + 1.8 * (\x / 100)) / asin(0.9) / 2))) * cos(90 - (144.0 + \x * 1.44))}, {(4.0 + 2.0 * sin(180 * (0.5 + asin(-0.9 + 1.8 * (\x / 100)) / asin(0.9) / 2))) * sin(90 - (144.0 + \x * 1.44))});
}
\end{tikzpicture}
\]
represent $y =t_{1,3} t_{2,12} t_{6,8} \in \tI_8$ and $[y]_E  = t_{1,7} t_{3,5} \in \tI_5$, respectively.
\end{example}

The following is a corollary of Theorem~\ref{excluded-thm} via the preceding lemmas.

\begin{corollary}\label{std-cor}
Let $y \in \tI_n$, $w \in \tS_n$, and $X = \{1,y(1),\dots,n,y(n)\}$. The following are equivalent:
\ben
\item[(a)] $w \in \cA(y)$.
\item[(b)] $[w]_E \in \cA([y]_E)$
for each subset $E\subset X$ with $y(E)=E$.
\item[(c)] $[w]_E \in \cA([y]_E)$
for each subset $E\subset X$ with $y(E)=E$ and containing at most two $y$-orbits.
\een
\end{corollary}

\begin{proof}
If $E\subset \ZZ$ has $y(E)=E$, then  $\cC([y]_E)$ consists of the pairs $(i,j) \in \ZZ\times \ZZ$
such that $(a,b) \in \cC(y)$ for $a=\tphi_{E,n}(i)$ and $b = \tphi_{E,n}(j)$.
If $w \in \tS_n$, then applying
$\tphi_{E,n}$ to each term in the
string representation of $([w]_E)^{-1}$  gives the same thing as removing all numbers not in $E + n\ZZ$ from
the string representation of $w^{-1}$. 
Since $\tphi_{E,n}$ is order-preserving, it follows from Theorem~\ref{excluded-thm} that (a) $\Rightarrow$ (b) $\Rightarrow$ (c).
The conditions in Theorem~\ref{excluded-thm} each 
apply to the relative ordering, within the string representation of $w^{-1}$, of  numbers from a set of at most two cycles of $y$,
so (c) $\Rightarrow$ (a).
\end{proof}

\section{Covering transformations}\label{ct-sect}

The sets $\cA(z)$ for $z \in \tI_n$
are closely related to the restriction of the Bruhat order $<$ on $\tS_n$
to $\tI_n$. 
If $z \in \tI_n$ then
$\ellhat(z) = \frac{1}{2}(\ell(z) + \ell'(z)) = \ell(w)$ for all $w \in \cA(z)$, where $\ell'$ 
is given by Definition~\ref{abs-len-def}.
 The following statements derive from results of Hultman \cite{H1,H2}; see \cite[\S3.1]{HMP3}.

 \begin{proposition} \label{graded-prop}
 The poset $(\tI_n,<)$ is graded with rank function $\ellhat$.
 \end{proposition}

\begin{lemma}
\label{cover-lem}
If $y,z \in \tI_n$ and $w \in \cA(z)$. Then $y \leq z$ if and only if some $v\in \cA(y)$
has $v \leq w$.
\end{lemma}

In the usual Bruhat order on $\tS_n$ 
all covering relations have the form $w \lessdot wt$ where $t$ is a reflection.
Describing the covering relations in the subposet 
$(\tI_n,<)$ is a more delicate problem, but one we can attack using Lemma~\ref{cover-lem}
and  the results of the previous two sections.

Incitti \cite{Incitti1,Incitti2} solves the analogue of this problem for
 involutions in the finite Weyl groups of type $A$, $B$, and $D$,
by associating to each reflection a \emph{covering transformation} to play the role of right multiplication.
In this section, we define our own covering transformations $\tau^n_{ij} : \tI_n \to \tI_n$.
Restricted to involutions in   $S_n \subset \tS_n$, these operators will coincide with the 
maps $\tau_{ij}$ in \cite{HMP3}, which are themselves just a different notation for
the maps $\ct_{ij}$ in \cite{Incitti1}.

 \begin{definition}
 Let $\sim_{\text{graph}}$ be the equivalence relation on vertex-coloured graphs with integer vertices
 in which $\cG \sim_{\text{graph}} \cH$ if and only there exists a graph isomorphism $\cG \to \cH$
 which is an order-preserving bijection.
 \end{definition}
 
 \begin{definition}\label{cG-def}
  Fix $y \in \tI_n$ and $i,j \in \ZZ$ with $i < j \not \equiv i \modu n)$.
Define $\cG_{ij}(y)$ as the graph with vertex set $ \{i,j,y(i),y(j)\}$ and
edge  set
$\{ \{ i, y(i)\}, \{j, y(j)\} \} \setminus \{ \{i\}, \{j\}\}$,
in which the vertices $i$ and $j$ are coloured white and 
all other vertices are coloured black.
Let $k$ be the number of vertices in $\cG_{ij}(y)$ and define
 $\cD_{ij}(y)$ as 
the
unique 
vertex-coloured graph on $[k]$  with $\cD_{ij}(y) \sim_{\text{graph}} \cG_{ij}(y)$.
\end{definition}

 We use these definitions to simplify our notation.
There are only 20 possibilities for $\cD_{ij}(y)$, which we draw by arranging the vertices $1,2,\dots,k$  in order from left to right,
using $\circ$ for the white vertices and $\bullet $ for the black vertices.

\begin{example}
If $y,z \in \tI_n$ are such that $y(i) < j =y(j) < i$ and $i < z(j) < j < z(i)$ then 
\[ \cD_{ij}(y) = \diagramcBA \qquand \cD_{ij}(z) = \diagramDcBa. \]
\end{example}

\begin{definition}\label{tau-def}
Fix $y \in \tI_n$ and $i,j \in \ZZ$  with $i < j \not \equiv i \modu n)$. 
Set $t_{ii} = t_{jj} = 1$ and define
 \[ 
 (\circ, \circ) = t_{ij},\quad
 (\circ,\bullet) = t_{i,y(j)},\quad
 (\bullet,\circ) = t_{y(i),j},
\quad
 \overline y = \begin{cases} y \cdot  t_{i,y(i)}  &\text{if $i \equiv y(j) \modu n)$}
 \\
 y \cdot t_{i,y(i)} \cdot t_{j,y(j)} &\text{if $i \not \equiv y(j) \modu n)$.}
 \end{cases}
 \]
Using this notation, let $\tau^n_{ij}(y) \in \tI_n$ be given as follows:
\[\tau^n_{ij}(y) =\begin{cases}
 (\circ, \circ)\cdot y\cdot (\circ, \circ) & \text{if $\cD_{ij}(y) $ is $
 \diagramACb$, $\diagrambAC$, $\diagrambADc$, $\diagramCDab$ or $\diagramcdAB$}
\\
 (\circ,\bullet)\cdot y\cdot (\circ,\bullet) &\text{if $\cD_{ij}(y)$ is $\diagramAcB$}
\\
(\bullet,\circ)\cdot y\cdot (\bullet,\circ)&\text{if $\cD_{ij}(y)$ is $\diagramBaC$}
 \\
(\circ,\bullet)\cdot y\cdot (\circ,\bullet)  & \text{if $\cD_{ij}(y) $ is $\diagramCdaB$ and $i \not \equiv y(j) \modu n)$} 
 \\
(\circ, \circ)\cdot \overline y  & \text{if $\cD_{ij}(y) $ is $\diagramCdaB$ and $i \equiv y(j) \modu n)$}
 \\
(\circ, \circ)\cdot \overline y&\text{if $\cD_{ij}(y)$ is $\diagramBadC$}
\\
 (\circ,\bullet) \cdot \overline y&\text{if $\cD_{ij}(y)$ is $\diagramBaDc$}
\\
(\bullet,\circ)  \cdot \overline y&\text{if $\cD_{ij}(y)$ is $\diagrambAdC$}
 \\
(\circ,\circ)\cdot  y &\text{if $\cD_{ij}(y)$ is $\diagramAB$}
\\
y&\text{otherwise}.
\end{cases}
\]
\end{definition}

\begin{remark}
There is a lot to unpack here. We include a few remarks about our notation:
\ben
\item[(a)] Let $i'$ be the number adjacent to $i$ in $\cG_{ij}(y)$, if one exists, and define $j'$ similarly.
Then $(\circ,\circ)$ transposes $i$ and $j$ (as well as $i+mn$ and $j+mn$ for all $m \in \ZZ$),
while $(\circ,\bullet)$ transposes $i$ and $j'$
and
$(\bullet,\circ)$ transposes $i'$ and $j$.
The unique order-preserving graph isomorphism $\cG_{ij}(y) \to \cD_{ij}(y)$ maps
$i$ and $j$ to vertices labeled $\circ$ and $i'$ and $j'$ to vertices labeled $\bullet$.

\item[(b)]
If $i\equiv y(j) \modu n)$ then $t_{i,y(i)} = t_{j,y(j)}$, so
 $\overline y$ is the unique element of $\tS_n$ which fixes each element of $\{i,j,y(i),y(j)\} + n\ZZ$
and which agrees with $y$ at all integers outside this set.

\item[(c)] 
If $i \not \equiv y(j) \modu n)$ then both $\cG_{ij}(y)$ and $\cG_{ij}(\tau^n_{ij}(y))$ have vertex set
$V = \{i,j,y(i),y(j)\}$. In this case, given $V$,
 the value of  $\tau^n_{ij}(y)$
is uniquely determined by the graphs
$\cD_{ij}(y)$ and $\cD_{ij}(\tau^n_{ij}(y))$.
If $i \equiv y(j) \modu n)$ then $j \equiv y(i) \modu n)$ also holds,
 and  $\cD_{ij}(y)$ must be
 \[ \diagramcDAb
 \quord
 \diagramCdaB
 \quord 
 \diagrambADc
 \quord
 \diagramBadC.
 \]
When this occurs
 $\cG_{ij}(y)$ and $\cG_{ij}(\tau^n_{ij}(y))$ may have different vertex sets; see Table~\ref{tau-table}.

\een
\end{remark}

\def\tableskip{
\\[-8pt] & 
\\  
\hline &
\\[-8pt]
}

\def\tablebegin{
\hline &
\\[-8pt]
}

\def\tableend{
\\[-8pt] &\\
\hline 
}

\begin{table}[h]
\[
\barr{| l | l |}
\tablebegin
\cD_{ij}(y)\text{ for $y \in \tI_n$ and $i<j \not \equiv i \modu n)$} & \cD_{ij}(\tau^n_{ij}(y))
\tableskip
\diagramAB & \diagramBA  \\
\diagramBadC \quad\text{if $i\equiv y(j)\modu n)$}  & \diagramBA  \\
\diagramCdaB \quad\text{if $i\equiv y(j)\modu n)$}  & \diagramBA 
\tableskip
\diagramBaC & \diagramCbA \\
\diagrambAC & \diagramcBA \\
\diagramAcB & \diagramCbA \\
\diagramACb & \diagramCBa 
\tableskip
\diagrambADc & \diagramcDAb 
\tableskip
\diagramBadC \quad\text{if $i\not\equiv y(j)\modu n)$} & \diagramDbcA \\
\diagramBaDc & \diagramDbCa \\
\diagrambAdC & \diagramdBcA
\tableskip
\diagramCdaB \quad\text{if $i \not \equiv y(j) \modu n)$}& \diagramDcbA \\
\diagramCDab & \diagramDCba \\
\diagramcdAB & \diagramdcBA 
\tableend
\earr
\]
\caption{Possible values for $\cD_{ij}(y)$ and $ \cD_{ij}(\tau^n_{ij}(y))$. 
If $\cD_{ij}(y)$ does not match any of the listed cases,
then $\tau^n_{ij}(y) = y$.
Only in the second two cases do $\cD_{ij}(y)$ and $ \cD_{ij}(\tau^n_{ij}(y))$ have different vertex sets.
}
\label{tau-table}
\end{table}

\begin{example}
For $z = t_{1,3}t_{5,7} \in \tI_7$ we have
\[
\tau^7_{1,4}\(
\begin{tikzpicture}[baseline=0,scale=0.2,label/.style={postaction={ decorate,transform shape,decoration={ markings, mark=at position .5 with \node #1;}}}]
{
\draw[fill,lightgray] (0,0) circle (4.0);
\node at (2.44929359829e-16, 4.0) {$_\bullet$};
\node at (1.71450551881e-16, 2.8) {$_1$};
\node at (3.12732592987, 2.49395920743) {$_\bullet$};
\node at (2.18912815091, 1.7457714452) {$_2$};
\node at (3.89971164873, -0.890083735825) {$_\bullet$};
\node at (2.72979815411, -0.623058615078) {$_3$};
\node at (1.73553495647, -3.60387547161) {$_\bullet$};
\node at (1.21487446953, -2.52271283013) {$_4$};
\node at (-1.73553495647, -3.60387547161) {$_\bullet$};
\node at (-1.21487446953, -2.52271283013) {$_5$};
\node at (-3.89971164873, -0.890083735825) {$_\bullet$};
\node at (-2.72979815411, -0.623058615078) {$_6$};
\node at (-3.12732592987, 2.49395920743) {$_\bullet$};
\node at (-2.18912815091, 1.7457714452) {$_7$};
\draw [-,>=latex,domain=0:100,samples=100] plot ({(4.0 + 2.0 * sin(180 * (0.5 + asin(-0.9 + 1.8 * (\x / 100)) / asin(0.9) / 2))) * cos(90 - (0.0 + \x * 1.0285714285714287))}, {(4.0 + 2.0 * sin(180 * (0.5 + asin(-0.9 + 1.8 * (\x / 100)) / asin(0.9) / 2))) * sin(90 - (0.0 + \x * 1.0285714285714287))});
\draw [-,>=latex,domain=0:100,samples=100] plot ({(4.0 + 2.0 * sin(180 * (0.5 + asin(-0.9 + 1.8 * (\x / 100)) / asin(0.9) / 2))) * cos(90 - (205.71428571428572 + \x * 1.0285714285714282))}, {(4.0 + 2.0 * sin(180 * (0.5 + asin(-0.9 + 1.8 * (\x / 100)) / asin(0.9) / 2))) * sin(90 - (205.71428571428572 + \x * 1.0285714285714282))});
}
\end{tikzpicture}
\) =
\begin{tikzpicture}[baseline=0,scale=0.2,label/.style={postaction={ decorate,transform shape,decoration={ markings, mark=at position .5 with \node #1;}}}]
{
\draw[fill,lightgray] (0,0) circle (4.0);
\node at (2.44929359829e-16, 4.0) {$_\bullet$};
\node at (1.71450551881e-16, 2.8) {$_1$};
\node at (3.12732592987, 2.49395920743) {$_\bullet$};
\node at (2.18912815091, 1.7457714452) {$_2$};
\node at (3.89971164873, -0.890083735825) {$_\bullet$};
\node at (2.72979815411, -0.623058615078) {$_3$};
\node at (1.73553495647, -3.60387547161) {$_\bullet$};
\node at (1.21487446953, -2.52271283013) {$_4$};
\node at (-1.73553495647, -3.60387547161) {$_\bullet$};
\node at (-1.21487446953, -2.52271283013) {$_5$};
\node at (-3.89971164873, -0.890083735825) {$_\bullet$};
\node at (-2.72979815411, -0.623058615078) {$_6$};
\node at (-3.12732592987, 2.49395920743) {$_\bullet$};
\node at (-2.18912815091, 1.7457714452) {$_7$};
\draw [-,>=latex,domain=0:100,samples=100] plot ({(4.0 + 4.0 * sin(180 * (0.5 + asin(-0.9 + 1.8 * (\x / 100)) / asin(0.9) / 2))) * cos(90 - (0.0 + \x * 1.5428571428571427))}, {(4.0 + 4.0 * sin(180 * (0.5 + asin(-0.9 + 1.8 * (\x / 100)) / asin(0.9) / 2))) * sin(90 - (0.0 + \x * 1.5428571428571427))});
\draw [-,>=latex,domain=0:100,samples=100] plot ({(4.0 + 2.0 * sin(180 * (0.5 + asin(-0.9 + 1.8 * (\x / 100)) / asin(0.9) / 2))) * cos(90 - (205.71428571428572 + \x * 1.0285714285714282))}, {(4.0 + 2.0 * sin(180 * (0.5 + asin(-0.9 + 1.8 * (\x / 100)) / asin(0.9) / 2))) * sin(90 - (205.71428571428572 + \x * 1.0285714285714282))});
}
\end{tikzpicture}
\quand
\tau^7_{1,14}\(
\begin{tikzpicture}[baseline=0,scale=0.2,label/.style={postaction={ decorate,transform shape,decoration={ markings, mark=at position .5 with \node #1;}}}]
{
\draw[fill,lightgray] (0,0) circle (4.0);
\node at (2.44929359829e-16, 4.0) {$_\bullet$};
\node at (1.71450551881e-16, 2.8) {$_1$};
\node at (3.12732592987, 2.49395920743) {$_\bullet$};
\node at (2.18912815091, 1.7457714452) {$_2$};
\node at (3.89971164873, -0.890083735825) {$_\bullet$};
\node at (2.72979815411, -0.623058615078) {$_3$};
\node at (1.73553495647, -3.60387547161) {$_\bullet$};
\node at (1.21487446953, -2.52271283013) {$_4$};
\node at (-1.73553495647, -3.60387547161) {$_\bullet$};
\node at (-1.21487446953, -2.52271283013) {$_5$};
\node at (-3.89971164873, -0.890083735825) {$_\bullet$};
\node at (-2.72979815411, -0.623058615078) {$_6$};
\node at (-3.12732592987, 2.49395920743) {$_\bullet$};
\node at (-2.18912815091, 1.7457714452) {$_7$};
\draw [-,>=latex,domain=0:100,samples=100] plot ({(4.0 + 2.0 * sin(180 * (0.5 + asin(-0.9 + 1.8 * (\x / 100)) / asin(0.9) / 2))) * cos(90 - (0.0 + \x * 1.0285714285714287))}, {(4.0 + 2.0 * sin(180 * (0.5 + asin(-0.9 + 1.8 * (\x / 100)) / asin(0.9) / 2))) * sin(90 - (0.0 + \x * 1.0285714285714287))});
\draw [-,>=latex,domain=0:100,samples=100] plot ({(4.0 + 2.0 * sin(180 * (0.5 + asin(-0.9 + 1.8 * (\x / 100)) / asin(0.9) / 2))) * cos(90 - (205.71428571428572 + \x * 1.0285714285714282))}, {(4.0 + 2.0 * sin(180 * (0.5 + asin(-0.9 + 1.8 * (\x / 100)) / asin(0.9) / 2))) * sin(90 - (205.71428571428572 + \x * 1.0285714285714282))});
}
\end{tikzpicture}
\) =
\begin{tikzpicture}[baseline=0,scale=0.2,label/.style={postaction={ decorate,transform shape,decoration={ markings, mark=at position .5 with \node #1;}}}]
{
\draw[fill,lightgray] (0,0) circle (4.0);
\node at (2.44929359829e-16, 4.0) {$_\bullet$};
\node at (1.71450551881e-16, 2.8) {$_1$};
\node at (3.12732592987, 2.49395920743) {$_\bullet$};
\node at (2.18912815091, 1.7457714452) {$_2$};
\node at (3.89971164873, -0.890083735825) {$_\bullet$};
\node at (2.72979815411, -0.623058615078) {$_3$};
\node at (1.73553495647, -3.60387547161) {$_\bullet$};
\node at (1.21487446953, -2.52271283013) {$_4$};
\node at (-1.73553495647, -3.60387547161) {$_\bullet$};
\node at (-1.21487446953, -2.52271283013) {$_5$};
\node at (-3.89971164873, -0.890083735825) {$_\bullet$};
\node at (-2.72979815411, -0.623058615078) {$_6$};
\node at (-3.12732592987, 2.49395920743) {$_\bullet$};
\node at (-2.18912815091, 1.7457714452) {$_7$};
\draw [-,>=latex,domain=0:100,samples=100] plot ({(4.0 + 2.0 * sin(180 * (0.5 + asin(-0.9 + 1.8 * (\x / 100)) / asin(0.9) / 2))) * cos(90 - (0.0 + \x * 6.685714285714286))}, {(4.0 + 2.0 * sin(180 * (0.5 + asin(-0.9 + 1.8 * (\x / 100)) / asin(0.9) / 2))) * sin(90 - (0.0 + \x * 6.685714285714286))});
}
\end{tikzpicture}.
\]
\end{example}

\begin{lemma}\label{tau-1-lem}
Let $y \in \tI_n$ and $i,j \in \ZZ$ with $i <j \not \equiv i \modu n)$. Then $y \leq \tau^n_{ij}(y)$.
\end{lemma}

\begin{proof}
Maintain the notation of Definition~\ref{tau-def}.
In the cases when $\tau^n_{ij}(y) \in \{y, ty, tyt\}$ for some
$t \in \{ (\circ,\circ), (\circ,\bullet), (\bullet,\circ)\}$,
the relation $y \leq \tau^n_{ij}(y)$ follows from Lemmas~\ref{bruhat0-lem} and
\ref{leftright-bruhat-lem}.
If $\cD_{ij}(y) $ is $\diagramCdaB$ and $i \equiv y(j) \modu n)$, then 
one can check using Lemma~\ref{bruhat0-lem}
 that for any integer $j'$ with $i<j' < y(j)$ and $j' \equiv j \modu n)$,
we have $y < y \cdot t_{j'y(j)}  <  y\cdot t_{j'y(j)}  \cdot t_{ij'}=\tau^n_{ij}(y)$.
If $\cD_{ij}(y)$ is $\diagramBadC$ and $i \equiv y(j) \modu n)$,
then the same statement holds with $j' = y(i)$.
In the remaining cases,
the set $\{i,j,y(i),y(j)\}$ has four elements  $a<b<c<d $ which represent distinct congruence classes modulo $n$, and
we have
$y < y \cdot t_{bc} = t_{ad} \cdot y < t_{ad} \cdot y \cdot t_{ab} < t_{ad} \cdot  y \cdot t_{ab} \cdot t_{cd} = \tau^n_{ij}(y)$.
\end{proof}

The following result is useful for determining when $\ellhat( \tau^n_{ij}(y)) = \ellhat(y) + 1$.

\begin{proposition}
Suppose $y \in \tI_n$ and $z = \tau^n_{ij}(y)\neq y$ for some integers $i<j \not \equiv i \modu n)$.
\ben
\item[(a)] Suppose $y(i) \leq i$ or $j\leq y(j)$. Let $\Delta = j -i$.
\ben
\item[i.] If $i \not \equiv y(j) \modu n)$ then $\ellhat(z) =\ellhat(y)+1$ if and only if $\ell(yt_{ij}) = \ell(y)+1$.

\item[ii.] If $i \equiv y(j) \modu n)$ then $\ellhat(z) =\ellhat(y)+1$ if and only if no $e \in \ZZ$ satisfies 
\[i < e <j\text{ and }y(i)- \Delta < y(e) < y(j) + \Delta.\]
\een

\item[(b)] Suppose $i<y(i) <y(j) <j\equiv y(i) \modu n)$.
Then $\ellhat(z) =\ellhat(y)+1$ if and only if 
$y(j) = i+n$
and
no $e \in \ZZ$ satisfies either pair of conditions
\[
\text{$ j-n < e < i +n$ and $ i < y(e) < j$}
 \quad\text{or}\quad
\text{$i < e < j-n$ and $ i-n < y(e) < j$}.
\]

\item[(c)] Suppose $i<y(j) <y(i) <j\equiv y(i) \modu n)$.
Then $\ellhat(z) =\ellhat(y)+1$ if and only if $y(j) = i+n$ and
no $e \in \ZZ$ satisfies either pair of conditions
\[
\text{$j-2n < e < i+n$ and $ i-n < y(e) < j$}
\quord
\text{$i < e < j-2n$ and $i -2n< y(e) < j$}.
\]
\een
\end{proposition}

\begin{remark}
If $ z =\tau^n_{kl}(y)\neq y$ for any integers $k,l$, then it is always possible to find
some other  integers $i,j$
such that $z =\tau^n_{ij}(y) = \tau^n_{kl}(y)$ and the hypotheses of (a), (b), or (c) hold.
\end{remark}

\begin{proof}
Maintain the notation of Definition~\ref{tau-def}.
First assume $y(i) \leq i$ or $j\leq y(j)$.
If $i$ and $j$ are both fixed points of $y$, then $z = yt_{ij}$ and $\ellhat(z) -\ellhat(y) = \frac{1}{2}(1+\ell(yt_{ij})-\ell(y))$.
If instead
\[\cD_{ij}(y) \in \left\{  \diagramACb,\ \diagrambAC,\ \diagrambADc,\ \diagramCDab,\ \diagramcdAB\right\},\]
then $z = t_{ij} y t_{ij}$ and $\ellhat(z) - \ellhat(y) = \frac{1}{2}(\ell(z)-\ell(y))$,
and it is straightforward,
using Lemma~\ref{leftright-bruhat-lem}, to check
that $\ell(z) - \ell(y)=2$ if and only if the conditions in part (a) hold.
If $i = y(i) < y(j) < j$
then \[z = t_{ij'} y t_{ij'} = \tau^n_{ij}(y) = \tau^n_{ij'}(y)\] for $j' = y(j)$,
and 
$\ell(yt_{ij})  =\ell(yt_{ij'})$. Our claim that $\ellhat(z) =\ellhat(y)+1$ if and only if $\ell(yt_{ij}) = \ell(y)+1$
therefore follows from the previous case with $j$ replaced by $j'$.
If $i < y(i) < y(j) = j$ then the same conclusion follows likewise.
Finally, if $i <i' < j< j'$ for $i' = y(i)$ and $j' = y(j)$, 
then \[y < y \cdot t_{i'j} < y \cdot t_{i'j} \cdot t_{ii'} < y \cdot t_{i'j} \cdot t_{ii'} \cdot t_{jj'} = z\]
and $\ellhat(z) -\ellhat(y) = \frac{1}{2}(\ell(z) -\ell(y)-1)$,
and in this case it is again an exercise to check that 
$\ell(z) = \ell(y)+3$ if and only if $\ell(yt_{ij}) = \ell(y)+1$.
In the symmetric case when $i' <i < j'< j$, we reach the same conclusion by a similar argument. 
This sketch suffices to prove part (a).

The proofs of (b) and (c) are similar. Assume we are in either  case.
 Define $i'=y(i)$ and $j' = y(j)$. In case (b) let $i''=i'$ and in case (c) let
 $i''$ be the smallest
 integer greater than $i$ with $i''\equiv j \modu n)$.
 As noted in the proof of Lemma~\ref{tau-1-lem},
 we then 
have $y < y \cdot t_{i''j'}  <  y\cdot t_{i''j'}  \cdot t_{ii''}=z$
and $\ellhat(z) - \ellhat(z) = \frac{1}{2}(\ell(z) - \ell(y))$.
We can only have $y \lessdot y \cdot t_{i''j'}$ if $j' = i+n$
and no integer $e$ satisfies $i'' < e < i + n$ and $y(i'') < y(e) < j$.
This translates to the first set of conditions in (b) and (c).
One checks that the second pair of conditions in each case is equivalent to
$y \cdot t_{i''j'} \lessdot y\cdot t_{i''j'}  \cdot t_{ii''}$.
\end{proof}
 
 Let $w \in \tS_n$ be an affine permutation.
 Refining the terminology of  Theorem~\ref{excluded-thm} slightly,
 we say that ``$w$ has the form $\dash a_1a_2\cdots a_k \dash b_1b_2\cdots  b_l\dash c_1c_2\cdots c_m \dash \dots \dash$''
if the string representation of $w$
 contains $a_1a_2\cdots a_k$ and $b_1b_2\cdots b_l$ and $c_1c_2\cdots c_m$ and so forth as consecutive subsequences,
 with every $a_i$ appearing to left of every $b_j$, every $b_i$ appearing to the left of every $c_j$, and so on.
The following statement is a constructive version of Theorem~\ref{bruh-cov-thm} from the introduction:

\begin{theorem}\label{last-thm}
Suppose $y,z \in \tI_n$ and $i,j \in \ZZ$
are such that $i<j\not\equiv i \modu n)$. If $w \in \cA(y)$
 and $w \lessdot wt_{ij} \in \cA(z)$ then $z = \tau^n_{ij}(y)$.
\end{theorem}

\begin{proof}
Let $t = t_{ij}$ and suppose $w \in \tS_n$ is such that 
$w^{-1} \in \cA(y)$,  $\ell(tw)=\ell(w)+1$, and $(tw)^{-1} \in \cA(z)$, so that $\ellhat(z) =\ellhat(y)+1$.
It suffices to show  that $z = \tau^n_{ij}(y)$.
Our strategy will be to compare the cycle removal processes for $w$ and $tw$, which differ in predictable ways,
and then invoke Theorem~\ref{cycle-process-thm}.
Note, from Lemma~\ref{bruhat0-lem}, that $\ell(tw)=\ell(w)+1$ if and only if $i$ appears to the left of $j$ in the string representation of $w$
and no integer $e \in \ZZ$ with $i<e<j$ appears between $i$ and $j$ in this string.

Define  $t\gamma = (t(a),t(b))$ for 
$\gamma = (a,b)\in \ZZ\times \ZZ$,
so that   if $\gamma \in \cC(y)$ then $t\gamma \in \cC(tyt)$ if and only if $t(a) < t(b)$.
Throughout, we let $p = \ell'(y)$ and suppose the cycle removal process for $w$ outputs the sequence 
$\gamma_1,\gamma_2,\dots,\gamma_p \in \cC(y)$. 
First assume $i,j \in \{a,b,a',b'\}$ where $(a,b),(a',b') \in \cC(y)$. 
The following useful observation then holds:
\begin{lemma}\label{*-lem}
If $t\gamma_1, t\gamma_2,\dots,t\gamma_p$ all belong to $\cC(tyt)$ then 
 the cycle removal process for $tw$ can  be carried out 
 to have these cycles as output, and it follows from Theorem~\ref{cycle-process-thm} that
  $z=tyt$.
\end{lemma}

Continuing the proof of the theorem,
suppose
 $w$ has the form
\[\dash b\dash a \dash b' \dash a' \dash \qquad\text{where $a \not \equiv a' \modu n)$ and $b\not \equiv b'\modu n)$}.\]
One of the following cases must then occur:
 \begin{itemize}
 
\item Suppose $i=a < a' = j$, so that  $tw$ has the form
$\dash b\dash a' \dash b' \dash a \dash$ and $a<b'$.
If $a'<b$ then it follows from Lemma~\ref{*-lem} that 
 $z=tyt$, so 
 $a<a'<b<b'$ as otherwise Lemma~\ref{bruhat0-lem} implies that $z=tyt<y$, contradicting Lemma~\ref{cover-lem}.
In this case we therefore have  $ \tau^n_{ij}(y)=tyt=z$.

Assume instead that  $b<a'$ so that $a<b<a'<b'$.
We claim that one may assume  that $\gamma_p = (a,b)$,
i.e., that $ba$ is
the last subsequence removed in the cycle removal process for $w$.
The only way this can fail is if some intermediate string $\ttS_k$ in the cycle removal process has the form
$\dash YbaX\dash$ for a cycle $(X,Y) \in \cC(y)$, which by Corollary~\ref{3412-cor} must have $a<X<Y<b<a'$.
But if this occurs then 
 $w$ must have 
the form 
\[\dash Y \dash b \dash a \dash X \dash b' \dash a'\dash
\qquord \dash Y \dash b\dash a\dash b' \dash a'\dash X\dash.
\]
which contradicts either $\ell(tw) = \ell(w) +1$ or Theorem~\ref{excluded-thm}.
Our claim therefore holds.
By similar reasoning, we deduce that no fixed point of $y$ less than $a'$ 
(respectively, greater than $b$) can appear to the right of $a$ 
(respectively, to the left of $b$) in the string representation of $w$.
It follows that the 
cycle removal process for $tw$ can be carried out to have 
output $t\gamma_1,t\gamma_2,\dots,t\gamma_{p-1}$,
which by
Theorem~\ref{cycle-process-thm} implies that $z = t_{ab'} \cdot y \cdot t_{ab} \cdot t_{a'b'}= \tau^n_{ij}(y)$.

\item
If $i=b<b'<j$, then the situation is symmetric to the previous case, 
and we deduce that $z=\tau^n_{ij}(y)$ by similar arguments.

\item If $i=b<a'=j$, so that $a<b<a'<b'$, then   Lemma~\ref{*-lem} implies that 
 $z = tyt = \tau^n_{ij}(y)$.

\item
Finally suppose $i=a<b'=j$,
so that    $tw$ has the form
$\dash b\dash b' \dash a \dash a' \dash$.
By Theorem~\ref{excluded-thm}, we cannot have both $a'<a$ and $b'<b$,
and the cases when $a<a'<b'<b$ or $a'<a<b<b'$  each lead to contradiction.
For example, suppose $a<a'<b'<b$. 
By considering what happens if we try to remove the 
same consecutive subsequences during the cycle removal process for $tw$ as we remove for $w$,
we deduce from Theorem~\ref{cycle-process-thm} that $(b',b) \in \cC(z)$ and that either $a$ is a fixed point of $z$
or $(a,c) \in \cC(z)$ for an integer $c<b'$. 
Both cases contradict Theorem~\ref{excluded-thm} since $a$ appears to the right of $b'$ in $tw$.
If $a'<a<b<b'$ then we obtain a similar contradiction by symmetric arguments.

Instead suppose both $a<a'$ and $b<b'$.
We claim that one may assume that some intermediate string $\ttS_k$ in the cycle removal process for $w$
contains $bab'a'$ as a consecutive subsequence.
To check this, first note that no fixed points of $y$
can appear in the string representation of $w$ between $a$ and $b'$
by
Theorem~\ref{excluded-thm} since  $\ell(tw)= \ell(w)+1$.
By Corollary~\ref{3412-cor}, 
the only way our claim can fail is if 
 $w$ has the form
\[ 
\dash Y \dash b \dash a \dash X \dash b' \dash a' \dash 
\qquord
 \dash b \dash a \dash Y'\dash b' \dash a' \dash X' \dash
\]
for some $(X,Y) \in \cC(y)$ with $a<X<Y<b$ or some $(X',Y') \in \cC(y)$ with $a'<X'<Y'<b'$.
Both cases would contradict the fact that $\ell(tw) = \ell(w)+1$, so our claim holds.

If $a'<b$, then it follows from our claim that 
we may construct the cycle removal processes for $w$ and $tw$ to be identical, except 
that at steps $k$ and $k+1$ the process for $w$ removes the subsequences $ba$ and then $b'a'$,
while the process for $tw$ removes $b'a$ and then $ba'$. 
In this case we deduce by Theorem~\ref{cycle-process-thm} that 
$z= t_{aa'} \cdot y \cdot t_{aa'} = t_{i,y(j)} \cdot y\cdot t_{i,y(j)} = \tau^n_{ij}(y)$.

Instead suppose $b<a'$. We then make the further claim that the index $k$ can chosen to be $p-1$, i.e., 
so that the subsequences $ba$ and $b'a'$
are the last ones removed in the cycle removal process for $w$.
This can only fail if an intermediate string in the cycle removal process for $w$ 
has the form $\dash Ybab'a'X\dash$ for some $(X,Y) \in \cC(y)$, but this never occurs 
since Corollary~\ref{3412-cor} would imply that $Y<b<a'<X$, contradicting $X<Y$.
We can therefore assume $(a_{p-1},b_{p-1})=(a,b)$ and $(a_p,b_p) = (a',b')$.
Under this hypothesis, it follows using Corollary~\ref{321cycle-cor}
that the cycle removal process for $tw$ can be carried out to have 
output $\gamma_1,\gamma_2,\dots,\gamma_{p-2},(a,b')$,
whence by Theorem~\ref{cycle-process-thm}
we  have $z = t_{ab'} \cdot y \cdot t_{ab} \cdot t_{a'b'} = \tau^n_{ij}(z)$.
 \end{itemize}
 This concludes the most complicated portion of our analysis.
For the next case, continue to let $(a,b),(a',b') \in \cC(y)$ 
and assume $i,j \in \{a,b,a',b'\}$, but now suppose $w$ has the form 
\[\dash b\dash a \dash b' \dash a' \dash \qquad\text{where $a \equiv a' \modu n)$ and $b\equiv b'\modu n)$.}\]
One of the following must then occur:
\begin{itemize}

\item Suppose $i = a<b'=j$. 
Since $\ell(tw)=\ell(w)+1$, Theorem~\ref{excluded-thm} 
implies that no fixed points of $y$ appear between $a$ and $b'$ in the string representation of $w$.
We may assume that $\gamma_p= (a,b)$, since if some string $\ttS_k$
in the cycle removal process for $w$ had the form $\dash YbaX \dash Y'b'a'X'\dash $ for $(X,Y),(X',Y')  \in \cC(y)$,
then Corollary~\ref{3412-cor} would imply that $a<X<Y<b < b'$, contradicting $\ell(tw) = \ell(w)+1$.
Hence the cycle removal process for $tw$ can be carried out to have 
output $\gamma_1,\gamma_2,\dots,\gamma_{p-1},(a,b')$,
so by
Theorem~\ref{cycle-process-thm} we have $z = t_{ab'} \cdot y \cdot t_{ab} = \tau^n_{ij}(y)$.

\item If $i =b< a' = j$, then it follows from Lemma~\ref{*-lem} that $z = t_{ba'} \cdot y \cdot t_{ba'} = \tau^n_{ij}(y)$.
\end{itemize}
Still with $(a,b),(a',b') \in \cC(y)$ and $i,j \in \{a,b,a',b'\}$, finally suppose $w$ has the form 
\[\dash b\dash b' \dash a' \dash a \dash\]
so that $a \not\equiv a' \modu n)$ and $b\not\equiv b'\modu n)$.
Since $w$ has the form $\dash i \dash j \dash $,
  Theorem~\ref{excluded-thm} implies that  $a'<a<b<b'$ and either $i=a'<a=j$ or $i=b<b'=j$.
  It follows by Lemma~\ref{*-lem} that $z=tyt$ so $(a,b'),(a',b) \in \cC(z)$.
This contradicts Theorem~\ref{excluded-thm}, however, since 
$tw$ has the form $
 \dash b \dash b' \dash a \dash a' \dash
$
or
$
 \dash b' \dash b \dash a' \dash a \dash
 $.
 Hence this last case cannot  occur. 

By Theorem~\ref{excluded-thm},
the preceding discussion exhausts 
the possibilities for $i$ and $j$ when these numbers belong to distinct 2-cycles of $y$.
We next consider the case when $i,j \in \{a,b,e\}$ where $e \in \ZZ$ is fixed point of $y$ and $(a,b) \in \cC(y)$.
Suppose this happens and 
 $w$ has the form
\[ \dash e \dash b\dash a \dash .\]
Since $i$ must appear to the left of $j$, one of the following must occur:
\begin{itemize}
\item Suppose $i=e<a=j<b$ so that $tw$ has the form $\dash a \dash b \dash e \dash$.
Since $\ell(tw) = \ell(w)+1$,
it follows by Theorem~\ref{excluded-thm} that no fixed points of $y$ 
less than $a$ can appear between to the right of $e$ in the string representation of $w$, 
and that likewise no fixed points greater than $a$ can appear to the left of $e$.
Hence, by Theorem~\ref{cycle-process-thm} the cycle removal process for $tw$ can be carried out to have 
output $t\gamma_1,t\gamma_2,\dots,t\gamma_p$ 
 so $z=tyt = \tau^n_{ij}(y)$.

\item Suppose $a<i=e<b=j$ so that $tw$ has the form $\dash b \dash e \dash a \dash$.
We argue that this leads to contradiction.
By considering what happens if we try to remove the same 
consecutive subsequences during the cycle removal process for $tw$ as we remove for $w$,
we deduce from Theorem~\ref{cycle-process-thm} that
 $(a,e) \in \cC(z)$
 and
 that either $b$ is a fixed point of $z$ or $(e',b) \in \cC(z)$ for an integer $e'>e$.
 Both cases contradict
 Theorem~\ref{excluded-thm} since $b$ appears to the left of $e$ in $tw$. 

\item Finally suppose $i=e<a<b=j$ so that $tw$ again has the form $\dash b \dash e \dash a \dash$.
Since $\ell(tw) = \ell(w)+1$, it follows by 
Theorems~\ref{cycle-process-thm} and \ref{excluded-thm} that no fixed points of $y$ 
 appear between $e$ and $a$ in the string representation of $w$.
 Moreover, by Theorem~\ref{excluded-thm}, all fixed points of $y$ 
 appearing in $w$ to the right of $a$ are bounded below by $a$, and 
 all fixed points of $y$ appearing to the left of $e$ are bounded above by $e$.
We can assume $\gamma_p=(a,b)$, since otherwise 
some intermediate string $\ttS_k$ in the cycle removal process of $w$ must have the form $\dash e \dash Y ba X \dash$
for a cycle $(X,Y) \in \cC(y)$, in which case Corollary~\ref{3412-cor} would imply that $e<a<X<Y<b$,
contradicting $\ell(tw) = \ell(w)+1$. 
Hence the penultimate string $\ttS_{p-1}$ contains $eba$ as a consecutive subsequence,
so the cycle removal process for $tw$ can be carried out to have 
output $\gamma_1,\gamma_2,\dots,\gamma_{p-1},(e,b)$.
By Theorem~\ref{cycle-process-thm}, we therefore have $z = t_{ea} \cdot y \cdot t_{ea} = \tau^n_{ij}(y)$.

\end{itemize}
If $w$ has the form $\dash b\dash a \dash e \dash $ then $z = \tau^n_{ij}(y)$ follows by a
symmetric argument.
The only remaining possibility is that $y(i)=i<j=y(j)$. 
Since   no fixed points can occur between $i$ and $j$ in the string representation of $w$,  
 the cycle removal process for $tw$ can be carried out to have 
output $\gamma_1,\gamma_2,\dots,\gamma_p,(j,i)$, and hence  $z = ty = \tau^n_{ij}(y)$.
We conclude that $z =\tau^n_{ij}(y)$ in all cases.
\end{proof}

Write $y \lessdot_I z$ for $y,z \in \tI_n$
if $z$ covers $y$ in    the Bruhat order of $\tS_n$ restricted to $\tI_n$, i.e.,
if $\{ w \in \tI_n : y \leq w < z\} = \{y\}$.
 If $y,z \in \tI_n$ then $y\lessdot z$ implies that $y \lessdot_I z$,
but not vice versa.

\begin{corollary}
If $y, z \in \tI_n$ then $y\lessdot_I z$
if and only if $\ellhat(z) = \ellhat(y)+1$ and
$z = \tau^n_{ij}(y)$ for some $i,j$.
\end{corollary}

\begin{proof}
If $\ellhat(z) = \ellhat(y)+1$ and
$z = \tau^n_{ij}(y)$ then Lemma~\ref{tau-1-lem} implies that $y\lessdot_I z$.
Suppose conversely that $y \lessdot_I z$. 
By Proposition~\ref{graded-prop}, we then have $y<z$ and $\ellhat(z) = \ellhat(y)+1$,
so by Lemma~\ref{cover-lem} there exists $v \in \cA(y)$ and $w \in \cA(z)$ with $v \lessdot w$,
so by Theorem~\ref{last-thm} we have $z = \tau^n_{ij}(y)$ for some $i,j$.
\end{proof}

\newpage
\appendix
\section{Index of symbols}
\label{not-sect}

The table below provides definitions and references for all frequently occurring symbols.

\def\skip{\\[-4pt]}

\begin{center}
\begin{tabular}{l | l r}
\text{Symbol} & \text{Meaning} & 
 \\
 \hline
  $(W,S)$ &An arbitrary Coxeter system &  \\
 $\circ$ & The Demazure product $W \times W \to W$ & \S\ref{intro-sect} \\
 $\DesR(w)$, $\DesL(w)$ & The left and right descent sets of $w \in W$ & \S\ref{prelim-sect} \\
 $\cA(z)$ & The set of  $w \in \tS_n$ of minimal length with $z = w^{-1}\circ w$ & Def.~\ref{ca-def}  \\
 $\ellhat(z)$ & The common length of each $w \in \cA(z)$ & Prop.~\ref{ellhat-prop} \\
   \\
  $\tS_n$ & The affine symmetric group of rank $n$ & \S\ref{intro-sect} \\
   $S_n$ & The group of permutations of $[n]=\{1,2,\dots,n\}$ & \S\ref{intro-sect} \\
  $I_n$, $\tI_n$ & The sets of involutions in $S_n$ and $\tS_n$ & \S\ref{intro-sect}  \\
    $\cW_n$ & The set of weighted involutions in $\tS_n$ & Def.~\ref{weighted-def} \\
    $\ell'(z)$ & The number of edges in the winding diagram of $z \in \tI_n$ & Def.~\ref{abs-len-def} \\
  $\cWmin_n$ & The subset of $\theta=(w,\phi)\in \cW_n$ with $\ell'(w)=\ell(w)$ & Def.~\ref{cWmin-def} \\
  \\
  $s_i$ & A simple reflection in $\tS_n$ & \S\ref{intro-sect} \\
  $t_{ij}$ & A reflection in $\tS_n$ & \S\ref{prelim-sect} \\
  $w\mapsto w^*$ & The automorphism of $\tS_n$ with $s_i \mapsto s_{n-i}$ & Lem.~\ref{*lem} \\
  $\lW a_1,a_2,\dots,a_n\rW$ & The window notation for an affine permutation & Prop.-Def.~\ref{window-def} \\
  $\cC(z)$ & The set of pairs $\{ (a,b) \in \ZZ\times \ZZ  : a<b=z(a)\}$ & \S\ref{prelim-sect}  \\
  $\pi_i$  & A certain left or right operator on $\cW_n$ & Def.~\ref{pi-def} \\
  \\
  $g_L(\theta)$, $g_R(\theta)$ & Certain affine permutations associated to $\theta \in \cW_n$ & Thm.~\ref{omega-thm} \\
  $\omega_L(\theta)$, $\omega_R(\theta)$ & Certain affine involutions associated to $\theta \in \cW_n$ & Thm.~\ref{omega-thm} \\
  $\lambda_L$, $\lambda_R$ & Maps $\tI_n \to \cW_n$ inverse to $\omega_L$ and $\omega_R$ & Prop.-Def.~\ref{lambda-def} \\
  $\alpha_R$, $\alpha_L$ & The maps $\tI_n \to \tS_n$ given by  $g_R\circ \lambda_R$ and  $ g_L\circ \lambda_L$ & Def.~\ref{alpha-def} \\
  $\alpha_{\min}(z)$, $\alpha_{\max}(z)$ & The affine permutations $\alpha_R(z)z$ and $\alpha_L(z)z$ & Def.~\ref{alpha-min-def} \\
  \\
    $[w]_E$ & The standardisation of an (affine) permutation & \S\ref{cycle-sect} \\
  $\cG_{ij}(z)$,   $\cD_{ij}(z)$ & Certain graphs associated to $z \in \tI_n$ & Def.~\ref{cG-def} \\
  $\tau^n_{ij}$ & Certain maps $\tI_n \to \tI_n$ & Def.~\ref{tau-def} \\
  \\
   $\leq$ & The Bruhat order on a Coxeter group (usually $\tS_n$)   \\
  $\prec$ & A partial order on $\cWmin_n$ & \S\ref{order-sect} \\
  $\prec_R$, $\prec_L$ & Certain graded suborders of Bruhat order on $\tI_n$  & Prop.~\ref{prec-prop} \\
   $\lessdot_\cA$ & A certain covering relation on $\tS_n$ & Lem.~\ref{231-lem} \\
  $<_\cA$ & The transitive closure of $\lessdot_\cA$ &  \\
  $\sim_\cA$ & The symmetric closure of $\leq_\cA$ &  \\
  $\lessdot_I$ & The covering relation of Bruhat order restricted to $\tI_n$  &
\end{tabular}
\end{center}

\end{document}